\documentclass[
reqno,
10pt,
oneside
]{article}

\usepackage[ngerman, english]{babel}
\usepackage[utf8]{inputenc}
\usepackage[normalem]{ulem}
\usepackage[babel,french=guillemets,german=swiss]{csquotes}
\usepackage[center,font=small]{caption}
\usepackage{cite}
\usepackage{bbm}
\usepackage[raggedright]{titlesec}
\usepackage{xcolor}
\usepackage{enumerate}
\usepackage{selinput}

\titleformat{\section}{\normalfont\large\bfseries}{\thesection}{1em}{}
\titleformat{\subsection}{\normalfont\bfseries}{\thesubsection}{1em}{}

\usepackage{tocbasic}
\DeclareTOCStyleEntry[
beforeskip=.2em plus 1pt,
]{tocline}{section}

\usepackage[%
left=1.5in,
right=1.5in,
bottom=1in,
top=0.8in,
footskip=0.4in,
]{geometry}

\usepackage{color}
\definecolor{LinkColor}{rgb}{0,0,1}
\definecolor{LinkColor2}{rgb}{0,0.5,0}
\definecolor{lbcolor}{rgb}{0.85,0.85,0.85}
\definecolor{FrameColor}{rgb}{0.85,0.85,0.85}
\definecolor{rosso}{rgb}{0.8,0,0}
\definecolor{lightgray}{rgb}{0.5,0.5,0.5}
\definecolor{violet}{rgb}{0.65,0,0.65}
\definecolor{darkgreen}{rgb}{0,0.5,0}

\usepackage{enumitem}

\usepackage{graphicx}
\usepackage{placeins}
\usepackage{overpic}

\usepackage{amsmath}
\usepackage{amssymb}
\usepackage{dsfont}
\usepackage{empheq}
\allowdisplaybreaks
\usepackage{amsthm}

\usepackage[%
pdftitle={Titel},%
pdfauthor={Autor},%
pdfcreator={LaTeX, LaTeX with hyperref and KOMA-Script},
pdfsubject={Betreff}, 
pdfkeywords={Keywords}
]{hyperref} 

\hypersetup{%
	colorlinks	=true,
	linkcolor	=LinkColor,%
	anchorcolor	=LinkColor,%
	citecolor	=LinkColor2,%
	filecolor	=LinkColor,%
	menucolor	=LinkColor,%
	urlcolor	=LinkColor,%
}
\usepackage{marginnote}

\usepackage{verbatim}

\newtheorem{theorem}{Theorem}[section]

\newtheorem{proposition}[theorem]{Proposition}
\newtheorem{corollary}[theorem]{Corollary}
\newtheorem{definition}[theorem]{Definition}

\theoremstyle{definition}
\newtheorem{remark}[theorem]{Remark}

\makeatletter
\renewenvironment{proof}[1][\proofname]{%
	\par\pushQED{\qed}\normalfont%
	\topsep6\p@\@plus6\p@\relax
	\trivlist\item[\hskip\labelsep\bfseries#1\@addpunct{.}]%
	\ignorespaces
}{%
	\popQED\endtrivlist\@endpefalse
}
\makeatother

\makeatletter
\renewcommand\paragraph{\@startsection{paragraph}{4}{\z@}%
	{1ex \@plus1ex \@minus.2ex}%
	{-1em}%
	{\normalfont\normalsize\bfseries}}
\renewcommand\subparagraph{\@startsection{paragraph}{4}{\z@}%
	{1ex \@plus1ex \@minus.2ex}%
	{-1em}%
	{\normalfont\normalsize\itshape}}
\makeatother

\newcommand{\abs}[1]{\left| #1 \right|}
\newcommand{\bigabs}[1]{\big| #1 \big|}

\newcommand{\norm}[1]{\| #1 \|}
\newcommand{\bignorm}[1]{\big\| #1 \big\|}

\newcommand{\ang}[2]{ \langle #1 , #2  \rangle}
\newcommand{\bigang}[2]{ \big< #1 , #2  \big>}

\newcommand{\scp}[2]{ \left( #1 , #2  \right)}
\newcommand{\bigscp}[2]{\big( #1 , #2 \big)}

\newcommand{\meano}[1]{{\langle #1 \rangle}_{\Omega}}
\newcommand{\meang}[1]{{\langle #1 \rangle}_{\Gamma}}
\newcommand{\mean}[2]{\textnormal{mean}\scp{#1}{#2}}

\newcommand{\R}{\mathbb R}
\newcommand{\N}{\mathbb N}
\newcommand{\n}{\mathbf{n}}

\newcommand{\intO}{\int_\Omega}

\newcommand{\intG}{\int_\Gamma}

\newcommand{\ep}{\varepsilon}

\newcommand{\dtau}{\;\mathrm d\tau}

\newcommand{\dx}{\;\mathrm{d}x}

\newcommand{\ds}{\;\mathrm ds}

\newcommand{\dxt}{\;\mathrm{d}x\;\mathrm{d}t}
\newcommand{\dGt}{\;\mathrm{d}\Ga\;\mathrm{d}t}
\newcommand{\dxs}{\;\mathrm{d}x\;\mathrm{d}s}
\newcommand{\dGs}{\;\mathrm{d}\Ga\;\mathrm{d}s}

\newcommand{\dG}{\;\mathrm d\Ga}

\newcommand{\ddt}{\frac{\mathrm d}{\mathrm dt}}

\newcommand{\del}{\partial}
\newcommand{\delt}{\partial_{t}}

\newcommand{\deln}{\partial_\n}

\newcommand{\Grad}{\nabla}
\newcommand{\Lap}{\Delta}
\newcommand{\Div}{\textnormal{div}}
\newcommand{\Gradg}{\nabla_\Ga}
\newcommand{\Lapg}{\Delta_\Ga}
\newcommand{\Divg}{\textnormal{div}_\Ga}

\newcommand{\emb}{\hookrightarrow}

\newcommand{\suchthat}{\;\ifnum\currentgrouptype=16 \middle\fi|\;}

\newcommand{\Om}{\Omega}
\newcommand{\Ga}{\Gamma}

\begin{document}

%
%

\title{\bfseries\Large 
    Strong well-posedness and separation properties\\
    for a bulk-surface convective Cahn--Hilliard system\\
    with singular potentials
\\[-1.5ex]$\;$}	

\author{
	Patrik Knopf \footnotemark[1]
    \and Jonas Stange \footnotemark[1]}

\date{ }

\maketitle

\renewcommand{\thefootnote}{\fnsymbol{footnote}}

\footnotetext[1]{
    Faculty for Mathematics, 
    University of Regensburg, 
    93053 Regensburg, 
    Germany \newline
	\tt(%
        \href{mailto:patrik.knopf@ur.de}{patrik.knopf@ur.de},
        \href{mailto:jonas.stange@ur.de}{jonas.stange@ur.de}%
        ).
}

\begin{center}
	\scriptsize
	{
		\textit{This is a preprint version of the paper. Please cite as:} \\  
		P.~Knopf, J.~Stange, 
        \textit{Journal of Differential Equations} \textbf{439}, no.~113408, 2025, \\
		\url{https://doi.org/10.1016/j.jde.2025.113408}
	}
\end{center}

\medskip

%
%

\begin{small}
\begin{center}
    \textbf{Abstract}
\end{center}
This paper addresses the well-posedness of a general class of bulk-surface convective Cahn--Hilliard systems with singular potentials. For this model, we first prove the existence of a global-in-time weak solution by approximating the singular potentials via a Yosida regularization, applying the corresponding results for regular potentials, and eventually passing to the limit in this approximation scheme. Then, we prove the uniqueness of weak solutions and their continuous dependence on the velocity fields and the initial data. Afterwards, assuming additional regularity of the domain as well as the velocity fields, we establish higher regularity properties of weak solutions and eventually the existence of strong solutions. In the end, we discuss strict separation properties for logarithmic type potentials in both two and three dimensions.
\\[1ex]
\\[1ex]	
\end{small}

\begin{small}
\setcounter{tocdepth}{2}
\hypersetup{linkcolor=black}
\tableofcontents
\end{small}

\setlength\parindent{0ex}
\setlength\parskip{1ex}
\allowdisplaybreaks
\numberwithin{equation}{section}
\renewcommand{\thefootnote}{\arabic{footnote}}

\newpage

\section{Introduction} 
\label{SECT:INTRO}

In this paper, we investigate the following \textit{bulk-surface convective Cahn--Hilliard system}:
\begin{subequations}\label{CH}
    \begin{align}
        \label{CH:1}
        &\delt\phi + \Div(\phi\boldsymbol{v}) = \Div(m_\Om(\phi)\Grad\mu) && \text{in} \ Q, \\
        \label{CH:2}
        &\mu = -\epsilon \Lap\phi + \epsilon^{-1}F'(\phi)   && \text{in} \ Q, \\
        \label{CH:3}
        &\delt\psi + \Divg(\psi\boldsymbol{w}) = \Divg(m_\Ga(\psi)\Gradg\theta) - \beta m_\Om(\phi)\deln\mu && \text{on} \ \Sigma, \\
        \label{CH:4}
        &\theta = - \epsilon_{\Gamma}\kappa\Lapg\psi + \epsilon_{\Gamma}^{-1}G'(\psi) + \alpha\epsilon\deln\phi && \text{on} \ \Sigma, \\
        \label{CH:5}
        &\begin{cases} \epsilon K\deln\phi = \alpha\psi - \phi &\text{if} \ K\in [0,\infty), \\
        \deln\phi = 0 &\text{if} \ K = \infty
        \end{cases} && \text{on} \ \Sigma, \\
        \label{CH:6}
        &\begin{cases} 
        L m_\Om(\phi)\deln\mu = \beta\theta - \mu &\text{if} \  L\in[0,\infty), \\
        m_\Om(\phi)\deln\mu = 0 &\text{if} \ L=\infty
        \end{cases} &&\text{on} \ \Sigma, \\
        \label{CH:7}
        &\phi\vert_{t=0} = \phi_0 &&\text{in} \ \Om, \\
        \label{CH:8}
        &\psi\vert_{t=0} = \psi_0 &&\text{on} \ \Ga.
    \end{align}
\end{subequations}
This system of equations has already been investigated in \cite{Knopf2024}, and in the special case $K=0$, it appears as a subsystem in the Navier--Stokes--Cahn--Hilliard model with dynamic boundary conditions that was derived in \cite{Giorgini2023}.

In \eqref{CH}, $\Omega\subset\R^d$ (with $d=2,3$) denotes a bounded domain with boundary $\Ga\coloneqq\del\Om$, $T>0$ is a given final time, and for brevity we write $Q\coloneqq \Om\times(0,T)$ and $\Sigma\coloneqq\Ga\times(0,T)$. We denote by $\mathbf{n}$ the outward pointing unit normal vector on $\Ga$ and by $\deln$ the outward normal derivative on the boundary. Furthermore, the symbols $\Gradg$, $\Divg$, and $\Lapg$ stand for the surface gradient, the surface divergence, and the Laplace--Beltrami operator on $\Ga$, respectively. 

The functions $\phi:Q\rightarrow\R$ and $\mu:Q\rightarrow\R$ represent the phase-field and the chemical potential in the bulk, whereas $\psi:\Sigma\rightarrow\R$ and $\theta:\Sigma\rightarrow\R$ denote the phase-field and the chemical potential on the surface, respectively. Moreover, the functions $\boldsymbol{v}:Q\rightarrow\R^d$ and $\boldsymbol{w}:\Sigma\rightarrow\R^d$ are prescribed velocity fields corresponding to the flow of the two materials in the bulk and on the boundary, respectively. Furthermore, the parameter $\epsilon > 0$ is related to the thickness of the diffuse interface in the bulk, while $\epsilon_\Ga > 0$ corresponds to the width of the diffuse interface on the boundary. The constant $\kappa > 0$ acts as a weight for surface diffusion effects.

In \eqref{CH}, the time evolution of $\scp{\phi}{\mu}$ is given by the \textit{bulk convective Cahn--Hilliard subsystem} \eqref{CH:1}-\eqref{CH:2}, while the evolution of $\scp{\psi}{\theta}$ is determined by the \textit{surface convective Cahn--Hilliard subsystem} \eqref{CH:3}-\eqref{CH:4}, which is coupled to the bulk variables by means of the involved normal derivatives $\deln\phi$ and $\deln\mu$. Moreover, the bulk and surface quantities are coupled by the boundary conditions \eqref{CH:5} and \eqref{CH:6}, which depend on parameters $K, L \in[0,\infty]$ and $\alpha, \beta\in\R$.

In \eqref{CH:5} and \eqref{CH:6}, the parameters $K, L\in[0,\infty]$ are used to distinguish different cases, each corresponding to a certain solution behavior related to a physical situation. The case that has been studied most extensively in the literature is $K=0$, which corresponds to the Dirichlet condition $\phi = \alpha\psi$ on $\Sigma$. This case makes particular sense along with $\alpha = 1$, if the materials on the surface are the same as in the bulk. Nevertheless, we also want to cover the case where the materials on the boundary are not the same as in the bulk. For instance, this might be the case if the materials on the boundary are transformed by chemical reactions. This is taken into account in our model by the choice $K\in(0,\infty]$. If $K=\infty$, the boundary condition \eqref{CH:5} degenerates to a homogeneous Neumann boundary condition for the phase-field $\phi$. Although, for most applications, this might not be the preferred choice, as such a Neumann boundary condition enforces the diffuse interface to always intersect the boundary at a perfect angle of ninety degrees, we include this case in our analysis for the sake of completeness. In the case $K\in(0,\infty)$, \eqref{CH:5} can be regarded as a Robin-type boundary condition. It can be used to describe a physical scenario where $\phi$ and $\psi$ represent different materials, and therefore, $\psi$ and the trace of $\phi$ are (in general) not proportional. However, this setup still allows for a dynamic change of the contact angle between the diffuse interface and the boundary. 
In some sense, including the cases $K\in(0,\infty)$ is also helpful for the mathematical analysis. For instance, even for regular potentials $F$ and $G$, a Faedo--Galerkin scheme to construct weak solutions cannot be used in the case $K=0$ and $L\in (0,\infty]$. For more details, we refer to \cite{Colli2024, Knopf2024}.

The bulk and surface chemical potentials $\mu$ and $\theta$ are coupled through the boundary condition \eqref{CH:5} involving a parameter $L\in[0,\infty]$, which accounts for a possible transfer of material between bulk and surface, see \cite{Knopf2021a}. For the non-convective case, i.e., $\boldsymbol{v}\equiv \mathbf{0}$ and $\boldsymbol{w}\equiv \mathbf{0}$, the choice $L=0$ was first proposed in \cite{Gal2006} and \cite{Goldstein2011}. Then, $\mu$ and $\theta$ are coupled through the Dirichlet condition $\mu = \beta\theta$ on $\Sigma$, which means that $\mu$ and $\theta$ are assumed to always remain in chemical equilibrium. In this case, a rapid transfer of material between bulk and boundary can be expected (see, e.g., \cite{Knopf2021a}). The choice $L=\infty$ was introduced in \cite{Liu2019}. In this case, \eqref{CH:6} is a homogeneous Neumann boundary condition on $\mu$, which means that the mass flux between the bulk and the surface is zero, and consequently, no transfer of material between bulk and surface will occur. In such a situation, the chemical potentials $\mu$ and $\theta$ are not directly coupled. The choice $L\in(0,\infty)$ was proposed and analyzed in \cite{Knopf2021a}. Here, the chemical potentials $\mu$ and $\theta$ are coupled by a Robin-type boundary condition \eqref{CH:6}. In this case, a transfer of material between bulk and surface will occur, and the coefficient $L^{-1}$ is related to the kinetic rate associated with adsorption and desorption processes or chemical reactions at the boundary. Furthermore, it was shown in \cite{Knopf2021a, Lv2024} (in the non-convective case) and in \cite{Knopf2024} (in the convective case), that the limits $L\rightarrow 0$ and $L\rightarrow\infty$ can be used to recover the boundary condition \eqref{CH:6} with $L = 0$ and $L = \infty$, respectively.

The functions $F^\prime$ and $G^\prime$ are the derivatives of double-well potentials $F$ and $G$, respectively.
A physically motivated example of the potentials $F$ and $G$, especially in applications related to material science, is the \textit{logarithmic potential}, which is also referred to as the \textit{Flory--Huggins potential}. It is given by
\begin{align}\label{DEF:W:LOG}
    W_{\textup{log}}(s) \coloneqq \frac{\Theta}{2}\big[(1+s)\ln(1+s) + (1-s)\ln(1-s)\big] -\frac{\Theta_c}{2}s^2, \quad s\in[-1,1].
\end{align}

Here, $\Theta$ represents the temperature of the system and $\Theta_c$ stands for the critical temperature under which phase separation occurs. These constants are supposed to satisfy $0 < \Theta < \Theta_c$. The logarithmic potential is  often approximated by a \textit{regular double-well potential} which reads as
\begin{align}\label{DEF:W:REG}
    W_{\textup{reg}}(s) \coloneqq c(s^2-1)^2, \quad s\in\R
\end{align}
for a suitable constant $c > 0$. A further common choice is the \textit{double-obstacle potential}, which is given by
\begin{align}\label{DEF:W:DOB}
    W_{\textup{obst}}(s) \coloneqq \begin{cases}
        1 - s^2, &\text{if~} s\in[-1,1], \\
        +\infty, &\text{if~} s\in\R\setminus[-1,1].
    \end{cases}
\end{align}

In contrast to the regular potential $W_{\text{reg}}$, the potentials $W_{\text{log}}$ and $W_{\text{obst}}$ are classified as \textit{singular potentials} as these functions or their derivatives exhibit singularities at $\pm 1$. 
Physically reasonable singular potentials can be decomposed as the sum of a smooth part and a convex part, which might not be differentiable but it at least possesses a subdifferential that is a maximal monotone graph.

An important quantity associated with the system \eqref{EQ:SYSTEM} is the energy functional
\begin{align}
    \label{INTRO:ENERGY}
    \begin{split}
        E_K(\phi,\psi) &= \intO \frac{\epsilon}{2}\abs{\Grad\phi}^2 + \epsilon^{-1}F(\phi) \dx + \intG \frac{\epsilon_{\Gamma}\kappa}{2}\abs{\Gradg\psi}^2 + \epsilon_{\Gamma}^{-1}G(\psi) \dG \\
        &\quad + \sigma(K)\intG\frac{1}{2}\abs{\alpha\psi - \phi}^2\dG,
    \end{split}
\end{align}
where the function
\begin{align*}
    \sigma(K) \coloneqq
    \begin{cases}
        K^{-1}, &\text{if } K\in (0,\infty), \\
        0, &\text{if } K\in\{0,\infty\}
    \end{cases}
\end{align*}
is used to distinguish the different cases corresponding to the choice of $K$. Sufficiently regular solutions of the system \eqref{EQ:SYSTEM} satisfy the the \textit{mass conservation law}
\begin{align}
    \label{INTRO:MASS}
    \begin{dcases}
        \beta\intO \phi(t)\dx + \intG \psi(t)\dG = \beta\intO \phi_0 \dx + \intG \psi_0\dG, &\textnormal{if } L\in[0,\infty), \\
        \intO\phi(t)\dx = \intO\phi_0\dx \quad\textnormal{and}\quad \intG\psi(t)\dG = \intG\psi_0\dG, &\textnormal{if } L = \infty
    \end{dcases}
\end{align}
for all $t\in[0,T]$. This means that the mass is conserved separately in $\Om$ and $\Ga$ in the case $L=\infty$, whereas in the case $L\in[0,\infty)$, a transfer of material between bulk and surface is allowed. Furthermore, sufficiently regular solutions satisfy the \textit{energy identity}
\begin{align}
    \label{INTRO:ENERGY:ID}
    \begin{split}
        \ddt E_K(\phi,\psi)  &= \intO\phi\boldsymbol{v}\cdot\Grad\mu\dx + \intG \psi\boldsymbol{w}\cdot\Gradg\theta\dG 
        - \intO m_\Om(\phi)\abs{\Grad\mu}^2\dx\\
        &\qquad
         - \intG m_\Ga(\psi)\abs{\Gradg\theta}^2\dG
        - \sigma(L) \intG (\beta\theta-\mu)^2\dG 
    \end{split}
\end{align}
on $[0,T]$. Note that in the non-convective case, i.e., $\boldsymbol{v} \equiv \mathbf{0}$ and $\boldsymbol{w} \equiv \mathbf{0}$, the right-hand side of \eqref{INTRO:ENERGY:ID} is clearly non-positive. This means that the energy dissipates over the course of time, and the terms on the right-hand side of \eqref{INTRO:ENERGY:ID} can be interpreted as the dissipation rate. In the convective case, i.e., $\boldsymbol{v}$ and/or $\boldsymbol{w}$ are non-trivial, we can (in general) not infer dissipation of the energy from the energy identity \eqref{INTRO:ENERGY:ID}. However, if it holds $\boldsymbol{w} = \boldsymbol{v}_\tau$ and the velocity field $\boldsymbol{v}$ is determined by a Navier--Stokes equations (cf. \cite{Giorgini2023, Gal2023a}), or a Brinkman/Stokes equation (cf. \cite{Colli2024}), an energy dissipation for the corresponding total energy can be obtained.

\paragraph{State of the art and related literature.}
In the case of regular potentials, system \eqref{CH} has already been studied in \cite{Knopf2024} by the authors. There, the existence of a weak solution has first been established in the cases $K,L\in(0,\infty)$ by means of a Faedo--Galerkin scheme. For all other cases, the existence of a weak solution was shown by passing to the respective asymptotic limits, i.e., $K\rightarrow 0$, $K\rightarrow\infty$, $L\rightarrow 0$ or $L\rightarrow\infty$.

In the recent contributions \cite{Lv2024, Lv2024a, Lv2024b}, a similar system to \eqref{CH} with singular potentials was investigated. There, they considered the special case without convection terms (i.e., $\boldsymbol{v}\equiv \mathbf{0}$ and $\boldsymbol{w} \equiv \mathbf{0}$), and only the choices $(K,L)\in \big(\{0\}\times [0,\infty]\big) \cup \big([0,\infty) \times \{\infty\}\big)$.

For further results on the mathematical analysis of system \eqref{CH} in the non-convective case (i.e., $\boldsymbol{v}\equiv \mathbf{0}$ and $\boldsymbol{w} \equiv \mathbf{0}$), we refer to \cite{Garcke2020,Garcke2022,Colli2020,Colli2022,Colli2022a,Fukao2021,Miranville2020}.
Related numerical results can be found in \cite{Metzger2021,Metzger2023,Harder2022,Meng2023,Bao2021,Bao2021a}. 
A nonlocal variant of the non-convective version of \eqref{CH} was proposed and analyzed in \cite{Knopf2021b}. Further results on the Cahn--Hilliard equation with dynamic boundary conditions of second-order (e.g., of Allen--Cahn or Wentzell type) can be found, for instance, in \cite{Colli2015,Colli2014,Miranville2010,Racke2003,Wu2004,Gal2009,Gilardi2010,Cavaterra2011}.

\paragraph{Goals and novelties of this paper.}
\begin{enumerate}[label=\textnormal{\bfseries(\Roman*)},topsep=0ex]
    \item \textit{Well-posedness of the convective Cahn--Hilliard system \eqref{CH} with singular potentials.} 
    In the present paper, we first extend the results for regular potentials obtained in \cite{Knopf2024} to the case of singular potentials. 
    This means we establish the existence of a global weak solution (see Theorem~\ref{THEOREM:EOWS}) as well as continuous dependence on the initial data and the velocity fields in the case of constant mobilities (see Theorem~\ref{THEOREM:UNIQUE:SING}), which entails the uniqueness of the respective weak solution. 
    Thanks to \cite{Knopf2024}, we already have the existence of a global weak solution of \eqref{CH} at hand, provided that $F$ and $G$ are regular potentials. Therefore, our strategy to approach singular potentials is to regularize the subdifferentials of their convex parts by means of a Yosida approximation. Then, the approximate problem can be solved by means of the results established in \cite{Knopf2024}. 
    Next, we derive suitable bounds on the approximate solutions which are uniform with respect to the approximation parameter $\ep$ from the Yosida approximation. Eventually, using compactness arguments, we can pass to the limit $\ep\rightarrow 0$, which yields a weak solution to \eqref{CH} with the original singular potentials. The continuous dependence estimate can be proved by the energy method.     
    In this way, our new results extend the existing literature concerning the well-posedness of convective Cahn--Hilliard system with dynamic boundary conditions and singular potentials (see, e.g., \cite{Colli2018, Colli2019,Gilardi2019}).
    \item \textit{Existence of strong solutions.} The next step is to prove higher regularity results for weak solutions in the case of constant mobilities and under additional regularity assumptions of the velocity fields (see Theorem~\ref{thm:highreg}). These higher regularity results are established by approximating the time derivatives in \eqref{CH} by suitable difference quotients, deriving uniform estimates, and applying regularity theory for bulk-surface elliptic systems (see Proposition~\ref{Prop:Appendix} and \cite[Theorem 3.3]{Knopf2021a}).
    \item \textit{Separation properties.} 
    Eventually, we show strict separation properties for the unique strong solution (see Theorem~\ref{thm:sepprop}). 
    This means that the corresponding phase-fields will stay strictly away from the pure phases $\pm 1$. For a general class of singular potentials (see~\ref{S1}-\ref{S4}), we obtain:
    \begin{align}
        \label{sepprop:intro:3d}
        \begin{aligned}
        &\text{For almost all $t\in[0,T]$, there exists $\delta(t)\in(0,1]$ such that:}\\
        &
        \norm{\phi(t)}_{L^\infty(\Omega)} \leq 1 - \delta(t)
        \quad\text{and}\quad
        \norm{\psi(t)}_{L^\infty(\Gamma)} \leq 1 - \delta(t).
        \end{aligned}
    \end{align}
    If $G$ is a logarithmic type potential (see~\ref{S6}), $\psi$ even satisfies the uniform strict separation property:
    \begin{align}
        \label{sepprop:psi:intro}
        \text{There exists $\delta_\ast\in(0,1]$ such that for all $(z,t)\in\Sigma$:}\quad
        \abs{\psi(z,t)}\leq 1 - \delta_\ast.
    \end{align}
    If $d=2$ and $F$ is a logarithmic type potential (see~\ref{S5}), we analogously obtain the uniform strict separation property for $\phi$:
    \begin{align}
        \label{sepprop:phi:intro:2d}
        \text{There exists $\delta_\star\in(0,1]$ such that for all $(x,t) \in Q$:}\quad
        \abs{\phi(x,t)} \leq 1 - \delta_\star.
    \end{align}
    It is not surprising that this result only holds in two dimensions, since even for the standard Cahn--Hilliard equation with homogeneous Neumann boundary conditions on both $\phi$ and $\mu$, the uniform strict separation property \eqref{sepprop:phi:intro:2d} has not yet been established in three dimensions. However, as the boundary $\Gamma$ is a two-dimensional submanifold, we are still able to obtain \eqref{sepprop:psi:intro} by adapting the techniques for the two-dimensional Cahn--Hilliard equation in the bulk.

    For the non-convective bulk-surface Cahn--Hilliard system (i.e., \eqref{CH} with $\boldsymbol{v}\equiv \mathbf{0}$ and $\boldsymbol{w}\equiv \mathbf{0}$), such separation properties have already been established in the cases $(K,L)\in \big(\{0\}\times [0,\infty]\big) \cup \big([0,\infty) \times \{\infty\}\big)$ (see \cite{Fukao2021, Lv2024a, Lv2024b}).

    In the present paper, we extend these results by showing the above separation properties for the \textit{convective} bulk-surface Cahn--Hilliard sysem (i.e., with non-trivial velocity fields $\boldsymbol{v}$ and $\boldsymbol{w}$) for all parameter choices $(K,L) \in [0,\infty]^2$.

    We further refer to \cite{Garcke2023} and \cite{Gal2017,Giorgini2024,Poiatti2025} for strict separation properties of the anisotropic Cahn--Hilliard equation and the nonlocal Cahn--Hilliard equation, respectively.
\end{enumerate}

\paragraph{Structure of this paper.} After introducing the notion of a weak solution of \eqref{EQ:SYSTEM}, we state our main results in Section~\ref{SECT:MAINRESULTS}. In Section~\ref{SECT:EOWS}, we prove the existence of weak solutions to the Cahn--Hilliard system \eqref{EQ:SYSTEM}. Afterwards, in Section~\ref{SECT:UNIQUENESS}, we establish the uniqueness of the weak solutions and its continuous dependence on the velocity fields and the initial data. Section~\ref{SECT:HIGHREG} is devoted to proving higher regularity results and in particular, the existence of strong solutions under suitable additional assumptions. Lastly, for logarithmic type potentials, we establish the separation properties in Section~\ref{SECT:SEPARATION}.

\section{Notation, assumptions and preliminaries} 
\label{SECT:PRELIM}

We fix some notation and assumptions that are supposed to hold throughout the remainder of this paper.

\subsection{Notation}
\begin{enumerate}[label=\textnormal{\bfseries(N\arabic*)}]
    \item $\N$ denotes the set of natural numbers excluding zero, whereas $\N_0 = \N\cup\{0\}$.
    
    \item Let $\Omega \subset \R^d$ with $d\in\{2,3\}$ be a bounded domain with sufficiently regular boundary $\Gamma \coloneqq\del\Omega$. For any $p\in[1,\infty]$ and $s\geq 0$, the Lebesgue and Sobolev spaces for functions mapping from $\Om$ to $\R$ are denoted as $L^p(\Om)$ and $W^{s,p}(\Om)$, respectively. Their standard norms are denoted by $\norm{\cdot}_{L^p(\Om)}$ and $\norm{\cdot}_{W^{s,p}(\Om)}$. In the case $p=2$, we write $H^s(\Om) = W^{s,2}(\Om)$. In particular, $H^0(\Om)$ can be identified with $L^2(\Om)$. The Lebesgue and Sobolev spaces on $\Ga$ are defined analogously, provided that $\Gamma$ is sufficiently regular.
    For vector-valued functions mapping from $\Om$ into $\R^d$, we write $\mathbf{L}^p(\Om)$, $\mathbf{W}^{s,p}(\Om)$ and $\mathbf{H}^s(\Om)$. For functions mapping from $\Ga$ to $\R^d$, we use an analogous notation.
    For any real numbers $p\in[1,\infty]$ and $s\geq 0$ and any Banach space $X$, the Bochner spaces of functions mapping from an interval $I$ into $X$ are denoted by $L^p(I;X)$ and $W^{s,p}(I;X)$.
    Furthermore, for any interval $I$ and any Banach space $X$, we denote the space of continuous functions mapping from $I$ to $X$ by $C(I;X)$.    
    
    \item For any Banach space $X$, its dual space is denoted by $X'$. The associated duality pairing of elements $\phi\in X'$ and $\zeta\in X$ is denoted by $\ang{\phi}{\zeta}_X$. If $X$ is a Hilbert space, denote its inner product by $\scp{\cdot}{\cdot}_X$. 
    
    \item For any bounded domain $\Om\subset \R^d$ ($d\in\N$) with Lipschitz boundary $\Ga$, and any $u\in H^1(\Om)'$ and $v\in H^1(\Ga)'$, we write
    \begin{align*}
        \meano{u}\coloneqq 
        \frac{1}{\abs{\Om}}\ang{u}{1}_{H^1(\Om)},
        \qquad
        \meang{v}\coloneqq 
        \frac{1}{\abs{\Ga}}\ang{v}{1}_{H^1(\Ga)}
    \end{align*}
    to denote the generalized means of $u$ and $v$, respectively. Here, $\abs{\Om}$ denotes the $d$-dimensional Lebesgue measure of $\Omega$, whereas $\abs{\Ga}$ denotes the $(d-1)$-dimensional Hausdorff measure of $\Gamma$. If $u\in L^1(\Om)$ or $v\in L^1(\Ga)$, their generalized mean can be expressed as
    \begin{align*}
        \meano{u} = \frac{1}{\abs{\Om}} \intO u \dx,
        \qquad
        \meang{v} = \frac{1}{\abs{\Ga}} \intG v \dG,
    \end{align*}
    respectively.
    \item For any bounded domain $\Om\subset \R^d$ ($d\in\N$) with Lipschitz boundary $\Gamma\coloneqq\partial\Omega$ and any $2\leq p <\infty$, we introduce the spaces
    \begin{align*}
        \mathbf{L}^p_\Div(\Om) &\coloneqq \{\boldsymbol{v}\in\mathbf{L}^p(\Om) : \Div\;\boldsymbol{v} = 0 \ \text{in~} \Om, \ \boldsymbol{v}\cdot\mathbf{n} = 0 \ \text{on~} \Ga\},
        \\
        \mathbf{L}^p_\tau(\Ga)&\coloneqq\{\boldsymbol{w}\in\mathbf{L}^p(\Ga) : \boldsymbol{w}\cdot\n = 0 \ \text{on~}\Ga\},
        \\
        \mathbf{L}^p_\Div(\Ga)&\coloneqq\{\boldsymbol{w}\in\mathbf{L}^p_\tau(\Ga) : \Divg\;\boldsymbol{w} = 0 \ \text{on~}\Ga\}.
    \end{align*}
    We point out that in the definitions of $\mathbf{L}^p_\Div(\Om)$ and $\mathbf{L}^p_\Div(\Ga)$, the expressions $\Div\;\boldsymbol{v}$ in $\Om$ and $\Divg\;\boldsymbol{w}$ on $\Ga$ are to be understood in the sense of distributions. If $\boldsymbol{v}\in \mathbf{L}^p(\Om)$ ($p\ge 2$) with $\Div\;\boldsymbol{v} = 0$ in $\Om$, we already know that $\boldsymbol{v}\cdot\mathbf{n} \in H^{-1/2}(\Ga)$, and therefore, the relation $\boldsymbol{v}\cdot\mathbf{n} = 0$ on $\Ga$ is well-defined.

\end{enumerate}

\subsection{Assumptions}

\begin{enumerate}[label=\textnormal{\bfseries(A\arabic*)}]
    \item  \label{ASSUMP:1} We consider a bounded domain $\emptyset\neq \Om\subset\R^d$ with $d\in\{2,3\}$ with Lipschitz boundary $\Ga\coloneqq\del\Om$ and a final time $T>0$. We further use the notation
    \begin{align*}
        Q\coloneqq \Om\times(0,T), \quad\Sigma\coloneqq\Ga\times(0,T).
    \end{align*}
    
    \item \label{ASSUMP:2} The constants occurring in the system \eqref{CH} satisfy $\epsilon, \epsilon_\Ga, \kappa > 0$, and $\alpha, \beta\in \R$ with $\alpha\beta\abs{\Omega} + \abs{\Gamma} \neq 0$. Since the choice of $\epsilon, \epsilon_\Ga$ and $\kappa$ has no impact on the mathematical analysis, we will simply set $\epsilon = \epsilon_\Ga = \kappa = 1$ without loss of generality.
    
    \item \label{ASSUMP:MOBILITY} The mobility functions $m_\Om:\R\rightarrow\R$ and $m_\Ga:\R\rightarrow\R$ are continuous, bounded, and uniformly positive. This means that there exist positive constants $m_\Om^\ast, M_\Om^\ast, m_\Ga^\ast, M_\Ga^\ast$ such that for all $s\in\R$,
    \begin{align*}
        0 < m_\Om^\ast \leq m_\Om(s) \leq M_\Om^\ast, \quad\text{and}\quad 0 < m_\Ga^\ast \leq m_\Ga(s) \leq M_\Ga^\ast.
    \end{align*}
    
\end{enumerate}

For the potentials $F,G :\R\rightarrow[0,\infty]$, we assume the decompositions $F = F_1 + F_2$ and $G = G_1 + G_2$ with the following properties:
\begin{enumerate}[label=\textnormal{\bfseries(S\arabic*)}]
    \item \label{S1} $F_1,G_1:\R\rightarrow[0,\infty]$ are proper, lower semicontinuous and convex functions with $F_1(0) = G_1(0) = 0$. Their effective domains are denoted by $D(F_1)$ and $D(G_1)$. We further define
    \begin{align*}
        f_1 \coloneqq \del F_1, \quad g_1 \coloneqq \del G_1,
    \end{align*}
    where $\del$ indicates the subdifferential. These subdifferentials are maximal monotone graphs in $\R\times\R$ (see, e.g., \cite[Theorem 2.8]{Barbu2010}), whose effective domains are denoted by $D(f_1)$ and $D(g_1)$, respectively, and satisfy $0\in f_1(0) \cap g_1(0)$. To be precise, it holds
    \begin{align*}
        D(f_1) \coloneqq\{r\in D(F_1) : f_1(r)\neq\emptyset\},
    \end{align*}
    and $D(g_1)$ is defined analogously.
    Moreover, we postulate the growth conditions
    \begin{align}\label{growth-sing}
        \lim_{\abs{r}\rightarrow\infty} \frac{F_1(r)}{\abs{r}^2} = +\infty
        \quad\text{and}\quad
        \lim_{\abs{r}\rightarrow\infty} \frac{G_1(r)}{\abs{r}^2} = +\infty.
    \end{align}
    
    \item \label{S2} $F_2, G_2:\R\rightarrow\R$ are continuously differentiable functions with Lipschitz continuous derivatives $f_2 \coloneqq F_2^\prime$ and $g_2 \coloneqq  G_2^\prime$, respectively.
    \item \label{S3} It holds $\alpha D(g_1) \subseteq D(f_1)$, that is for any $r\in D(g_1)$ we have $\alpha r\in D(f_1)$, and the boundary graph dominates the bulk graph in the following sense: There exist $\kappa_1,\kappa_2 > 0$ such that
    \begin{align}\label{domination}
        \abs{f_1^\circ(\alpha r)} \leq\kappa_1\abs{g_1^\circ(r)} + \kappa_2
        \quad\text{for all $r\in D(g_1)$.}
    \end{align}
    Here, $\alpha$ denotes the constant occurring in \eqref{CH} (cf.~\ref{ASSUMP:2}) and $f_1^\circ$ denotes the minimal section of the graph $f_1$, i.e., 
    \begin{align*}
        f_1^\circ(r) \coloneqq \Big\{ r^\ast \in f_1(r) : \abs{r^\ast} = \min_{s\in f_1(r)} \abs{s} \Big\}
        \quad\text{for all $r\in D(f_1)$.}
    \end{align*}
    The minimal section $g_1^\circ$ of the graph $g_1$ is defined analogously. It is well-known that $f_1^\circ$ and $g_1^\circ$ are single-valued. Therefore, they can be interpreted as functions
    \begin{align*}
        f_1^\circ:D(f_1) \to \R
        \quad\text{and}\quad
        g_1^\circ:D(g_1) \to \R.
    \end{align*}
\end{enumerate}
$\;$

For some results that will be established in this paper, we need additional assumptions on the potentials $F$ and $G$.
\begin{enumerate}[label=\textnormal{\bfseries(S\arabic*)},start=4]
    \item \label{S4} It additionally holds $F_1,G_1\in C([-1,1])\cap C^2(-1,1)$,
    $F_1(s) = G_1(s) = +\infty$ for all $s\in\R\setminus [-1,1]$, and
    \begin{align*}
        \lim_{s\searrow -1} F_1^\prime(s) = \lim_{s\searrow -1} G_1^\prime(s) = -\infty \quad\text{and}\quad
        \lim_{s\nearrow 1} F_1^\prime(s) = \lim_{s\nearrow 1} G_1^\prime(s) = +\infty.
    \end{align*}
    Moreover, there exists a constant $\Theta > 0$ such that
    \begin{align*}
        F_1^{\prime\prime}(s) \geq \Theta \quad\text{and}\quad G_1^{\prime\prime}(s)\geq \Theta
        \quad\text{for all $s\in(-1,1)$.}
    \end{align*}
    This means that $F_1$ and $G_1$ are strongly convex on $[-1,1]$.
    Without loss of generality, we assume $F_1(0) = G_1(0) = 0$ and $F_1^\prime(0) = G_1^\prime(0) = 0$. In particular, we thus have 
    \begin{align*}
        F_1(s) \geq 0 \quad\text{and}\quad G_1(s) \geq 0
        \quad\text{for all $s\in[-1,1]$}
    \end{align*}
    due to the strong convexity of $F_1$ and $G_1$, respectively. Furthermore, we assume that 
    \begin{align*}
        -1 \leq \alpha \leq 1.
    \end{align*}
    \item \label{S5} In addition to the assumptions made in \ref{S4}, there exist constants $C_1,C_2 > 0$ as well as $\lambda\in[1,2)$ such that
    \begin{align}
        \label{EST:LOGPOT}
        \abs{F_1^{\prime\prime}(s)} \leq C_1 \mathrm{e}^{C_2 \abs{F_1^\prime(s)}^\lambda}
        \quad\text{for all~}s\in(-1,1).
    \end{align}
    \item \label{S6} In addition to the assumptions made in \ref{S4}, there exist constants $c_1,c_2 > 0$ as well as $\gamma\in[1,2)$ such that
    \begin{align}
        \label{EST:LOGPOT:G}
        \abs{G_1^{\prime\prime}(s)} \leq c_1 \mathrm{e}^{c_2 \abs{G_1^\prime(s)}^\gamma}
        \quad\text{for all~}s\in(-1,1).
    \end{align}
\end{enumerate}

Similar assumptions on the potentials have first been made in \cite{Calatroni2013}, and then in various contributions in the literature; see, e.g., \cite{Colli2015, Colli2020, Colli2022, Colli2014, Colli2017, Colli2018, Colli2018a, Colli2018b, Colli2022a, Colli2024}.

\medskip

\begin{remark}
    We point out that \ref{S3} excludes the cases where $F_1$ is a singular potential and $G_1$ is a regular potential, unless $\alpha=0$.
    
    Let us consider the classical choices $W_\mathrm{reg}$, $W_\mathrm{log}$ and $W_\mathrm{obst}$ introduced in \eqref{DEF:W:REG}, \eqref{DEF:W:LOG} and \eqref{DEF:W:DOB}.     
    These potentials can be decomposed into
    \begin{gather*}
        W_\mathrm{reg} = W_{\mathrm{reg},1} + W_{\mathrm{reg},2},
        \quad
        W_\mathrm{log} = W_{\mathrm{log},1} + W_{\mathrm{log},2},
        \quad
        W_\mathrm{reg} = W_{\mathrm{obst},1} + W_{\mathrm{obst},2}
    \end{gather*}
    with
    \begin{alignat*}{2}
        W_{\mathrm{reg},1}(s) &= c(s^4+1)
        &&\quad\text{for~} s\in D(W_{\mathrm{reg},1}) = \R,
        \\
        W_{\mathrm{reg},1}'(s) &= 4cs^3
        &&\quad\text{for~} s\in D(\del W_{\mathrm{reg},1}) = \R,
        \\
        W_{\mathrm{reg},2}(s) &= -2cs^2
        &&\quad\text{for~} s\in \R,
        \\[1ex]
        W_{\mathrm{log},1}(s) &= \tfrac{1}{2}\Theta\big[(1+s)\ln(1+s) + (1-s)\ln(1-s)\big]
        &&\quad\text{for~} s\in D(W_{\mathrm{log},1}) = [-1,1],
        \\
        W_{\mathrm{log},1}'(s) &= \tfrac{1}{2}\Theta\big[\ln(1+s) - \ln(1-s)\big]
        &&\quad\text{for~} s\in D(\del W_{\mathrm{log},1}) = (-1,1),
        \\
        W_{\mathrm{log},2}(s) &= -\tfrac{1}{2}\Theta_cs^2
        &&\quad\text{for~} s\in \R,
        \\[1ex]
        W_{\mathrm{obst},1}(s) &= 0
        &&\quad\text{for~} s\in D(W_{\mathrm{obst},1}) = [-1,1],
        \\
        W_{\mathrm{obst},1}'(s) &= 0
        &&\quad\text{for~} s\in (-1,1),
        \\
        W_{\mathrm{obst},2}(s) &= 1-s^2
        &&\quad\text{for~} s\in \R.
    \end{alignat*}
    In the above formula for $W_{\mathrm{log},1}$, we used the convention $0\ln 0 = 0$. We further have $D(\del W_{\mathrm{obst},1}) = [-1,1]$ with $\del W_{\mathrm{obst},1}(-1) = (-\infty,0]$ and $\del W_{\mathrm{obst},1}(-1) = [0,\infty)$.
    Therefore, the following holds:
    \begin{enumerate}[label={$\bullet$},topsep=0ex]
        \item The choice $(F,G) = (W_\mathrm{log},W_\mathrm{log})$ satisfies \ref{S1}-\ref{S6}.
        \item The choices 
        $(F,G) = (W_\mathrm{reg},W_\mathrm{reg})$, 
        $(F,G) = (W_\mathrm{reg},W_\mathrm{log})$ and  
        $(F,G) = (W_\mathrm{reg},W_\mathrm{obst})$  
        satisfy \ref{S1}-\ref{S3}, but not \ref{S4} and \ref{S5}. 
        Moreover, \ref{S6} is fulfilled if $G=W_\mathrm{log}$.
        
        \item The choices 
        $(F,G) = (W_\mathrm{log},W_\mathrm{reg})$ and 
        $(F,G) = (W_\mathrm{obst},W_\mathrm{reg})$ 
        satisfy \ref{S1}-\ref{S3} if and only if $\alpha=0$. 
        
        If $\alpha\neq 0$,
        these choices are not allowed since $\alpha D(g_1) \not\subset  D(f_1)$ and thus, \ref{S3} is not fulfilled.

        Obviously, \ref{S4} and \ref{S6} are not fulfilled for any choice of $\alpha\in\R$.
        Moreover, \ref{S5} is fulfilled if $F=W_\mathrm{log}$.
        
        \item The choice 
        $(F,G)=(W_\mathrm{log}, W_\mathrm{obst})$ satisfies \ref{S1}-\ref{S3}
        if and only if $\abs{\alpha}<1$. 
        
        For $\abs{\alpha}< 1$, we have 
        $$\alpha D(g_1) = \big[-\abs{\alpha}, \abs{\alpha}\big] \subset (-1,1) = D(f_1),$$
        and the domination property \eqref{domination} holds since $W_{\mathrm{log},1}'$ 
        remains bounded on $\alpha D(g_1)$. 
        
        For $\abs{\alpha}\ge 1$, the choice 
        $(F,G)=(W_\mathrm{log}, W_\mathrm{obst})$
        is not allowed since $\alpha D(g_1) \not\subset D(f_1)$ and thus, \ref{S3} is not fulfilled.

        Obviously, \ref{S4} and \ref{S6} are not fulfilled for any choice of $\alpha\in\R$.
        However, \ref{S5} is fulfilled since $F=W_\mathrm{log}$.

        \item The choices $(F,G) = (W_\mathrm{obst},W_\mathrm{log})$ and         $(F,G) = (W_\mathrm{obst},W_\mathrm{obst})$ satisfy \ref{S1}-\ref{S3} 
        if and only if $\abs{\alpha} \le 1$.

        For $\abs{\alpha}\le 1$, we have 
        $$\alpha D(g_1) \subset \big[-\abs{\alpha}, \abs{\alpha}\big] \subset [-1,1] = D(f_1),$$
        and the domination property \eqref{domination} holds since $f_1^\circ \equiv 0$ on $\alpha D(g_1)$. 
        
        For $\abs{\alpha}> 1$, the choices 
        $(F,G) = (W_\mathrm{obst},W_\mathrm{log})$ and         
        $(F,G) = (W_\mathrm{obst},W_\mathrm{obst})$
        are not allowed since $\alpha D(g_1) \not\subset D(f_1)$ and thus, \ref{S3} is not fulfilled.

        Obviously, \ref{S4} and \ref{S5} are not fulfilled for any choice of $\alpha\in\R$. Moreover, \ref{S6} is fulfilled if $G=W_\mathrm{log}$.
    \end{enumerate}
\end{remark}

\subsection{Preliminaries}

\begin{enumerate}[label=\textnormal{\bfseries(P\arabic*)}]
    \item For any real numbers $s\geq 0$ and $p\in[1,\infty]$, we set
    \begin{align*}
        \mathcal{L}^p \coloneqq L^p(\Om)\times L^p(\Ga), \quad\text{and}\quad \mathcal{W}^{s,p}\coloneqq W^{s,p}(\Om)\times W^{s,p}(\Ga),
    \end{align*}
    provided that $\Gamma$ is sufficiently regular. We abbreviate $\mathcal{H}^s \coloneqq \mathcal{W}^{s,2}$ and identify $\mathcal{L}^2$ with $\mathcal{H}^0$. Note that $\mathcal{H}^s$ is a Hilbert space with respect to the inner product
    \begin{align*}
        \bigscp{\scp{\phi}{\psi}}{\scp{\zeta}{\xi}}_{\mathcal{H}^s} \coloneqq \scp{\phi}{\zeta}_{H^s(\Om)} + \scp{\psi}{\xi}_{H^s(\Ga)} \quad\text{for all } \scp{\phi}{\psi}, \scp{\zeta}{\xi}\in\mathcal{H}^s
    \end{align*}
    and its induced norm $\norm{\cdot}_{\mathcal{H}^s} \coloneqq \scp{\cdot}{\cdot}_{\mathcal{H}^s}^{1/2}$. We recall that the duality pairing can be expressed as
    \begin{align*}
        \bigang{\scp{\phi}{\psi}}{\scp{\zeta}{\xi}}_{\mathcal{H}^s} = \scp{\phi}{\zeta}_{L^2(\Om)} + \scp{\psi}{\xi}_{L^2(\Ga)}
    \end{align*}
    for all $\scp{\zeta}{\xi}\in \mathcal{H}^s$ if $\scp{\phi}{\psi}\in \mathcal{L}^2$.
    
    \item Let $L\in[0,\infty]$ and $\beta\in\R$ be real numbers. We introduce the subspace
    \begin{align*}
        \mathcal{H}_{L,\beta}^1 \coloneqq
        \begin{cases}
            \mathcal{H}^1, &\text{if } L \in (0,\infty] , \\
             \{(\phi,\psi)\in\mathcal{H}^1 : \phi = \beta\psi \text{ a.e.~on } \Ga\}, &\text{if } L=0,
        \end{cases}
    \end{align*}
    endowed with the inner product $\scp{\cdot}{\cdot}_{\mathcal{H}_{L,\beta}^1} \coloneqq \scp{\cdot}{\cdot}_{\mathcal{H}^1}$ and its induced norm. The space $\mathcal{H}_{L,\beta}^1$ is a Hilbert space. Moreover, we define the product
    \begin{align*}
        \ang{\scp{\phi}{\psi}}{\scp{\zeta}{\xi}}_{\mathcal{H}_{L,\beta}^1} \coloneqq \scp{\phi}{\zeta}_{L^2(\Om)} + \scp{\psi}{\xi}_{L^2(\Ga)}
    \end{align*}
    for all $\scp{\phi}{\psi}, \scp{\zeta}{\xi}\in\mathcal{L}^2$. By means of the Riesz representation theorem, this product can be extended to a duality pairing on $(\mathcal{H}_{L,\beta}^1)^\prime\times\mathcal{H}_{L,\beta}^1$, which will also be denoted as $\ang{\cdot}{\cdot}_{\mathcal{H}_{L,\beta}^1}$.
    
    \item Let $L\in[0,\infty]$ and $\beta\in\R$ be real numbers. We define for $\scp{\phi}{\psi}\in(\mathcal{H}^1_{L,\beta})^\prime$ the generalized bulk-surface mean
    \begin{align*}
        \mean{\phi}{\psi} \coloneqq \frac{\ang{\scp{\phi}{\psi}}{\scp{\beta}{1}}_{\mathcal{H}^1_{L,\beta}}}{\beta^2\abs{\Om} + \abs{\Ga}},
    \end{align*}
    which reduces to
    \begin{align*}
        \mean{\phi}{\psi} = \frac{\beta\abs{\Om}\meano{\phi} + \abs{\Ga}\meang{\psi}}{\beta^2\abs{\Om} + \abs{\Ga}}
    \end{align*}
    if $\phi\in L^2(\Om)$ and $\psi\in L^2(\Ga)$. We then define the closed linear subspaces
    \begin{align*}
        \mathcal{V}_{L,\beta}^1 &\coloneqq \begin{cases} 
        \{\scp{\phi}{\psi}\in\mathcal{H}^1_{L,\beta} : \mean{\phi}{\psi} = 0 \}, &\text{if~} L\in[0,\infty), \\
        \{\scp{\phi}{\psi}\in\mathcal{H}^1: \meano{\phi} = \meang{\psi} = 0 \}, &\text{if~}L=\infty.
        \end{cases} 
    \end{align*}
    Note that these subspaces are Hilbert spaces with respect to the inner product $\scp{\cdot}{\cdot}_{\mathcal{H}^1}$.
    
    \item Let $L\in[0,\infty]$ and $\beta\in\R$ be real numbers. We set
    \begin{align*}
        \sigma(L) \coloneqq
        \begin{cases}
            L^{-1}, &\text{if } L\in(0,\infty), \\
            0, &\text{if } L\in\{0,\infty\},
        \end{cases}
    \end{align*}
    and we define a bilinear form on $\mathcal{H}^1\times\mathcal{H}^1$ by
    \begin{align*}
         \bigscp{\scp{\phi}{\psi}}{\scp{\zeta}{\xi}}_{L,\beta} \coloneqq &\intO\Grad\phi\cdot\Grad\zeta \dx + \intG\Gradg\psi\cdot\Gradg\xi \dG \\  
         &\quad + \sigma(L)\intG (\beta\psi-\phi)(\beta\xi-\zeta)\dG
    \end{align*}
    for all $ \scp{\phi}{\psi}, \scp{\zeta}{\xi}\in\mathcal{H}^1$. Moreover, we set 
    \begin{align*}
        \norm{\scp{\phi}{\psi}}_{L,\beta} \coloneqq \bigscp{\scp{\phi}{\psi}}{\scp{\phi}{\psi}}_{L,\beta}^{1/2}
    \end{align*}
    for all $\scp{\phi}{\psi}\in\mathcal{H}^1$. The bilinear form $\scp{\cdot}{\cdot}_{L,\beta}$ defines an inner product on $\mathcal{V}^1_{L,\beta}$, and $\norm{\cdot}_{L,\beta}$ defines a norm on $\mathcal{V}^1_{L,\beta}$, that is equivalent to the norm $\norm{\cdot}_{\mathcal{H}^1}$ (see \cite[Corollary~A.2]{Knopf2021}). The space $\big(\mathcal{V}^1_{L,\beta}, \scp{\cdot}{\cdot}_{L,\beta}, \norm{\cdot}_{L,\beta}\big)$ is a Hilbert space.
    
    \item \label{PRELIM:bulk-surface-elliptic} For any $L\in[0,\infty]$ and $\beta\in\R$, we define the spaces
    \begin{align*}
        \mathcal{V}_{L,\beta}^{-1} \coloneqq \begin{cases} 
        \{\scp{\phi}{\psi}\in(\mathcal{H}^1_{L,\beta})^\prime : \mean{\phi}{\psi} = 0 \}, &\text{if~} L\in[0,\infty), \\
        \{\scp{\phi}{\psi}\in(\mathcal{H}^1)^\prime: \meano{\phi} = \meang{\psi} = 0 \}, &\text{if~}L=\infty.
        \end{cases}
    \end{align*}
    Using the Lax--Milgram theorem, one can show that for any $\scp{\phi}{\psi}\in\mathcal{V}_{L,\beta}^{-1}$, there exists a unique weak solution $\mathcal{S}_{L,\beta}(\phi,\psi) = \bigscp{\mathcal{S}_{L,\beta}^\Om(\phi,\psi)}{\mathcal{S}_{L,\beta}^\Ga(\phi,\psi)}\in\mathcal{V}^1_{L,\beta}$ to the elliptic problem with bulk-surface coupling
    \begin{alignat}{3}
        -\Lap\mathcal{S}_{L,\beta}^\Om &= - \phi \qquad &&\text{in } \Om, \\
        -\Lapg\mathcal{S}_{L,\beta}^\Ga + \beta\deln\mathcal{S}_{L,\beta}^\Om &= -\psi &&\text{on } \Ga, \\
        L\deln\mathcal{S}_{L,\beta}^\Om &= \beta\mathcal{S}_{L,\beta}^\Ga - \mathcal{S}_{L,\beta}^\Om \qquad&&\text{on } \Ga.
    \end{alignat}
    This means that $\mathcal{S}_{L,\beta}(\phi,\psi)$ satisfies the weak formulation
    \begin{align*}
        \bigscp{S_{L,\beta}(\phi,\psi)}{\scp{\zeta}{\xi}}_{L,\beta} = - \ang{\scp{\phi}{\psi}}{\scp{\zeta}{\xi}}_{\mathcal{H}^1_{L,\beta}}
    \end{align*}
    for all test functions $\scp{\xi}{\zeta}\in\mathcal{H}^1_{L,\beta}$. Consequently, we have
    \begin{align*}
        \norm{\mathcal{S}_{L,\beta}(\phi,\psi)}_{L,\beta} \leq \norm{\scp{\phi}{\psi}}_{(\mathcal{H}^1_{L,\beta})^\prime}
    \end{align*}
    for all $\scp{\phi}{\psi}\in\mathcal{V}^{-1}_{L,\beta}$. Thus, we can define the solution operator
    \begin{align*}
        \mathcal{S}_{L,\beta}:\mathcal{V}_{L,\beta}^{-1}\rightarrow\mathcal{V}^1_{L,\beta}, \quad \scp{\phi}{\psi}\mapsto\mathcal{S}_{L,\beta}(\phi,\psi) = \bigscp{\mathcal{S}_{L,\beta}^\Om(\phi,\psi)}{\mathcal{S}_{L,\beta}^\Ga(\phi,\psi)},
    \end{align*}
    as well as an inner product and its induced norm on $\mathcal{V}^{-1}_{L,\beta}$ by
    \begin{align*}
        \bigscp{\scp{\phi}{\psi}}{\scp{\zeta}{\xi}}_{L,\beta,\ast} &\coloneqq \bigscp{\mathcal{S}_{L,\beta}(\phi,\psi)}{\mathcal{S}_{L,\beta}(\zeta,\xi)}_{L,\beta}, \\
        \norm{\scp{\phi}{\psi}}_{L,\beta,\ast} &\coloneqq \bigscp{\scp{\phi}{\psi}}{\scp{\phi}{\psi}}_{L,\beta,\ast}^{1/2},
    \end{align*}
    for $\scp{\phi}{\psi}, \scp{\zeta}{\xi}\in\mathcal{V}_{L,\beta}^{-1}$. This norm is equivalent to the norm $\norm{\cdot}_{(\mathcal{H}^1_{L,\beta})^\prime}$ on $\mathcal{V}_{L,\beta}^{-1}$. For the case $L\in (0,\infty)$, we refer the reader to \cite[Theorem 3.3 and Corollary 3.5]{Knopf2021} for a proof of these statements. In the other cases, the results can be proven analogously.

    \item \label{PRELIM:POINCINEQ} We further recall the following \textit{bulk-surface Poincar\'{e} inequality}, which has been established in \cite[Lemma A.1]{Knopf2021}: \\[0.3em]
    Let $K\in[0,\infty)$ and $\alpha,\beta\in\mathbb{R}$ with $\alpha\beta\abs{\Om} + \abs{\Ga} \neq 0$ be arbitrary. Then there exists a constant $C_P >0$ depending only on $K, \alpha, \beta$ and $\Omega$ such that
    \begin{align*}
        \norm{\scp{\phi}{\psi}}_{\mathcal{L}^2} \leq C_P \norm{\scp{\phi}{\psi}}_{K,\alpha}
    \end{align*}
    for all pairs $\scp{\phi}{\psi}\in\mathcal{H}^1_{K,\alpha}$ satisfying $\mean{\phi}{\psi} = 0$. 

\end{enumerate}

\section{Main results}
\label{SECT:MAINRESULTS}

As mentioned in \ref{ASSUMP:2}, we set $\epsilon = \epsilon_\Ga = \kappa = 1$. This does not mean any loss of generality as the exact values of $\epsilon$, $\epsilon_\Ga$, and $\kappa$ do not have any impact on the mathematical analysis (as long as they are positive). With this choice, the system \eqref{CH} can be restated as follows:%
\begin{subequations}\label{EQ:SYSTEM}
    \begin{align}
        \label{EQ:SYSTEM:1}
        &\delt\phi + \Div(\phi\boldsymbol{v}) = \Div(m_\Om(\phi)\Grad\mu) && \text{in} \ Q, \\
        \label{EQ:SYSTEM:2}
        &\mu = -\Lap\phi + F'(\phi)   && \text{in} \ Q, \\
        \label{EQ:SYSTEM:3}
        &\delt\psi + \Divg(\psi\boldsymbol{w}) = \Divg(m_\Ga(\psi)\Gradg\theta) - \beta m_\Om(\phi)\deln\mu && \text{on} \ \Sigma, \\
        \label{EQ:SYSTEM:4}
        &\theta = - \Lapg\psi + G'(\psi) + \alpha\deln\phi && \text{on} \ \Sigma, \\
        \label{EQ:SYSTEM:5}
        &\begin{cases} K\deln\phi = \alpha\psi - \phi &\text{if} \ K\in [0,\infty), \\
        \deln\phi = 0 &\text{if} \ K = \infty
        \end{cases} && \text{on} \ \Sigma, \\
        \label{EQ:SYSTEM:6}
        &\begin{cases} 
        L m_\Om(\phi)\deln\mu = \beta\theta - \mu &\text{if} \  L\in[0,\infty), \\
        m_\Om(\phi)\deln\mu = 0 &\text{if} \ L=\infty
        \end{cases} &&\text{on} \ \Sigma, \\
        \label{EQ:SYSTEM:7}
        &\phi\vert_{t=0} = \phi_0 &&\text{in} \ \Om, \\
        \label{EQ:SYSTEM:8}
        &\psi\vert_{t=0} = \psi_0 &&\text{on} \ \Ga.
    \end{align}
\end{subequations}

The total energy associated with this system reads as
\begin{align}\label{energy:K}
    \begin{split}
        E_K(\phi,\psi) &= \intO\frac{1}{2} \abs{\Grad\phi}^2 + F(\phi) \dx + \intG\frac{1}{2} \abs{\Gradg\psi}^2 + G(\psi) \dG \\
        &\quad + \sigma(K) \intG \frac{1}{2} \abs{\alpha\psi - \phi}^2 \dG,
    \end{split}
\end{align}
for $K\in[0,\infty]$.

\subsection{Weak solutions for possibly singular potentials}

The notion of a weak solution to system \eqref{EQ:SYSTEM} is defined as follows:

\begin{definition}[Weak solutions of system \eqref{EQ:SYSTEM}]\label{DEF:SING:WS}
    Suppose that the assumptions \ref{ASSUMP:1}-\ref{ASSUMP:MOBILITY} and \ref{S1}-\ref{S2} hold. Let $K,L\in[0,\infty]$, $\boldsymbol{v}\in L^2(0,T;\mathbf{L}^3_{\Div}(\Om))$ and $\boldsymbol{w}\in L^2(0,T;\mathbf{L}_\tau^{2+\omega}(\Ga))$ for some $\omega > 0$ be given velocity fields, and let $\scp{\phi_0}{\psi_0}\in\mathcal{H}^1_{K,\alpha}$ be an arbitrary initial datum satisfying
    \begin{subequations}\label{cond:init}
    \begin{align}\label{cond:init:int}
        F_1(\phi_0)\in L^1(\Om), \quad G_1(\psi_0)\in L^1(\Ga).
    \end{align}
    Furthermore, in the case $L\in[0,\infty)$ we assume that
    \begin{align}\label{cond:init:mean:L}
        \beta\,\mean{\phi_0}{\psi_0}\in\textup{int}(D(f_1)), \quad \mean{\phi_0}{\psi_0}\in\textup{int}(D(g_1)),
    \end{align}
    while in the case $L=\infty$ we assume that
    \begin{align}\label{cond:init:mean:inf}
        \meano{\phi_0}\in\textup{int}(D(f_1)), \quad \meang{\psi_0}\in\textup{int}(D(g_1)).
    \end{align}
    \end{subequations}
    The sextuplet $(\phi,\psi,\mu,\theta,\xi,\xi_\Ga)$ is called a weak solution of the system \eqref{EQ:SYSTEM} on $[0,T]$ if the following properties hold:
    \begin{enumerate}[label=\textnormal{(\roman*)}, ref=\thetheorem(\roman*), topsep=0ex]
        \item The functions $\phi, \psi, \mu, \theta, \xi$ and $\xi_\Ga$ have the regularities
        \begin{subequations}
            \begin{align}
                &\scp{\phi}{\psi} \in C([0,T];\mathcal{L}^2)\cap H^1(0,T;(\mathcal{H}_{L,\beta}^1)^\prime)\cap L^\infty(0,T;\mathcal{H}_{K,\alpha}^1), \label{REGPP:SING}\\
                &\scp{\mu}{\theta}\in L^2(0,T;\mathcal{H}_{L,\beta}^1) \label{REGMT:SING}, \\
                &\scp{\xi}{\xi_\Ga}\in L^2(0,T;\mathcal{L}^2) \label{REGLC:SING},
            \end{align}
        \end{subequations}
        and it holds $\phi(x,t) \in D(f_1)$ for almost all $(x,t)\in Q$ and $\psi(z,t) \in D(g_1)$ for almost all $(z,t) \in \Sigma$.
    \item \label{DEF:SING:WS:IC} The initial conditions are satisfied in the following sense:
    \begin{align}
        \phi\vert_{t=0} = \phi_0 \quad\text{a.e.~in } \Omega, \quad\text{and} \quad\psi\vert_{t=0} = \psi_0 \quad\text{a.e.~on }\Gamma.
    \end{align}
    \item \label{DEF:SING:WS:WF} The variational formulation
    \begin{subequations}\label{SING:WF}
        \begin{align}
        &\begin{aligned}
            &\ang{\scp{\delt\phi}{\delt\psi}}{\scp{\zeta}{\zeta_\Ga}}_{\mathcal{H}_{L,\beta}^1} - \intO \phi\boldsymbol{v}\cdot\Grad\zeta\dx - \intG \psi\boldsymbol{w}\cdot\Gradg\zeta_\Ga\dG 
            \\
            &= - \intO m_\Om(\phi)\Grad\mu\cdot\Grad\zeta\dx - \intG  m_\Ga(\psi)\Gradg\theta\cdot\Gradg\zeta_\Ga\dG \label{WF:PP:SING}\\
            &\quad - \sigma(L)\intG(\beta\theta-\mu)(\beta\zeta_\Ga - \zeta)\dG, 
        \end{aligned}
            \\
        &\begin{aligned}
            &\intO \mu\,\eta\dx + \intG\theta\,\eta_\Ga\dG
            \\
            &=  \intO\Grad\phi\cdot\Grad\eta + \xi\eta + F_2^\prime(\phi)\eta\dx
            + \intG \Gradg\psi\cdot\Gradg\eta_\Ga + \xi_\Ga\eta_\Ga 
            + G_2^\prime(\psi)\eta_\Ga\dG\label{WF:MT:SING}
            \\
            &\quad + \sigma(K)\intG(\alpha\psi-\phi)(\alpha\eta_\Ga - \eta) \dG, 
        \end{aligned}
        \end{align}
    \end{subequations}
    holds a.e. on $[0,T]$ for all $\scp{\zeta}{\zeta_\Ga}\in\mathcal{H}_{L,\beta}^1, \scp{\eta}{\eta_\Ga}\in\mathcal{H}_{K,\alpha}^1$, where
    \begin{align}
        \xi\in f_1(\phi) \quad\text{a.e.~in~} Q,\qquad\xi_\Ga\in g_1(\psi) \quad\text{a.e.~on~} \Sigma.
    \end{align}
    \item \label{DEF:SING:WS:MCL} The functions $\phi$ and $\psi$ satisfy the mass conservation law
    \begin{align}\label{MCL:SING}
        \begin{dcases}
            \beta\intO \phi(t)\dx + \intG \psi(t)\dG = \beta\intO \phi_0 \dx + \intG \psi_0\dG, &\textnormal{if } L\in[0,\infty), \\
            \intO\phi(t)\dx = \intO\phi_0\dx \quad\textnormal{and}\quad \intG\psi(t)\dG = \intG\psi_0\dG, &\textnormal{if } L = \infty
        \end{dcases}
    \end{align}
    for all $t\in[0,T]$.
    \item \label{DEF:SING:WS:WEDL} The energy inequality
    \begin{align}\label{WEDL:SING}
        \begin{split}
            &E_K(\phi(t),\psi(t)) + \int_0^t\intO m_\Om(\phi)\abs{\Grad\mu}^2\dxs + \int_0^t\intG m_\Ga(\psi)\abs{\Gradg\theta}^2\dGs \\
            &\quad + \sigma(L) \int_0^t\intG (\beta\theta-\mu)^2\dGs \\
            &\quad - \int_0^t\intO\phi\boldsymbol{v}\cdot\Grad\mu\dxs - \int_0^t\intG \psi\boldsymbol{w}\cdot\Gradg\theta\dGs 
            \\[1ex]
            &\leq E_K(\phi_0,\psi_0)
        \end{split}
    \end{align}
    holds for all $t\in[0,T]$.
    \end{enumerate}
\end{definition}

\medskip

\begin{remark}
    Assume that the potentials $F$ and $G$ are more regular, say $F\in C^1(D(f_1))$ and $G\in C^1(D(g_1))$. In addition, suppose that $f_1$ and $g_1$ are single-valued in their respective domains. Then, from the theory of subdifferentials, it is well known that $\xi = F_1^\prime(\phi)$ and $\xi_\Ga = G_1^\prime(\psi)$. Therefore, the weak formulation \eqref{SING:WF} can be rewritten as
    \begin{subequations}\label{REG:WF}
        \begin{align}
            &\begin{aligned}
            &\ang{\scp{\delt\phi}{\delt\psi}}{\scp{\zeta}{\zeta_\Ga}}_{\mathcal{H}_{L,\beta}^1} - \intO \phi\boldsymbol{v}\cdot\Grad\zeta\dx - \intG \psi\boldsymbol{w}\cdot\Gradg\zeta_\Ga\dG \\
            &\quad= - \intO m_\Om(\phi)\Grad\mu\cdot\Grad\zeta\dx - \intG  m_\Ga(\psi)\Gradg\theta\cdot\Gradg\zeta_\Ga\dG \label{REG:WF:PP}\\
            &\qquad - \sigma(L)\intG(\beta\theta-\mu)(\beta\zeta_\Ga - \zeta)\dG, 
            \end{aligned}
            \\
            &\begin{aligned}
            &\intO \mu\,\eta\dx + \intG\theta\,\eta_\Ga\dG 
            \\
            &\quad = \intO\Grad\phi\cdot\Grad\eta + F'(\phi)\eta \dx + \intG\Gradg\psi\cdot\Gradg\eta_\Ga + G'(\psi)\eta_\Ga \dG \label{REG:WF:MT}\\
            &\qquad + \sigma(K)\intG(\alpha\psi-\phi)(\alpha\eta_\Ga - \eta) \dG, 
            \end{aligned}
        \end{align}
    \end{subequations}
    a.e. on $[0,T]$ for all $\scp{\zeta}{\zeta_\Ga}\in\mathcal{H}_{L,\beta}^1, \scp{\eta}{\eta_\Ga}\in\mathcal{H}_{K,\alpha}^1$, which coincides with the weak formulation for regular potentials stated in \cite{Knopf2024}.
\end{remark}

\medskip

\begin{remark}
    The difference in the condition of the initial data \eqref{cond:init:mean:L} and \eqref{cond:init:mean:inf} in the cases $L\in[0,\infty)$ and $L=\infty$ reflects the difference in the mass conservation law \eqref{MCL:SING} depending on the value of $L\in[0,\infty]$.
\end{remark}

\medskip

We can now state our first main results regarding the existence of weak solutions for singular potentials.

\begin{theorem}\textnormal{(Existence of weak solutions of \eqref{EQ:SYSTEM})}\label{THEOREM:EOWS}
    Suppose that the assumptions \ref{ASSUMP:1}-\ref{ASSUMP:MOBILITY} and \ref{S1}-\ref{S2} hold. Let $K,L\in[0,\infty]$, let $\scp{\phi_0}{\psi_0}\in\mathcal{H}^1_{K,\alpha}$ be an arbitrary initial datum satisfying \eqref{cond:init}, and let $\boldsymbol{v}\in L^2(0,T;\mathbf{L}^3_\Div(\Om))$ and $\boldsymbol{w}\in L^2(0,T;\mathbf{L}_\tau^{2+\omega}(\Ga))$ for some $\omega > 0$ be given velocity fields. If $K\in [0,\infty)$, we further suppose that \ref{S3} holds, and if $K=0$, we additionally assume that the domain $\Om$ is of class $C^2$. 
    
    Then, the Cahn--Hilliard system \eqref{EQ:SYSTEM} admits at least one weak solution $(\phi,\psi,\mu,\theta,\xi,\xi_\Ga)$ in the sense of Definition~\ref{DEF:SING:WS}. In all cases, if the domain $\Om$ is at least of class $C^2$, it holds that
    \begin{align}
        \scp{\phi}{\psi}\in L^2(0,T;\mathcal{H}^2),
    \end{align}
    and the equations
    \begin{align*}
        &\mu = -\Lap\phi + \xi + F^\prime_2(\phi) &&\text{a.e.~in } Q, \\
        &\theta = -\Lapg\psi + \xi_\Ga + G^\prime_2(\psi) + \alpha\deln\phi &&\text{a.e.~on } \Sigma, \\
        & \begin{cases} 
            K\deln\phi = \alpha\psi - \phi &\text{if} \ K\in [0,\infty), \\
            \deln\phi = 0 &\text{if} \ K = \infty
        \end{cases} &&  \text{a.e.~on } \Sigma
    \end{align*}
    are fulfilled in the strong sense.
    If additionally \ref{S4} holds, we even have
    \begin{align}
        \label{REG:PP:W26}
        \scp{\phi}{\psi}\in L^2(0,T;\mathcal{W}^{2,6}),\quad
        \scp{F^\prime(\phi)}{G^\prime(\psi)}&\in L^2(0,T;\mathcal{L}^6).
    \end{align}
\end{theorem}

\medskip

Theorem~\ref{THEOREM:EOWS} is proved by approximating the singular potentials by means of a Moreau--Yosida approximation, which will be explained in more detail in Subsection~\ref{SUBSECT:YOS}. The proof of Theorem~\ref{THEOREM:EOWS} is then presented in Subsection~\ref{SUBSECT:CWS}.

\medskip

\begin{remark}\label{REM:REG:WEAK}
    Concerning the regularities stated in \eqref{REG:PP:W26}, we note the following.
    \begin{enumerate}[label=\textnormal{(\alph*)},topsep=0ex]
        \item We even get
        \begin{align*}
            G^\prime(\psi)\in L^2(0,T;L^s(\Ga))
            \qquad\text{~and~}\qquad \psi\in L^2(0,T;W^{2,s}(\Ga))
        \end{align*}
        for any $s\in[1,\infty)$, see Proposition~\ref{Prop:FG:Linf} and Corollary~\ref{Cor:pp:W2r}. 
        \item If the spatial dimension is two (i.e., $d=2$), we even get
        \begin{alignat*}{2}
            &F^\prime(\phi)\in L^2(0,T;L^r(\Om)), \qquad &&\phi\in L^2(0,T;W^{2,r}(\Om)), \\
            &G^\prime(\psi)\in L^2(0,T;L^r(\Ga)), \qquad &&\psi\in L^2(0,T;W^{2,r}(\Ga))
        \end{alignat*}
        for all $r\in [1,\infty)$, see Proposition~\ref{Prop:FG:Linf} and Corollary~\ref{Cor:pp:W2r}.
    \end{enumerate}
\end{remark}

\subsection{Uniqueness and continuous dependence on the initial data}

Our second result shows the continuous dependence of the phase-fields on the velocity fields and the initial data in the case of constant mobilities. 

\begin{theorem}\label{THEOREM:UNIQUE:SING}
    Suppose that the assumptions \ref{ASSUMP:1}-\ref{ASSUMP:MOBILITY} and \ref{S1}-\ref{S2} hold, and that the mobility functions $m_\Om$ and $m_\Ga$ are constant. Let $K, L\in[0,\infty]$, let $\scp{\phi_0^1}{\psi_0^1}$, $\scp{\phi_0^2}{\psi_0^2}\in\mathcal{H}^1_{K,\alpha}$ be two pairs of initial data, which satisfy
    \begin{align}\label{initial-data-mean-value}
        \begin{cases}
            \mean{\phi_0^1}{\psi_0^1} = \mean{\phi_0^2}{\psi_0^2} \quad&\text{if~} L\in[0,\infty), \\
            \meano{\phi_0^1} = \meano{\phi_0^2} \quad\text{and}\quad \meang{\psi_0^1} = \meang{\psi_0^2} &\text{if~} L=\infty,
        \end{cases}
    \end{align}
    as well as \eqref{cond:init}, and let $\boldsymbol{v}_1,\boldsymbol{v}_2\in L^2(0,T;\mathbf{L}^3_{\Div}(\Om))$, and $\boldsymbol{w}_1, \boldsymbol{w}_2 \in L^2(0,T;\mathbf{L}_\tau^{2+\omega}(\Ga))$ for some $\omega > 0$ be given bulk and surface velocity fields, respectively. Furthermore, suppose that $(\phi_1,\psi_1,\mu_1,\theta_1,\xi_1,\xi_{\Gamma 1})$ and $(\phi_2,\psi_2,\mu_2,\theta_2,\xi_2,\xi_{\Gamma 2})$ are weak solutions in the sense of Definition~\ref{DEF:SING:WS} corresponding to $(\phi_0^1, \psi_0^1, \boldsymbol{v}_1, \boldsymbol{w}_1)$ and $(\phi_0^2, \psi_0^2, \boldsymbol{v}_2, \boldsymbol{w}_2)$, respectively. Then, the continuous dependence estimate
    \begin{align}\label{EST:continuous-dependence:sing}
        \begin{split}
            &\bignorm{\bigscp{\phi_1(t) - \phi_2(t)}{\psi_1(t) - \psi_2(t)}}_{L,\beta,\ast}^2 + \int_0^t \bignorm{\bigscp{\phi_1(\tau) - \phi_2(\tau)}{\psi_1(\tau) - \psi_2(\tau)}}_{K,\alpha}^2 \dtau\\[0.4em]
            &\leq C\bignorm{\bigscp{\phi_0^1 - \phi_0^2}{\psi_0^1 - \psi_0^2}}_{L,\beta,\ast}^2 \exp\left(C\int_0^t \mathcal{F}(\tau)\dtau\right) \\
            &\quad + C\int_0^t \bignorm{\bigscp{\boldsymbol{v}_1(s) - \boldsymbol{v}_2(s)}{\boldsymbol{w}_1(s) - \boldsymbol{w}_2(s)}}_{\mathbf{L}^3(\Om)\times \mathbf{L}^{2+\omega}(\Ga)}^2 \exp\left(C\int_s^t \mathcal{F}(\tau)\dtau\right)\ds
        \end{split}
    \end{align}
    holds for almost all $t\in[0,T]$, where $\mathcal{F} \coloneqq \norm{\scp{\boldsymbol{v}_2}{\boldsymbol{w}_2}}_{\mathbf{L}^3(\Om)\times \mathbf{L}^{2+\omega}(\Ga)}^2,$
    and the constant $C>0$ depends only on $\Om$, the parameters of the system and the initial data.

    In particular, if $\scp{\phi_0^1}{\psi_0^1} = \scp{\phi_0^2}{\psi_0^2}$ a.e.~in $\Om\times\Ga$, $\boldsymbol{v}_1 = \boldsymbol{v}_2$ a.e.~in $Q$ and $\boldsymbol{w}_1 = \boldsymbol{w}_2$ a.e.~on $\Sigma$, \eqref{EST:continuous-dependence:sing} yields uniqueness of the phase-fields of the corresponding weak solution. Furthermore, the whole solution is unique if both the operators $f_1$ and $g_1$ are single-valued on their respective domains.
\end{theorem}

\medskip

The proof of Theorem~\ref{THEOREM:UNIQUE:SING} is presented in Section~\ref{SECT:UNIQUENESS}.

\subsection{Strong solutions for singular potentials}

Our next main result is concerned with establishing higher regularity properties of weak solutions. Assuming that the mobility functions are constant and that the potentials as well as the velocity fields are more regular, we prove the existence of strong solutions.

\begin{theorem}\label{thm:highreg}
    In addition to the assumptions \ref{ASSUMP:1}-\ref{ASSUMP:MOBILITY} and \ref{S1}-\ref{S4}, we assume that the domain $\Om$ is at least of class $C^2$ and that the mobility functions $m_\Om$ and $m_\Ga$ are constant. Let $K,L\in[0,\infty]$ and let $\scp{\phi_0}{\psi_0}\in\mathcal{H}^1_{K,\alpha}$ be an arbitrary initial datum satisfying \eqref{cond:init}. Moreover, we suppose that the velocity fields satisfy 
    \begin{align*}
        \boldsymbol{v} &\in H^1(0,T;\mathbf{L}^{6/5}(\Om))\cap L^2(0,T;\mathbf{L}_\Div^3(\Om))\cap L^\infty(0,T;\mathbf{L}^2(\Om)),
        \\
        \boldsymbol{w}&\in H^1(0,T;\mathbf{L}^{1+\omega}(\Ga))\cap L^2(0,T;\mathbf{L}_\Div^{2}(\Ga))\cap L^\infty(0,T;\mathbf{L}^2(\Ga))
    \end{align*}
    for some $\omega > 0$.    
    We further assume that the following compatibility condition holds:
    \begin{enumerate}[label=\textnormal{\bfseries(C)},topsep=0ex]
        \item \label{cond:MT:0} There exists $\scp{\mu_0}{\theta_0}\in\mathcal{H}^1_{L,\beta}$ such that for all $\scp{\eta}{\eta_\Ga}\in\mathcal{H}^1_{K,\alpha}$, it holds
        \begin{align*}
        \begin{aligned}
            &\intO\mu_0\eta\dx + \intG\theta_0\eta_\Ga\dG 
            \\
            &= \intO\Grad\phi_0\cdot\Grad\eta + F^\prime(\phi_0)\eta\dx + \intG\Gradg\psi_0\cdot\Gradg\eta_\Ga + G^\prime(\psi_0)\eta_\Ga\dG 
            \\
            &\quad + \sigma(K)\intG(\alpha\psi_0 - \phi_0)(\alpha\eta_\Ga - \eta)\dG.
        \end{aligned}
        \end{align*}
    \end{enumerate}
    Let $(\phi,\psi,\mu,\theta)$ be the corresponding unique solution in the sense of Definition~\ref{DEF:SING:WS}. Then, $(\phi,\psi,\mu,\theta)$ has the following regularity:%
    \begin{subequations}\label{REG:STRG}
    \begin{align}
        \scp{\delt\phi}{\delt\psi}&\in L^\infty(0,T;(\mathcal{H}^1_{L,\beta})^\prime)\cap L^2(0,T;\mathcal{H}^1), \\
        \label{REG:PP:W26inf}
        \scp{\phi}{\psi}&\in L^\infty(0,T;\mathcal{W}^{2,6})\cap\Big( C\big(\overline{Q}\big)\times C\big(\overline{\Sigma}\big)\Big), \\
        \label{REG:MT:H1inf}
        \scp{\mu}{\theta}&\in L^\infty(0,T;\mathcal{H}^1)\cap L^2(0,T;\mathcal{H}^2), \\
        \scp{F^\prime(\phi)}{G^\prime(\psi)}&\in L^2(0,T;\mathcal{L}^\infty)\cap L^\infty(0,T;\mathcal{L}^6).
    \end{align}
    \end{subequations}
    In particular, $(\phi, \psi, \mu, \theta)$ is a strong solution of system \eqref{EQ:SYSTEM}, i.e., all equations of \eqref{EQ:SYSTEM} are satisfied almost everywhere.
\end{theorem}

\medskip

The proof of Theorem~\ref{thm:highreg} can be found in Section~\ref{SECT:HIGHREG}.

\medskip

\begin{remark}\label{REM:REG:STRONG}
    Concerning the regularities stated in \eqref{REG:PP:W26inf}, we note the following.
    \begin{enumerate}[label=\textnormal{(\alph*)},topsep=0ex]
        \item We even get
        \begin{align*}
            G^\prime(\psi)\in L^\infty(0,T;L^s(\Ga))
            \qquad\text{~and~}\qquad 
            \psi\in L^\infty(0,T;W^{2,s}(\Ga))
        \end{align*}
        for any $s\in[1,\infty)$. This can be shown by invoking \eqref{REG:MT:H1inf} and proceeding as in the proofs of Proposition~\ref{Prop:FG:Linf} and Corollary~\ref{Cor:pp:W2r}. 
        \item If the spatial dimension is two (i.e., $d=2$), we even get
        \begin{alignat*}{2}
            &F^\prime(\phi)\in L^\infty(0,T;L^r(\Om)), \qquad &&\phi\in L^\infty(0,T;W^{2,r}(\Om)), \\
            &G^\prime(\psi)\in L^\infty(0,T;L^r(\Ga)), \qquad &&\psi\in L^\infty(0,T;W^{2,r}(\Ga))
        \end{alignat*}
        for all $r\in [1,\infty)$. In view of \eqref{REG:MT:H1inf}, this can also be obtained similarly to Proposition~\ref{Prop:FG:Linf} and Corollary~\ref{Cor:pp:W2r}.
    \end{enumerate}
\end{remark}

\medskip

\begin{remark}
    If the spatial dimension is two (i.e., $d=2$), the assumptions on the velocity fields $\boldsymbol{v}$ and $\boldsymbol{w}$ can be relaxed.
    Namely, it is enough to assume that 
    \begin{align*}
        \boldsymbol{v} &\in H^1(0,T;\mathbf{L}^{1 + \omega}(\Om))\cap L^2(0,T;\mathbf{L}_\Div^{2 + \omega}(\Om))\cap L^\infty(0,T;\mathbf{L}^2(\Om)),
        \\
        \boldsymbol{w} &\in H^1(0,T;\mathbf{L}^1(\Ga))\cap L^2(0,T;\mathbf{L}_\Div^{2}(\Ga))\cap L^\infty(0,T;\mathbf{L}^2(\Ga))
    \end{align*}
    for some $\omega > 0$.
    This is due to the better Sobolev embeddings in two dimensions, which allow us to perform the computations presented in the proof of Theorem~\ref{thm:highreg} with less regular velocity fields.
\end{remark}

\medskip

\begin{remark} 
    Concerning the compatibility assumption \ref{cond:MT:0}, which was made in Theorem~\ref{thm:highreg} and Theorem~\ref{thm:sepprop}, we note the following.
    \begin{enumerate}[label=\textnormal{(\alph*)},topsep=0ex]
    \item Even without assuming \ref{cond:MT:0},
    it would still be possible to obtain the regularities \eqref{REG:STRG} stated in Theorem~\ref{thm:highreg}, but with the time interval $(0,T)$ replaced by $(t_0,T)$ for any time $t_0 > 0$. This is because solutions of system \eqref{CH} instantly regularize for positive times due to parabolic smoothing.     
    \item If $L\in (0,\infty]$, it is quite easy to find initial data $\scp{\phi_0}{\psi_0}$, which satisfy the compatibility condition \ref{cond:MT:0}.

    Let $(\phi_0,\psi_0)\in \mathcal{H}^3$ fulfill the conditions \eqref{cond:init}, and we additionally assume 
    \begin{equation}
        \label{cond:FG}
        \big( F'(\phi_0) , G'(\psi_0) \big) \in \mathcal{H}^1
    \end{equation}
    and
    \begin{align}
        \begin{cases}
            K\deln\phi_0 = \alpha\psi_0 - \phi_0 &\text{if~} K\in[0,\infty), \\
            \deln\phi_0 = 0 &\text{if~} K = \infty
        \end{cases}
        \qquad\text{on~}\Ga.
    \end{align}  
    We then define $\scp{\mu_0}{\theta_0}\in\mathcal{H}^1 = \mathcal{H}^1_{L,\beta}$ as
    \begin{alignat*}{2}
        \mu_0 &\coloneqq -\Lap\phi_0 + F^\prime(\phi_0) &&\qquad\text{in~}\Om, \\
        \theta_0 &\coloneqq - \Lapg\psi_0 + G^\prime(\psi_0) + \alpha\deln\phi_0 &&\qquad\text{on~}\Gamma.
    \end{alignat*}
    Now, via integration by parts, it is easy to see that \ref{cond:MT:0} is fulfilled.

    Of course, condition \eqref{cond:FG} (and thus \eqref{cond:init:int}) is clearly satisfied if the initial data are strictly separated from $\pm 1$. This means that there exists $\delta_0>0$ such that for all $x\in\Omega$ and $z\in\Gamma$, it holds
    \begin{equation*}
        -1+\delta_0 \le \phi_0(x)\le 1-\delta_0
        \quad\text{and}\quad
        -1+\delta_0 \le \psi_0(z)\le 1-\delta_0.
    \end{equation*}
    \item In the case $L=0$, the situation is more involved and it is not so easy to see that admissible, non-trivial initial data satisfying \ref{cond:MT:0} actually exist. In particular, proceeding as in (b) is not possible as this construction would not ensure that the resulting chemical potentials $\mu_0$ and $\theta_0$ actually fulfill the Dirichlet type coupling condition $\mu_0 = \beta \theta_0$ on $\Ga$ that is demanded in \ref{cond:MT:0}.
    
    \item In all cases $(K,L) \in [0,\infty]^2$, choosing $(\phi_0,\psi_0) = (0,0)$ as well as $(\mu_0,\theta_0) = (0,0)$ provides admissible initial data satisfying \ref{cond:MT:0} as well as \eqref{cond:init} since 
    $f_1(0) = 0$ and $g_1(0) = 0$ was assumed.
    \end{enumerate}
\end{remark}

\subsection{Separation properties for logarithmic type potentials}

Our last theorem is concerned with establishing separation properties for logarithmic type potentials. To be precise, we show that the phase-fields $\phi$ and $\psi$ remain separated from the values $\pm 1$, which correspond to the pure phases.

\begin{theorem}\label{thm:sepprop}
    Suppose that the assumptions of Theorem~\ref{thm:highreg} hold. Let $(\phi, \psi, \mu, \theta)$ be the corresponding unique strong solution, which exists according to Theorem~\ref{thm:highreg}.
    \begin{enumerate}[label=\textnormal{(\alph*)},topsep=0ex]
        \item \label{thm:sepprop:a} The phase-fields $\phi$ and $\psi$ remain strictly separated almost everywhere in time, that is:
        \begin{align}
            \label{sepprop:phi:3d}
            \begin{aligned}
            &\text{For almost all $t\in[0,T]$, there exists $\delta(t)\in(0,1]$ such that:}\\
            &
            \norm{\phi(t)}_{L^\infty(\Omega)} \leq 1 - \delta(t)
            \quad\text{and}\quad
            \norm{\psi(t)}_{L^\infty(\Gamma)} \leq 1 - \delta(t).
            \end{aligned}
        \end{align}
        \item \label{thm:sepprop:b}If additionally \ref{S6} holds, the phase-field $\psi$ fulfills the uniform strict separation property:
        \begin{align}
            \label{sepprop:psi}
            \text{There exists $\delta_\ast\in(0,1]$ such that for all $(z,t)\in\Sigma$:}\quad
            \abs{\psi(z,t)}\leq 1 - \delta_\ast.
        \end{align}
        \item \label{thm:sepprop:c} If additionally $d=2$ and \ref{S5} holds, the phase-field $\phi$ fulfills the uniform strict separation property:
        \begin{align}
        \label{sepprop:phi:2d}
        \text{There exists $\delta_\star\in(0,1]$ such that for all $(x,t) \in Q$:}\quad
        \abs{\phi(x,t)} \leq 1 - \delta_\star.
    \end{align}
    \end{enumerate}
\end{theorem}

\medskip

The proof of Theorem~\ref{thm:sepprop} is presented in Section~\ref{SECT:SEPARATION}.

In the remainder of this paper, the letter $C$ will denote generic positive constants that may depend on $\Omega$, the quantities introduced in \ref{ASSUMP:1}-\ref{ASSUMP:MOBILITY}, the potentials $F$ and $G$ and the corresponding constants introduced in \ref{S1}-\ref{S6}, the prescribed initial data, and the prescribed $(\phi_0,\psi_0)$ and velocity fields $(\boldsymbol{v},\boldsymbol{w})$.

\section{Existence of weak solutions}
\label{SECT:EOWS}
\subsection{Maximal monotone graphs and the Yosida approximation}
\label{SUBSECT:YOS}
In this subsection, we briefly review the basic concepts regarding the Yosida approximation of maximal monotone operators. For simplicity, we will restrict ourselves to the case which is relevant to us, i.e., we only consider $F_1,G_1:\R\rightarrow [0,\infty]$ as in \ref{S1}. For a more comprehensive overview related to maximal monotone operators and the Yosida approximation, we refer the reader to \cite[Chapter~2]{Brezis}, \cite[Chapter~2]{Barbu2010}, \cite[Chapter~IV]{Showalter} and \cite[Section~5]{Garcke2017}. \\[0.3em]
 For each $\ep \in(0,1)$, we define the \textit{Yosida approximations} $f_{1,\ep},g_{1,\ep}:\R\rightarrow\R$ of the subdifferentials $f_1 = \partial F_1$ and $g_1 = \partial G_1$, respectively,  as
\begin{align}\label{yosida}
    f_{1,\ep}(r) \coloneqq \frac{1}{\ep}\Big( r - (I + \ep f_1)^{-1}(r)\Big), \quad g_{1,\ep}(r) \coloneqq \frac{1}{\ep}\Big( r - (I + \ep g_1)^{-1}(r)\Big)
\end{align}
for all $r\in\R$, respectively, where $I$ denotes the identity on $\R$. Note that while $f_1$ and $g_1$ might be multi-valued, their Yosida approximations are single-valued. It is well known that for all $r\in\R$, it holds
\begin{align}\label{jonas-1}
    \abs{f_{1,\ep}(r)} \leq \abs{f_1^\circ(r)}\ \text{for all~} \ep\in(0,1), \quad \text{and~} f_{1,\ep}(r)\rightarrow f_1^\circ(r) \ \text{as~} \ep\rightarrow 0.
\end{align}
The analogous properties hold for $g_{1,\ep}$. Furthermore, both $f_{1,\ep}$ and $g_{1,\ep}$ are increasing and Lipschitz continuous with Lipschitz constant $\frac{1}{\ep}$. In particular, since $0\in f_1(0)\cap g_1(0)$,  this implies in combination with \eqref{jonas-1} that
\begin{align}\label{jonas-2}
    \abs{f_{1,\ep}(r)} \leq \frac{1}{\ep}\abs{r}, \quad \abs{g_{1,\ep}(r)} \leq \frac{1}{\ep}\abs{r}
\end{align}
for all $r\in\R$. Now, we define for each $\ep\in(0,1)$ the \textit{Moreau--Yosida approximations} $F_{1,\ep}$ and $G_{1,\ep}$ of $F_1$ and $G_1$, respectively, as 
\begin{align}\label{moreau-yosida}
    F_{1,\ep}(r) \coloneqq \inf_{s\in\R}\left\{\frac{1}{2\ep}\abs{r-s}^2 + F_1(s)\right\}, \quad G_{1,\ep}(r) \coloneqq \inf_{s\in\R}\left\{\frac{1}{2\ep}\abs{r-s}^2 + G_1(s)\right\}
\end{align}
for all $r\in\R$, respectively. It is well known that $F_{1,\ep}$ is convex and differentiable with $F_{1,\ep}^\prime = f_{1,\ep}$ for all $\ep\in(0,1)$ (see, e.g., \cite[Theorem~2.9]{Barbu2010}).
Moreover, it holds that
\begin{align}\label{est:Fep:F}
    0 \leq F_{1,\ep}(r) \leq F_1(r) \ \text{for all~} \ep\in(0,1), \ \text{and~} F_{1,\ep}(r)\nearrow F_1(r) \ \text{as~}\ep\rightarrow 0
\end{align}
for all $r\in\R$. The analogous properties hold for $G_{1,\ep}$. In particular, we may write
\begin{align*}
    F_{1,\ep}(r) = \int_0^r f_{1,\ep}(s)\ds, \quad G_{1,\ep}(r) = \int_0^r g_{1,\ep}(s)\ds
\end{align*}
for all $r\in\R$. In addition, we readily infer from the definition \eqref{moreau-yosida} and the fact that $F_1(0) = G_1(0) = 0$ that
\begin{align}\label{jonas-3}
    \abs{F_{1,\ep}(r)} \leq \frac{1}{2\ep} \abs{r}^2, \quad \abs{G_{1,\ep}(r)} \leq \frac{1}{2\ep} \abs{r}^2
\end{align}
for all $r\in\R$.
Now, the domination condition \eqref{domination} implies that 
\begin{align}\label{domination-reg}
    \abs{f_{1,\ep}(\alpha r)} \leq \kappa_1\abs{g_{1,\ep}(r)} + \kappa_2
\end{align}
for all $r\in\R$ and all $\ep\in(0,1)$, where the constants $\kappa_1, \kappa_2$ are the same as in \eqref{domination} (see \cite[Lemma 4.4]{Calatroni2013}). In summary, we conclude that the approximate potentials $F_\ep \coloneqq F_{1,\ep} + F_2$ and $G_\ep \coloneqq G_{1,\ep} + G_2$ fulfill the following growth conditions:
\begin{enumerate}[label=\textnormal{\bfseries (R\arabic*)}]
    \item \label{R1} There exist constants $c_{F_\ep}, c_{G_\ep}, c_{F_\ep^\prime}, c_{G_\ep^\prime} \geq 0$, which may depend on $\ep$, such that the first-order derivatives of $F_\ep$ and $G_\ep$ satisfy the growth conditions
    \begin{align}
        \abs{F_\ep^\prime(s)} &\leq c_{F_\ep^\prime}(1 + \abs{s}), \label{GC:F'ep}\\
        \abs{G_\ep^\prime(s)} &\leq c_{G_\ep^\prime}(1 + \abs{s}), \label{GC:G'ep} \\
        \abs{F_\ep(s)} &\leq c_{F_\ep}(1 + \abs{s}^2), \label{GC:Fep}\\
        \abs{G_\ep(s)} &\leq c_{G_\ep}(1 + \abs{s}^2) \label{GC:Gep}
    \end{align}
    for all $s\in\R$. 
\end{enumerate}
Furthermore, the growth conditions in \ref{S1} entail (see, e.g., \cite[Section~3]{Colli2024}):
\begin{enumerate}[label=\textnormal{\bfseries (R\arabic*)},start=2]
    \item\label{R2} There exists $\ep_\star \in(0,1)$ as well as constants $C_\Omega,C_\Gamma\ge 0$ such that
    \begin{align}
         \label{grow-reg-f}
        F_\ep(s) &\geq \abs{s}^2 - C_\Omega, \\
         \label{grow-reg-g}
        G_\ep(s) &\geq \abs{s}^2 - C_\Gamma
    \end{align}
    for all $s\in\R$ and $\ep\in(0,\ep_\star)$.
\end{enumerate}
In particular, this implies that for $\ep\in (0,\ep_\star)$, $F_\ep$ and $G_\ep$ are bounded from below by $-C_\Omega$ and $-C_\Gamma$, respectively. Hence, without loss of generality, we may assume that $F_\ep$ and $G_\ep$ are non-negative provided that $\ep\in (0,\ep_\star)$. (Otherwise, we simply replace $F_2$ by $F_2 + C_\Omega$ and $G_2$ by $G_2+C_\Gamma$.) 

\subsection{Construction of a weak solution}
\label{SUBSECT:CWS}
\begin{proof}[Proof of Theorem~\ref{THEOREM:EOWS}]
\textbf{Step 1.} We first construct a sequence of approximate solutions.
Therefore, let $\ep\in(0,\ep_\star)$ with $\ep_\star$ as in \ref{R2}. 
Then, due to \cite[Theorem 3.2]{Knopf2024}, there exists a weak solution $(\phi_\ep, \psi_\ep, \mu_\ep, \theta_\ep)$ satisfying the weak formulation \eqref{REG:WF:PP}-\eqref{REG:WF:MT} with $F = F_\ep$ and $G = G_\ep$. It is worth mentioning that in \cite[Theorem 3.2]{Knopf2024}, the stronger regularity $\boldsymbol{w}\in L^2(0,T;\mathbf{L}^3_\tau(\Ga))$ was imposed for convenience. However, it is easy to see that by slightly modifying the proof, this result actually holds true for the regularity $\boldsymbol{w}\in L^2(0,T;\mathbf{L}_\tau^{2+\omega}(\Ga))$ that is prescribed in the present paper. As a further consequence of \cite[Theorem 3.2]{Knopf2024}, the energy inequality
\begin{align}\label{EI:ep}
    \begin{split}
        &E^\ep_K(\phi_\ep(t),\psi_\ep(t)) + \int_0^t\intO m_\Om(\phi_\ep)\abs{\Grad\mu_\ep}^2\dxs  \\
        &\quad + \int_0^t\intG m_\Ga(\psi_\ep)\abs{\Gradg\theta_\ep}^2\dGs
        + \sigma(L) \int_0^t\intG (\beta\theta_\ep-\mu_\ep)^2\dGs \\
        &\quad - \int_0^t\intO\phi_\ep\boldsymbol{v}\cdot\Grad\mu_\ep\dxs - \int_0^t\intG \psi_\ep\boldsymbol{w}\cdot\Gradg\theta_\ep\dGs 
        \\[1ex]
        &\leq E_K^\ep(\phi_0,\psi_0)
    \end{split}
\end{align}
holds for all $t\in[0,T]$. Here, $E_K^\ep$ denotes the energy given by \eqref{energy:K} with $F_\ep$ and $G_\ep$ in place of $F$ and $G$, respectively. Additionally, as also stated in \cite[Theorem 3.2]{Knopf2024}, if $\Om$ is of class $C^2$, then $\scp{\phi_\ep}{\psi_\ep}\in L^2(0,T;\mathcal{H}^2)$, and the equations
\begin{subequations}\label{STRG:EP}
    \begin{align}
        \label{STRG:MU:EP}
        &\mu_\ep = -\Lap\phi_\ep + F_\ep'(\phi_\ep) &&\text{a.e.~in } Q, \\
        \label{STRG:THETA:EP}
        &\theta_\ep = -\Lapg\psi_\ep + G_\ep'(\psi_\ep) + \alpha\deln\phi_\ep &&\text{a.e.~on } \Sigma, \\
        \label{STRG:PHIPSI:EP}
        & \begin{cases} 
            K\deln\phi_\ep = \alpha\psi_\ep - \phi_\ep &\text{if} \ K\in [0,\infty), \\
            \deln\phi_\ep = 0 &\text{if} \ K = \infty
        \end{cases} &&  \text{a.e.~on } \Sigma
    \end{align}
\end{subequations}
are fulfilled in the strong sense.

\textbf{Step 2.} We next establish uniform bounds on the approximate solution following from \eqref{EI:ep}. 
By the definition of $E_K^\ep$, we infer from \eqref{EI:ep} that
\begin{align}\label{jonas1}
    \begin{split}
        &\frac{1}{2}\norm{\Grad\phi_\ep(t)}_{L^2(\Om)}^2 + \intO F_\ep(\phi_\ep(t))\dx + \frac{1}{2}\norm{\Gradg\psi_\ep(t)}_{L^2(\Ga)}^2 + \intG G_\ep(\psi_\ep(t))\dG \\
        &\qquad +\sigma(K)\intG \frac12\abs{\alpha\psi_\ep(t)-\phi_\ep(t)}^2\dG + \int_0^t\intO m_\Om(\phi_\ep)\abs{\Grad\mu_\ep}^2\dxs \\
        &\qquad + \int_0^t\intG m_\Ga(\psi_\ep)\abs{\Gradg\theta_\ep}^2\dGs + \sigma(L)\int_0^t\intG (\beta\theta_\ep-\mu_\ep)^2\dGs \\
        &\quad \leq E_K^\ep(\phi_0,\psi_0) + \int_0^t\intO\phi_\ep\boldsymbol{v}\cdot\Grad\mu_\ep\dxs + \int_0^t\intG \psi_\ep\boldsymbol{w}\cdot\Gradg\theta_\ep\dGs
    \end{split}
\end{align}
holds for all $t\in[0,T]$. Now, observe that due to \eqref{cond:init:int} and \eqref{est:Fep:F}, we find
\begin{align*}
    \intO F_\ep(\phi_0)\dx + \intG G_\ep(\psi_0)\dG \leq C,
\end{align*}
and thus, $E_K^\ep(\phi_0,\psi_0) \leq C$. On the other hand, for the convective terms on the right-hand side of \eqref{jonas1}, we use Hölder's and Young's inequality to obtain the estimate
\begin{align}\label{jonas2}
    \begin{split}
        &\int_0^t\intO\phi_\ep\boldsymbol{v}\cdot\Grad\mu_\ep\dxs + \int_0^t\intG \psi_\ep\boldsymbol{w}\cdot\Gradg\theta_\ep\dGs \\
        &\leq \frac{m_\Om^\ast}{2}\int_0^t \intO \abs{\Grad\mu_\ep}^2\dxs + \frac{m_\Ga^\ast}{2}\int_0^t\intG \abs{\Gradg\theta_\ep}^2 \dGs \\
        &\quad + C\int_0^t\Big(\norm{\boldsymbol{v}}_{\mathbf{L}^3(\Om)}^2 + \norm{\boldsymbol{w}}^2_{\mathbf{L}^{2+\omega}(\Ga)}\Big)\norm{\scp{\phi_\ep}{\psi_\ep}}_{\mathcal{H}^1}^2\ds
    \end{split}
\end{align}
for almost all $t\in[0,T]$. In view of \ref{R2}, it holds that
\begin{align}\label{jonas3}
    \begin{split}
        \intO F_\ep(\phi_\ep(t)) \dx + \intG G_\ep(\psi_\ep(t))\dG &\geq \intO \abs{\phi_\ep(t)}^2\dx + \intG \abs{\psi_\ep(t)}^2\dG - C \\
        &= \norm{\phi_\ep(t)}^2_{L^2(\Om)} + \norm{\psi_\ep(t)}^2_{L^2(\Ga)} - C
    \end{split}
\end{align}
for all $t\in[0,T]$ since $\ep\in(0,\ep_\star)$.
This implies that
\begin{align}\label{jonas4}
    \begin{split}
        \frac12\norm{\scp{\phi_\ep(t)}{\psi_\ep(t)}}_{\mathcal{H}^1}^2 
        &\leq C + \frac{1}{2}\norm{\scp{\Grad\phi_\ep(t)}{\Gradg\psi_\ep(t)}}^2_{\mathcal{L}^2} 
        \\
        &\quad + \intO F_\ep(\phi_\ep(t))\dx + \intG G_\ep(\psi_\ep(t))\dG 
    \end{split}
\end{align}
for almost all $t\in[0,T]$. Hence, combining \eqref{jonas1}, \eqref{jonas2} and \eqref{jonas4}, and recalling that $F_\ep$ and $G_\ep$ are also non-negative, an application of Gronwall's lemma readily yields that
\begin{align}
    &\norm{\scp{\phi_\ep}{\psi_\ep}}_{L^\infty(0,T;\mathcal{H}^1)} + \norm{\scp{F_\ep(\phi_\ep)}{G_\ep(\psi_\ep)}}_{L^\infty(0,T;\mathcal{L}^1)} + \sigma(K)\norm{\alpha\psi_\ep - \phi_\ep}_{L^\infty(0,T;L^2(\Ga))} \nonumber\\
    &\quad + \norm{\scp{\Grad\mu_\ep}{\Gradg\theta_\ep}}_{L^2(0,T;\mathcal{L}^2)} + \sigma(L)\norm{\beta\theta_\ep - \mu_\ep}_{L^2(0,T;L^2(\Ga))} \leq C. \label{est:energy}
\end{align}

\textbf{Step 3.} We derive a uniform estimate on the time derivatives of the phase-fields $\phi_\ep$ and $\psi_\ep$. \\
To this end, we test \eqref{WF:PP:SING} with an arbitrary test functions $\scp{\zeta}{\zeta_\Ga}\in\mathcal{H}^1_{L,\beta}$. Then, recalling that the mobility functions $m_\Om$ and $m_\Ga$ are bounded (see \ref{ASSUMP:MOBILITY}), we use Hölder's inequality in combination with the continuous embeddings $H^1(\Om)\emb L^6(\Om)$ and $H^1(\Ga)\emb L^r(\Ga)$ for all $r\in[1,\infty)$ to infer
\begin{align*}
    \begin{split}
        &\abs{\ang{\scp{\delt\phi_\ep}{\delt\psi_\ep}}{\scp{\zeta}{\zeta_\Ga}}_{\mathcal{H}^1_{L,\beta}}} \\
        &\leq \norm{\phi_\ep}_{H^1(\Om)}\norm{\boldsymbol{v}}_{\mathbf{L}^3(\Om)}\norm{\zeta}_{H^1(\Om)} + C \norm{\Grad\mu_\ep}_{\mathbf{L}^2(\Om)}\norm{\zeta}_{H^1(\Om)} \\
        &\quad + \norm{\psi_\ep}_{H^1(\Ga)}\norm{\boldsymbol{w}}_{\mathbf{L}^{2+\omega}(\Ga)} \norm{\zeta_\Ga}_{H^1(\Ga)} + C\norm{\Gradg\theta_\ep}_{\mathbf{L}^2(\Ga)}\norm{\zeta_\Ga}_{H^1(\Ga)} \\
        &\quad + \sigma(L)\norm{\beta\theta_\ep - \mu_\ep}_{L^2(\Ga)}\norm{\beta\zeta_\Ga - \zeta}_{L^2(\Ga)}
    \end{split}
\end{align*}
a.e.~on $[0,T]$. Thus, after taking the supremum over all $\scp{\zeta}{\zeta_\Ga}\in\mathcal{H}^1_{L,\beta}$ with $\norm{\scp{\zeta}{\zeta_\Ga}}_{\mathcal{H}^1} \leq 1$, we take the square of the resulting inequality and integrate in time from $0$ to $T$. This yields
\begin{align*}
    \begin{split}
        &\norm{\scp{\delt\phi_\ep}{\delt\psi_\ep}}_{L^2(0,T;(\mathcal{H}^1_{L,\beta})^\prime)}^2 \\
        &\leq C\norm{\phi_\ep}^2_{L^\infty(0,T;H^1(\Om))}\norm{\boldsymbol{v}}^2_{L^2(0,T;\mathbf{L}^3(\Om))} + C \norm{\Grad\mu_\ep}^2_{L^2(0,T;\mathbf{L}^2(\Om))} \\
        &\quad + C\norm{\psi_\ep}^2_{L^\infty(0,T;H^1(\Ga))}\norm{\boldsymbol{w}}^2_{L^2(0,T;\mathbf{L}^{2+\omega}(\Ga))} + C\norm{\Gradg\theta_\ep}^2_{L^2(0,T;\mathbf{L}^2(\Ga))} \\
        &\quad + C\sigma(L)^2\norm{\beta\theta_\ep - \mu_\ep}^2_{L^2(0,T;L^2(\Ga))}.
    \end{split}
\end{align*}
In view of the uniform estimate \eqref{est:energy} and the regularity of the velocity fields $\boldsymbol{v}$ and $\boldsymbol{w}$, we deduce that
\begin{align}\label{est:delt}
    \norm{\scp{\delt\phi_\ep}{\delt\psi_\ep}}_{L^2(0,T;(\mathcal{H}^1_{L,\beta})^\prime)} \leq C.
\end{align}
\textbf{Step 4.} We next establish $L^1$-bounds on $f_{1,\ep}(\phi_\ep)$ and $g_{1,\ep}(\psi_\ep)$. 
Therefore, we define
\begin{align*}
    m_P \coloneqq \text{mean}(\phi_0,\psi_0), \quad\text{and}\quad m_C \coloneqq \text{mean}(\mu_\ep,\theta_\ep).
\end{align*}
In the following, we will only consider the case $L\in[0,\infty)$. The proof in the case $L=\infty$ works similarly, but one uses the means $\meano{\phi_0}$ and $\meang{\psi_0}$ as well as $\meano{\mu_\ep}$ and $\meang{\theta_\ep}$ instead of the ones defined above. This distinction is related to the mass conservation law \eqref{MCL:SING} and the conditions on the initial data \eqref{cond:init:mean:L} and \eqref{cond:init:mean:inf}. We refer to \cite{Colli2024} for more details if $K\in[0,\infty)$ and $L=\infty$. The case $K = L = \infty$ is similar, but easier. 

To handle the cases $K = 0$ and $K\in(0,\infty]$ simultaneously, we introduce the the notation
\begin{align*}
    \gamma(K) \coloneqq \begin{cases}
        0, &\text{if~} K\in(0,\infty], \\
        1, &\text{if~} K = 0.
    \end{cases}
\end{align*}
We now test \eqref{WF:MT:SING} with
\begin{align*}
    \scp{\eta}{\eta_\Ga} \coloneqq\begin{cases}
        \scp{\phi_\ep - \beta m_P}{\psi_\ep - m_P}, \quad&\text{if~} K\in(0,\infty], \\
        \scp{\phi_\ep - \alpha\beta m_P}{\psi_\ep - \beta m_P}, &\text{if~} K = 0,
    \end{cases}
\end{align*}
which clearly belongs to $\mathcal{H}^1_{K,\alpha}$. After some rearrangements, as well as adding and subtracting $\beta m_P$ and $m_P$ multiple times in the case $K = 0$, we obtain
\begin{align}
    \begin{split}
        &\norm{\Grad\phi_\ep}^2_{L^2(\Om)} + \intO f_{1,\ep}(\phi_\ep)(\phi_\ep - \beta m_P)\dx 
        \\
        &\qquad
        + \norm{\Gradg\psi_\ep}^2_{L^2(\Ga)} + \intG g_{1,\ep}(\psi_\ep)(\psi_\ep - m_P)\dG \\
        &\quad = \intO (\mu_\ep - \beta m_C)(\phi_\ep - \beta m_P)\dx + \intG (\theta_\ep - m_C)(\psi_\ep - m_P) \dG \label{jonas6}\\
        &\qquad - \sigma(K)\intG (\alpha\psi_\ep - \phi_\ep)\big(\alpha\psi_\ep - \phi_\ep - (\alpha\beta m_P - m_P)\big)\dG  \\
        &\qquad - \intO f_2(\phi_\ep)(\phi_\ep - \beta m_P) \dx - \intG g_2(\psi_\ep)(\psi_\ep - m_P) \dG  
        \\
        &\qquad + \gamma(K) \Bigg( \intO \mu_\ep(\beta m_P - \alpha\beta m_P) \dx + \intG \theta_\ep(m_P - \beta m_P)\dG \\
        &\qquad\qquad- \intO F_\ep^\prime(\phi_\ep)(\beta m_P - \alpha\beta m_P)\dx - \intG G^\prime_\ep(\psi_\ep)(m_P - \beta m_P)\dG \Bigg).
    \end{split}
\end{align}
Note that the subtracted mean values $\beta m_C$ and $m_C$ in the first two summands on the right-hand side of \eqref{jonas6} do not change the values of the respective integrals, since
\begin{align*}
    \beta\abs{\Om}\meano{\phi_\ep - \beta m_P} + \abs{\Ga}\meang{\psi_\ep - m_P} = 0
\end{align*}
due to the mass conservation law \eqref{MCL:SING}. 

For the left-hand side of \eqref{jonas6}, we can use the Miranville--Zelik inequality (see \cite[Appendix A.1]{Miranville2004} or \cite[p.~908]{Gilardi2009}) due to the assumption that $\beta m_P$ and $m_P$ lie in the interior of the effective domains $D(f_1)$ and $D(g_1)$, respectively, see \eqref{cond:init:mean:L}. Thus, there exist positive constants $c_1,c_2$ and a non-negative constant $c_3$ such that
\begin{align}\label{est:mz}
    \begin{split}
        &c_1 \norm{f_{1,\ep}(\phi_\ep)}_{L^1(\Om)} + c_2\norm{g_{1,\ep}(\psi_\ep)}_{L^1(\Ga)} - c_3 \\
        &\quad \leq \intO f_{1,\ep}(\phi_\ep)(\phi_\ep - \beta m_P)\dx + \intG g_{1,\ep}(\psi_\ep)(\psi_\ep - m_P)\dG.
    \end{split}
\end{align}
On the other hand, we use the bulk-surface Poincar\'{e} inequality stated in \ref{PRELIM:POINCINEQ} for the first two terms on the right-hand side of \eqref{jonas6} to find that
\begin{align}\label{jonas7}
    \begin{split}
        &\intO (\mu_\ep - \beta m_C)(\phi_\ep - \beta m_P)\dx + \intG (\theta_\ep - m_C)(\psi_\ep - m_P) \dG \\
        &\quad\leq C\big(1 + \norm{\scp{\phi_\ep}{\psi_\ep}}_{\mathcal{L}^2}\big)\norm{\scp{\mu_\ep}{\theta_\ep}}_{L,\beta}.
    \end{split}
\end{align}
Furthermore, in view of the Lipschitz continuity of $f_2$ and $g_2$, we infer that
\begin{align*}
    \begin{split}
        &-\intO f_2(\phi_\ep)(\phi_\ep - \beta m_P)\dx - \intG g_2(\psi_\ep)(\psi_\ep - m_P)\dG \\
        &\quad =  - \intO\big(f_2(\phi_\ep) - f_2(\beta m_P)\big)(\phi_\ep - \beta m_P)\dx \\
        &\qquad - \intG \big(g_2(\psi_\ep) - g_2(m_P)\big)(\psi_\ep - m_P)\dG \\
        &\qquad - \intO f_2(\beta m_P)(\phi_\ep - \beta m_P)\dx - \intG g_2(m_P)(\psi_\ep - m_P)\dG \\
        &\quad\leq C + C\intO \abs{\phi_\ep - \beta m_P}^2\dx + C\intG \abs{\psi_\ep - m_P}^2\dG \\
        &\quad\leq C\Big(1 + \norm{\phi_\ep}^2_{L^2(\Om)} + \norm{\psi_\ep}^2_{L^2(\Ga)}\Big),
    \end{split}
\end{align*}
and thus
\begin{align}\label{jonas8}
    \begin{split}
        &-\sigma(K)\intG (\alpha\psi_\ep - \phi_\ep)\big(\alpha\psi_\ep - \phi_\ep - (\alpha\beta m_P - m_P)\big)\dG - \intO f_2(\phi_\ep)(\phi_\ep - \beta m_P)\dx \\
        &\quad - \intG g_2(\psi_\ep)(\psi_\ep - m_P)\dG \leq C\Big(1 + \norm{\phi_\ep}^2_{H^1(\Om)} + \norm{\psi_\ep}^2_{L^2(\Ga)}\Big).
    \end{split}
\end{align}
For the remaining terms on the right-hand side of \eqref{jonas6}, which are only present if $K=0$, we make use of the additional boundary regularity. In this case, we know that $\scp{\phi_\ep}{\psi_\ep}\in L^2(0,T;\mathcal{H}^2)$ and that the equations \eqref{STRG:EP} are satisfied. Hence, after integration by parts, we deduce that
\begin{align}\label{jonas9}
    \begin{split}
        &\intO (F_\ep^\prime(\phi_\ep) - \mu_\ep)(\alpha\beta m_P - \beta m_P)\dx + \intG (G_\ep^\prime(\psi_\ep) - \theta_\ep)(\beta m_P - m_P)\dG \\
        &\quad =  \intO \Lap\phi_\ep(\alpha\beta m_P - \beta m_P)\dx + \intG (\Lapg\psi_\ep - \alpha\deln\phi_\ep)(\beta m_P - m_P)\dG \\
        &\quad = \intG \deln\phi_\ep(\alpha m_P - \beta m_P)\dG \\
        &\quad\leq C\norm{\deln\phi_\ep}_{L^2(\Ga)}.
    \end{split}
\end{align}
Collecting \eqref{jonas6}-\eqref{jonas9}, we find that
\begin{align}\label{jonas10}
    \norm{f_{1,\ep}(\phi_\ep(t))}_{L^1(\Om)} + \norm{g_{1,\ep}(\psi_\ep(t))}_{L^1(\Ga)} \leq \Phi_\ep(t) + C\gamma(K)\norm{\deln\phi_\ep(t)}_{L^2(\Ga)},
\end{align}
for almost all $t\in[0,T]$, where $t\mapsto \Phi_\ep(t)$ denotes a function in $L^2(0,T)$, which satisfies
\begin{align}\label{uniform-Phi}
    \norm{\Phi_\ep}_{L^2(0,T)} \leq C,
\end{align}
(cf. the right-hand side of \eqref{jonas7} and \eqref{jonas8}).

To eventually pass to the limit $\ep\rightarrow 0$ in the weak formulation \eqref{SING:WF}, estimate \eqref{jonas10} is not sufficient. To this end, we prove $L^2$-bounds for the terms $f_{1,\ep}(\phi_\ep)$ and $g_{1,\ep}(\psi_\ep)$. For this, we take advantage of the domination property \eqref{domination} in the case $K\in[0,\infty)$ as well as the additional boundary regularity for $K=0$. This leads to the fact that the related analysis has to be performed differently for the cases $K\in(0,\infty]$ and $K=0$. 

\textbf{Step 5.} We first derive $L^2$-bounds of $f_{1,\ep}(\phi_\ep)$ and $g_{1,\ep}(\psi_\ep)$ in the case $K\in(0,\infty]$. 

In this case, we have $\gamma(K) = 0$, and thus due to \eqref{jonas10}, it holds
\begin{align}\label{jonas11}
    \norm{f_{1,\ep}(\phi_\ep)}_{L^2(0,T;L^1(\Om))} + \norm{g_{1,\ep}(\psi_\ep)}_{L^2(0,T;L^1(\Ga))} \leq C.
\end{align}
Furthermore, since both $f_2$ and $g_2$ are Lipschitz continuous, exploiting the estimate \eqref{est:energy} yields
\begin{align}\label{jonas12}
    \norm{F_\ep^\prime(\phi_\ep)}_{L^2(0,T;L^1(\Om))} + \norm{G_\ep^\prime(\psi_\ep)}_{L^2(0,T;L^1(\Ga))} \leq C.
\end{align}
Now, since $K\in(0,\infty]$, we have $\mathcal{H}^1_{K,\alpha} = \mathcal{H}^1$ and may test \eqref{WF:MT:SING} with $\big(\beta^2\abs{\Om} + \abs{\Ga}\big)^{-1}\scp{\beta}{1}$. Hence, \eqref{est:energy} and \eqref{jonas12} imply that
\begin{align}\label{est:mt-mean}
    \norm{\mean{\mu_\ep}{\theta_\ep}}_{L^2(0,T)} \leq C.
\end{align}
Whence, using \eqref{est:energy} and the bulk-surface Poincar\'{e} inequality once again, we infer
\begin{align}\label{est:mt-L2}
    \norm{\scp{\mu_\ep}{\theta_\ep}}_{L^2(0,T;\mathcal{L}^2)} \leq C.
\end{align}
Since $f_{1,\ep}$ and $g_{1,\ep}$ are Lipschitz continuous, a chain rule for the combination of Lipschitz and Sobolev functions (see, e.g., \cite[Corollary~3.2]{Ziemer}) implies that $\scp{f_{1,\ep}(\phi_\ep)}{g_{1,\ep}(\psi_\ep)}\in\mathcal{H}^1$ 
with%
\begin{alignat*}{2}
    \Grad \big(f_{1,\ep}(\phi_\ep)\big) &= f_{1,\ep}'(\phi_\ep) \Grad\phi_\ep 
    &&\quad\text{a.e.~in $\Omega$},
    \\
    \Gradg \big(g_{1,\ep}(\psi_\ep)\big) &= g_{1,\ep}'(\psi_\ep) \Gradg\psi_\ep
    &&\quad\text{a.e.~on $\Gamma$}.
\end{alignat*}
Testing \eqref{WF:MT:SING} with $\scp{f_{1,\ep}(\phi_\ep)}{g_{1,\ep}(\psi_\ep)}$ yields
\begin{align}\label{jonas13}
    \begin{split}
        &\norm{f_{1,\ep}(\phi_\ep)}^2_{L^2(\Om)} + \norm{g_{1,\ep}(\psi_\ep)}^2_{L^2(\Ga)} 
        \\
        &\qquad
        + \intO f_{1,\ep}'(\phi_\ep)\abs{\Grad\phi_\ep}^2\dx 
        + \intG g_{1,\ep}'(\psi_\ep)\abs{\Gradg\psi_\ep}^2 \dG \\
        &\quad = \intO \big(\mu_\ep - f_2(\phi_\ep)\big)f_{1,\ep}(\phi_\ep)\dx + \intG \big(\theta_\ep - g_2(\psi_\ep)\big)g_{1,\ep}(\psi_\ep)\dG \\
        &\qquad - \sigma(K)\intG (\alpha\psi_\ep - \phi_\ep)\big(\alpha g_{1,\ep}(\psi_\ep) - f_{1,\ep}(\phi_\ep)\big)\dG.
    \end{split}
\end{align}
First, note that due to the monotonicity of $f_{1,\ep}$ and $g_{1,\ep}$, the third and the fourth term on the left-hand side of \eqref{jonas13} are non-negative.

Then, for the first two terms on the right-hand side of \eqref{jonas13} we use the Lipschitz continuity of $f_2$ and $g_2$ to obtain
\begin{align}\label{jonas15}
    \begin{split}
        &\intO \big(\mu_\ep - f_2(\phi_\ep)\big)f_{1,\ep}(\phi_\ep)\dx + \intG \big(\theta_\ep - g_2(\psi_\ep)\big)g_{1,\ep}(\psi_\ep)\dG \\
        &\quad\leq \frac12\norm{f_{1,\ep}(\phi_\ep)}^2_{L^2(\Om)} + \frac14\norm{g_{1,\ep}(\psi_\ep)}^2_{L^2(\Ga)} \\
        &\qquad + C\Big(1 + \norm{\phi_\ep}^2_{L^2(\Om)} + \norm{\psi_\ep}^2_{L^2(\Ga)} + \norm{\mu_\ep}^2_{L^2(\Om)} + \norm{\theta_\ep}^2_{L^2(\Ga)}\Big).
    \end{split}
\end{align}
If $K=\infty$, \eqref{jonas15} already implies
\begin{align}\label{fg-L^2}
    \norm{f_{1,\ep}(\phi_\ep)}_{L^2(0,T;L^2(\Om))} + \norm{g_{1,\ep}(\psi_\ep)}_{L^2(0,T;L^2(\Ga))} \leq C
\end{align}
by virtue of \eqref{jonas13} and \eqref{est:energy}. In the case $K\in(0,\infty)$, we have to deal with the additional boundary term in \eqref{jonas13}, which is done with the help of the domination property \eqref{domination-reg} and the monotonicity of $f_{1,\ep}$. Namely, it holds that
\begin{align*}
    \begin{split}
        &-\frac1K\intG (\alpha\psi_\ep - \phi_\ep)(\alpha g_{1,\ep}(\psi_\ep) - f_{1,\ep}(\phi_\ep))\dG \\
        &\quad =  -\frac1K \intG (\alpha\psi_\ep - \phi_\ep)(f_{1,\ep}(\alpha\psi_\ep) - f_{1,\ep}(\phi_\ep))\dG \\
        &\qquad - \frac1K \intG (\alpha\psi_\ep - \phi_\ep)(\alpha g_{1,\ep}(\psi_\ep) - f_{1,\ep}(\alpha\psi_\ep))\dG \\
        &\quad \leq -\frac1K \intG (\alpha\psi_\ep - \phi_\ep)(\alpha g_{1,\ep}(\psi_\ep) - f_{1,\ep}(\alpha\psi_\ep))\dG \\
        &\quad\leq \frac{\abs{\alpha} + \kappa_1}{K}\intG \abs{\alpha\psi_\ep - \phi_\ep}\abs{g_{1,\ep}(\psi_\ep)} \dG + \frac{\kappa_2}{K} \intG\abs{\alpha\psi_\ep - \phi_\ep}\dG \\
        &\quad\leq \frac14 \norm{g_{1,\ep}(\psi_\ep)}^2_{L^2(\Ga)} + C\Big(\norm{\phi_\ep}^2_{H^1(\Om)} + \norm{\psi_\ep}^2_{L^2(\Ga)}\Big),
    \end{split}
\end{align*}
from which we deduce \eqref{fg-L^2} as well. 

\textbf{Step 6.} We now derive $L^2$-bounds of $f_{1,\ep}(\phi_\ep)$ and $g_{1,\ep}(\psi_\ep)$ in the case $K = 0$. 

In this case, we multiply \eqref{STRG:MU} and \eqref{STRG:THETA} by $\big(\beta^2\abs{\Om} + \abs{\Ga}\big)^{-1}\beta$ and $\big(\beta^2\abs{\Om} + \abs{\Ga}\big)^{-1}$ and integrate over $\Om$ and $\Ga$, respectively. After adding the resulting equations and applying integration by parts, we infer
\begin{align}\label{jonas16}
    \abs{\mean{\mu_\ep}{\theta_\ep}} \leq C\Big(\norm{F^\prime_\ep(\phi_\ep)}_{L^1(\Om)} + \norm{G^\prime_\ep(\psi_\ep)}_{L^1(\Ga)}\Big).
\end{align}
Exploiting the Lipschitz continuity of $F_2^\prime$ and $G_2^\prime$, respectively, as well as \eqref{est:energy} and \eqref{jonas10}, we deduce from \eqref{jonas16} that
\begin{align}\label{jonas17}
    \abs{\mean{\mu_\ep}{\theta_\ep}} \leq C\Big( 1 + \Phi_\ep + \norm{\deln\phi_\ep}_{L^2(\Ga)}\Big).
\end{align}
Thus, after another application of the bulk-surface Poincar\'{e} inequality, we arrive at
\begin{align}\label{jonas18}
    \norm{\mu_\ep}_{H^1(\Om)} + \norm{\theta_\ep}_{H^1(\Ga)} \leq C\Big( \norm{\Grad\mu_\ep}_{\mathbf{L}^2(\Om)} + \norm{\Gradg\theta_\ep}_{\mathbf{L}^2(\Ga)} + \Phi_\ep + \norm{\deln\phi_\ep}_{L^2(\Ga)}\Big).
\end{align}
Now, in view of the estimate \eqref{est:energy}, we can update the function $\Phi_\ep$ and conclude
\begin{align}\label{jonas19}
    \norm{\mu_\ep}_{H^1(\Om)} + \norm{\theta_\ep}_{H^1(\Ga)} \leq C\Big(\Phi_\ep + \norm{\deln\phi_\ep}_{L^2(\Ga)}\Big).
\end{align}
Next, we multiply \eqref{STRG:MU:EP} by $f_{1,\ep}(\phi_\ep)$, integrate over $\Om$, and integrate by parts. This yields
\begin{align}\label{jonas20}
    \begin{split}
        &\norm{f_{1,\ep}(\phi_\ep)}_{L^2(\Om)}^2 + \int_\Om \abs{\Grad\phi_\ep}^2 f_{1,\ep}^\prime(\phi_\ep)\dx \\ 
        &\quad = \intO \big(\mu_\ep - F_2^\prime(\phi_\ep)\big)f_{1,\ep}(\phi_\ep)\dx + \intG \deln\phi_\ep f_{1,\ep}(\phi_\ep)\dG \\
        &\quad \leq \frac12 \norm{f_{1,\ep}(\phi_\ep)}_{L^2(\Om)}^2 + \frac12 \intO \abs{\mu_\ep - F_2^\prime(\phi_\ep)}^2 \dx + \intG \deln\phi_\ep f_{1,\ep}(\phi_\ep)\dG.
    \end{split}
\end{align}
Similarly, after multiplying \eqref{STRG:THETA:EP} by $-g_{1,\ep}(\psi_\ep)$, integrating over $\Ga$ and applying integration by parts, we obtain
\begin{align}\label{jonas21}
    \begin{split}
        &\norm{g_{1,\ep}(\psi_\ep)}_{L^2(\Ga)}^2 + \intG \abs{\Gradg\psi_\ep}^2 g_{1,\ep}^\prime(\psi_\ep)\dG \\
        &\quad = \intG \big(\theta_\ep - G_2^\prime(\psi_\ep)\big)g_{1,\ep}(\psi_\ep)\dG - \intG \alpha\deln\phi_\ep g_{1,\ep}(\psi_\ep)\dG \\
        &\quad\leq \frac14 \norm{g_{1,\ep}(\psi_\ep)}_{L^2(\Ga)}^2  + \intG \abs{\theta_\ep - G_2^\prime(\psi_\ep)}^2 \dG - \intG \alpha\deln\phi_\ep g_{1,\ep}(\psi_\ep)\dG.
    \end{split}
\end{align}
Before adding \eqref{jonas20} and \eqref{jonas21}, note that due to the boundary condition $\phi_\ep = \alpha\psi_\ep$ a.e.~on $\Sigma$ and the domination property of the regularized potentials \eqref{domination-reg}, we find that
\begin{align}\label{jonas22}
    \begin{split}
        &\abs{\intG \deln\phi_\ep f_{1,\ep}(\phi_\ep)\dG - \intG\alpha\deln\phi_\ep g_{1,\ep}(\psi_\ep)\dG} \\
        &\quad\leq\norm{\deln\phi_\ep}_{L^2(\Ga)}\norm{f_{1,\ep}(\alpha\psi_\ep) - \alpha g_{1,\ep}(\psi_\ep)}_{L^2(\Ga)} \\
        &\quad\leq C\norm{\deln\phi_\ep}_{L^2(\Ga)}\big( 1 + \norm{g_{1,\ep}(\psi_\ep)}_{L^2(\Ga)}\big) \\
        &\quad\leq \frac14 \norm{g_{1,\ep}(\psi_\ep)}_{L^2(\Ga)}^2 + C\big(1 + \norm{\deln\phi_\ep}_{L^2(\Ga)}^2\big).
    \end{split}
\end{align}
Hence, combining \eqref{jonas20}-\eqref{jonas22} along with the respective monotonicity of $f_{1,\ep}$ and $g_{1,\ep}$ as well as the respective Lipschitz continuity of $F_2^\prime$ and $G_2^\prime$, we conclude
\begin{align*}
    \begin{split}
        \norm{f_{1,\ep}(\phi_\ep)}_{L^2(\Om)}^2 + \norm{g_{1,\ep}(\psi_\ep)}_{L^2(\Ga)}^2 &\leq C\Big( 1 + \norm{\phi_\ep}_{L^2(\Om)}^2 + \norm{\psi_\ep}_{L^2(\Ga)}^2  + \norm{\mu_\ep}_{L^2(\Om)}^2 \\
        &\qquad + \norm{\theta_\ep}_{L^2(\Ga)}^2 + \norm{\deln\phi_\ep}_{L^2(\Ga)}^2\Big).
    \end{split}
\end{align*}
Consequently, we again find a function $\Phi_\ep$ satisfying \eqref{uniform-Phi} such that
\begin{align}\label{jonas23}
    \norm{f_{1,\ep}(\phi_\ep)}_{L^2(\Om)} + \norm{g_{1,\ep}(\psi_\ep)}_{L^2(\Ga)}  \leq C\big( 1 + \Phi_\ep + \norm{\deln\phi_\ep}_{L^2(\Ga)}\big).
\end{align}
Now, recalling \eqref{STRG:MU:EP}-\eqref{STRG:PHIPSI:EP}, we apply regularity theory for bulk-surface elliptic systems (see Proposition~\ref{Prop:Appendix} with $p=2$) to deduce that
\begin{align*}
    \begin{split}
    \norm{\phi_\ep}_{H^2(\Om)} + \norm{\psi_\ep}_{H^2(\Ga)}&\leq C\big(\norm{\mu_\ep - F_\ep^\prime(\phi_\ep)}_{L^2(\Om)} + \norm{\theta_\ep - G_\ep^\prime(\psi_\ep)}_{L^2(\Ga)}\big). 
    \end{split}
\end{align*}
Now, in view of \eqref{est:energy}, \eqref{jonas19} and \eqref{jonas23}, and after possibly redefining $\Phi_\ep$, we infer
\begin{align}\label{jonas24}
    \norm{\phi_\ep}_{H^2(\Om)} + \norm{\psi_\ep}_{H^2(\Ga)} \leq \Phi_\ep + C\norm{\deln\phi_\ep}_{L^2(\Ga)}.
\end{align}
Lastly, we establish a uniform $L^2$-bound of $\deln\phi_\ep$. To this end, by the trace theorem (see, e.g., \cite[Chapter~2, Theorem~2.24]{Brezzi1987}), we find $3/2 < s < 2$ and $C_s > 0$ such that
\begin{align}\label{est:trace}
    \norm{\deln\phi_\ep}_{L^2(\Ga)} \leq C_s\norm{\phi_\ep}_{H^s(\Om)}.
\end{align}
Thus, combining \eqref{jonas24} and \eqref{est:trace}, we exploit the compact embeddings $H^2(\Om)\emb H^s(\Om)\emb L^2(\Om)$ along with Ehrling's lemma to obtain 
\begin{align}\label{est:delnphi}
    \begin{split}
        &\norm{\phi_\ep}_{H^2(\Om)} + \norm{\psi_\ep}_{H^2(\Ga)} + \norm{\deln\phi_\ep}_{L^2(\Ga)} 
        \\[1ex]
        &\quad\leq \Phi_\ep + C\norm{\phi_\ep}_{H^s(\Om)} \\
        &\quad\leq \frac12\norm{\phi_\ep}_{H^2(\Om)} + \Phi_\ep + C\norm{\phi_\ep}_{L^2(\Om)}.
    \end{split}
\end{align}
Eventually, from \eqref{est:energy} and \eqref{est:delnphi}, we infer
\begin{align}\label{EST:PP:H^2:K0}
    \norm{\phi_\ep}_{L^2(0,T;H^2(\Om))} + \norm{\psi_\ep}_{L^2(0,T;H^2(\Ga))} + \norm{\deln\phi_\ep}_{L^2(0,T;L^2(\Ga))} \leq C,
\end{align}
and consequently, recalling \eqref{jonas19} and \eqref{jonas23}, we have
\begin{align}\label{est:fgmt:L^2}
    \begin{split}
    &\norm{f_{1,\ep}(\phi_\ep)}_{L^2(0,T;L^2(\Om))} + \norm{g_{1,\ep}(\psi_\ep)}_{L^2(0,T;L^2(\Ga))} 
    \\
    &\quad
    + \norm{\mu_\ep}_{L^2(0,T;H^1(\Om))} + \norm{\theta_\ep}_{L^2(0,T;H^1(\Ga))} \leq C.
    \end{split}
\end{align}

\textbf{Step 7.} We eventually pass to the limit $\ep\rightarrow 0$. 

In view of the estimates \eqref{est:energy}, \eqref{est:delt}, \eqref{est:mt-L2}, \eqref{fg-L^2} and \eqref{est:fgmt:L^2}, the Banach--Alaoglu theorem and the Aubin--Lions--Simon lemma imply the existence of functions $\phi, \psi, \mu, \theta, \xi$ and $\xi_\Ga$ such that
\begin{alignat}{3}
     \scp{\delt\phi_\ep}{\delt\psi_\ep} &\rightarrow \scp{\delt\phi}{\delt\psi} \qquad&&\text{weakly in~} L^2(0,T;(\mathcal{H}^1_{L,\beta})^\prime), \\
     \scp{\phi_\ep}{\psi_\ep} &\rightarrow \scp{\phi}{\psi} &&\text{weakly-star in~} L^\infty(0,T;\mathcal{H}^1_{K,\alpha}), \nonumber \\
     & &&\text{strongly in~} C([0,T];\mathcal{H}^s) \ \text{for any~} s\in[0,1), \label{CONV:PP} \\
     \scp{f_{1,\ep}(\phi_\ep)}{g_{1,\ep}(\psi_\ep)} &\rightarrow \scp{\xi}{\xi_\Ga} &&\text{weakly in~} L^2(0,T;\mathcal{L}^2), \label{CONV:fg}\\
     \scp{\mu_\ep}{\theta_\ep} &\rightarrow \scp{\mu}{\theta} \qquad&&\text{weakly in~} L^2(0,T;\mathcal{H}^1_{L,\beta}),
\end{alignat}
along a subsequence $\ep\rightarrow 0$, which will not be relabeled.
Arguing as in \cite[Theorem 3.2]{Knopf2024}, we can easily show that the above weak and strong convergences suffice to pass to the limit in the weak formulation \eqref{WF:PP:SING}-\eqref{WF:MT:SING} written for $f_1 = f_{1,\ep}$ and $g_1 = g_{1,\ep}$. Regarding Definition~\ref{DEF:SING:WS:WF}, we have to make sure that the inclusions
\begin{align}\label{WF:incl}
    \left\{
    \begin{aligned}
    \phi(x,t) &\in D(f_1) &&\;\text{for almost all $(x,t)\in Q$,}
    \qquad
    &\xi&\in f_1(\phi) &&\;\text{a.e.~in~} Q,
    \\
    \psi(z,t) &\in D(g_1) &&\;\text{for almost all $(z,t)\in \Sigma$,}
    \qquad
    &\xi_\Ga&\in g_1(\psi) &&\;\text{a.e.~on~} \Sigma
    \end{aligned}
    \right.
\end{align}
hold true. To this end, first note that due to the convergence \eqref{CONV:PP}-\eqref{CONV:fg} along with the weak-strong convergence principle, we infer
\begin{align*}
    \lim_{\ep\rightarrow 0}\int_0^T\intO f_{1,\ep}(\phi_\ep)\phi_\ep\dxt &= \int_0^T\intO\xi\phi\dxt, 
    \\
    \lim_{\ep\rightarrow 0}\int_0^T\intG g_{1,\ep}(\psi_\ep)\psi_\ep\dGt &= \int_0^T\intG\xi_\Ga\psi\dGt.
\end{align*}
The inclusions \eqref{WF:incl} then follow from the maximality of the monotone operators $f_1$ and $g_1$ (see \cite[Proposition 1.1, p.~42]{Barbu1976} and also \cite[Section~5.2]{Garcke2017}). Furthermore, it is straightforward to check that the mass conservation law as in Definition~\ref{DEF:SING:WS:MCL}  holds, due to the aforementioned strong convergences of $\phi_\ep$ and $\psi_\ep$. Lastly, the energy inequality in Definition~\ref{DEF:SING:WS:WEDL} can be established using the above convergences and lower semicontinuity arguments (cf.~\cite[Theorem 3.2]{Knopf2024} or \cite[Theorem 2.6]{Colli2024}).

\textbf{Step 8.} Finally, we establish higher regularity results. 

We have already seen that in the case $K=0$ we additionally have the bound
\begin{align*}
    \norm{\phi_\ep}_{L^2(0,T;H^2(\Om))} + \norm{\psi_\ep}_{L^2(0,T;H^2(\Ga))} \leq C,
\end{align*}
see \eqref{EST:PP:H^2:K0}. Another application of the Banach--Alaoglu theorem readily implies the desired regularity $\scp{\phi}{\psi}\in L^2(0,T;\mathcal{H}^2)$. 

In the case $K\in(0,\infty]$, we can argue similarly by exploiting the equations \eqref{STRG:MU:EP}-\eqref{STRG:PHIPSI:EP} and regularity theory for bulk-surface elliptic systems (see Proposition~\ref{Prop:Appendix} with $p=2$). 
For more details in the case $K\in(0,\infty)$ and $L=\infty$, we refer to \cite[Theorem 6.1]{Knopf2024}. 

The regularities stated in \eqref{REG:PP:W26} are a direct consequence of Proposition~\ref{Prop:FG:Linf} and Corollary~\ref{Cor:pp:W2r}, which will be presented in Subsection~\ref{SUBSECT:HIGHREG}.
\end{proof}

\section{Continuous dependence and uniqueness results}
\label{SECT:UNIQUENESS}

In this section, we present the proof of the result on continuous dependence and uniqueness of weak solutions to the system \eqref{EQ:SYSTEM} as stated in Theorem~\ref{THEOREM:UNIQUE:SING}.

\begin{proof}[Proof of Theorem~\ref{THEOREM:UNIQUE:SING}.]
    As the mobilities $m_\Om$ and $m_\Ga$ are supposed to be constant, we assume, without loss of generality, that $m_\Om \equiv m_\Ga \equiv 1$. In the following, $C$ denotes generic positive constants that may change their value from line to line and only depend on $\Om$, the parameters of the system \eqref{EQ:SYSTEM}, and the norms of the initial data.\\[0.3em]
    We consider two weak solutions $(\phi_1,\psi_1,\mu_1,\theta_1,\xi_1,\xi_{\Ga,1})$ and $(\phi_2,\psi_2,\mu_2,\theta_2,\xi_2,\xi_{\Ga,2})$ corresponding to the initial data $\scp{\phi_0^1}{\psi_0^1}, \scp{\phi_0^2}{\psi_0^2}\in\mathcal{H}^1_{K,\alpha}$ and the velocity fields $\boldsymbol{v}_1, \boldsymbol{v}_2\in L^2(0,T;\mathbf{L}_{\Div}^3(\Om))$ and $\boldsymbol{w}_1, \boldsymbol{w}_2\in L^2(0,T;\mathbf{L}_\tau^{2+\omega}(\Ga))$, respectively, and set
    \begin{align*}
        (\phi, \psi, \mu, \theta, \xi, \xi_\Ga) \coloneqq (\phi_2 - \phi_1, \psi_2 - \psi_1, \mu_2 - \mu_1, \theta_2 - \theta_1, \xi_2 - \xi_1, \xi_{\Ga,2} - \xi_{\Ga,1}),
    \end{align*}
    as well as
    \begin{align*}
        (\boldsymbol{v}, \boldsymbol{w}) = ( \boldsymbol{v}_2 - \boldsymbol{v}_1, \boldsymbol{w}_2 - \boldsymbol{w}_1).
    \end{align*}
    Then, repeating the line of argument presented in \cite[Theorem 3.5]{Knopf2024}, we can show that for any $\ep > 0$ it holds
    \begin{align}\label{est:pre:Gronwall}
        \begin{split}
            &\ddt\frac12\norm{\scp{\phi}{\psi}}_{L,\beta,\ast}^2 - C\ep\norm{\scp{\phi}{\psi}}_{K,\alpha}^2\\
            &\quad\leq\frac12\norm{\scp{\boldsymbol{v}}{\boldsymbol{w}}}_{\mathcal{Y}_\omega}^2 +  C\norm{\scp{\boldsymbol{v}_1}{\boldsymbol{w}_1}}_{\mathcal{Y}_\omega}^2\norm{\scp{\phi}{\psi}}_{L,\beta,\ast}^2 - \bigscp{\scp{\mu}{\theta}}{\scp{\phi}{\psi}}_{\mathcal{L}^2}
        \end{split}
    \end{align}
    almost everywhere on $[0,T]$, where we use the notation $\mathcal{Y}_\omega = \mathbf{L}^3(\Om)\times \mathbf{L}^{2+\omega}(\Ga)$. For the last term on the right-hand side of \eqref{est:pre:Gronwall}, we use the weak formulation for $\scp{\mu}{\theta}$ and obtain
    \begin{align*}
        \begin{split}
            -\bigscp{\scp{\mu}{\theta}}{\scp{\phi}{\psi}}_{\mathcal{L}^2} &= - \norm{\scp{\phi}{\psi}}_{K,\alpha}^2 - \intO \xi\,\phi + [F^\prime_2(\phi_2) - F^\prime_2(\phi_1)]\;\phi\dx \\
            &\quad - \intG \xi_\Ga\,\psi + [G^\prime_2(\psi_2) - G^\prime_2(\psi_1)]\;\psi\dG.
        \end{split}
    \end{align*}
    almost everywhere on $[0,T]$. Then, using the fact that $f_1$ and $g_1$ are maximal monotone graphs, we infer together with the Lipschitz continuity of $F_2^\prime$ and $G_2^\prime$ that
    \begin{align}\label{jonas75} 
        \begin{split}
            &-\intO \xi\,\phi + [F^\prime_2(\phi_2) - F^\prime_2(\phi_1)]\phi\dx - \intG \xi_\Ga\,\psi + [G^\prime_2(\psi_2) - G^\prime_2(\psi_1)]\psi \dG \\
            &\quad\leq C \norm{\scp{\phi}{\psi}}_{\mathcal{L}^2}^2.
        \end{split}
    \end{align}
    almost everywhere on $[0,T]$. Combining \eqref{est:pre:Gronwall} and \eqref{jonas75} along with an application of Ehrling's lemma yields that
    \begin{align*}
        &\ddt\frac12\norm{\scp{\phi}{\psi}}_{L,\beta,\ast}^2 + \frac12\norm{\scp{\phi}{\psi}}_{K,\alpha}^2 \\
        &\quad\leq \frac12\norm{\scp{\boldsymbol{v}}{\boldsymbol{w}}}_{\mathcal{Y}_\omega}^2 +  C\norm{\scp{\boldsymbol{v}_1}{\boldsymbol{w}_1}}_{\mathcal{Y}_\omega}^2\norm{\scp{\phi}{\psi}}_{L,\beta,\ast}^2
    \end{align*}
    almost everywhere on $[0,T]$, and then Gronwall's lemma readily entails the continuous dependence estimate \eqref{EST:continuous-dependence:sing}. 
    
    Then, the uniqueness for the phase-fields immediately follows by taking $\scp{\phi_0^1}{\psi_0^1} = \scp{\phi_0^2}{\psi_0^2}$ a.e.~in $\Om\times\Ga$, $\boldsymbol{v}_1 = \boldsymbol{v}_2$ a.e.~in $Q$ and $\boldsymbol{w}_1 = \boldsymbol{w}_2$ a.e.~on $\Sigma$. Next, assuming that the operators $f_1$ and $g_1$ are single-valued, the uniqueness of the phase-fields immediately implies that $\xi_1 = \xi_2$ a.e.~in $Q$ and $\xi_{\Ga,1} = \xi_{\Ga,2}$ a.e.~on $\Sigma$. Hence, from the weak formulation for $\scp{\mu}{\theta}$ we infer
    \begin{align*}
        \intO \mu\,\eta\dx + \intG \theta\,\eta_\Ga\dG = 0
    \end{align*}
    for all $\scp{\eta}{\eta_\Ga}\in\mathcal{H}^1_{K,\alpha}$. Thus, $\scp{\mu}{\theta} = 0$ a.e.~in $Q\times\Sigma$, which completes the proof.
\end{proof}

\section{Higher regularity and existence of strong solutions}
\label{SECT:HIGHREG}

\subsection{Higher regularity for weak solutions}
\label{SUBSECT:HIGHREG}

We first establish higher regularity of the potential terms.

\begin{proposition}\label{Prop:FG:Linf}
    Suppose that the assumptions of Theorem~\ref{THEOREM:EOWS} hold. Furthermore, assume that $\Om$ is of class $C^2$ and that \ref{S4} holds. Then
    \begin{align}
        F^\prime(\phi)&\in L^2(0,T;L^r(\Om)), \label{F':LinfLr}\\
        G^\prime(\psi)&\in L^2(0,T;L^s(\Ga)), \label{G':LinfLr}
    \end{align}
    where \eqref{F':LinfLr} holds for $r\in[1,6]$ if $d = 3$, and $r\in[1,\infty)$ if $d = 2$, and \eqref{G':LinfLr} holds for $s\in[1,\infty)$.
\end{proposition}

\begin{proof}
In this proof, without loss of generality, we assume $\alpha\neq 0$. The case $\alpha=0$ can be handled similarly and the computations are even easier.
As previously done, e.g., in \cite{Conti2020, Giorgini2017}, we consider for $k\in\N$ the truncation
\begin{align}\label{sigmak}
    \sigma_k:\R\rightarrow\R, \quad \sigma_k(\tau)\coloneqq \begin{cases}
        1 - \frac1k, &\tau\geq 1 - \frac1k, \\
        \tau, &\abs{\tau} < 1 - \frac1k, \\
        -1 + \frac1k, &\tau\leq -1 + \frac1k.
    \end{cases}
\end{align}
Now, for any $K\in[0,\infty]$, we define $\phi_k \coloneqq \alpha \,\sigma_k\circ(\alpha^{-1}\phi)$ and $\psi_k \coloneqq \sigma_k\circ\psi$.
To prove the assertions \eqref{F':LinfLr} and \eqref{G':LinfLr}, we need to handle the cases $K\in(0,\infty]$ and $K=0$ separately.

\textit{The case $K\in(0,\infty]$.}
Let $r\in[2,6]$ if $d = 3$ and $r\in[2,\infty)$ if $d = 2$. 
For $k\in\N$, we define
\begin{align}
    \label{DEF:ETAK}
    \eta_k &\coloneqq \abs{F_1^\prime(\phi_k)}^{r-2}F_1^\prime(\phi_k),  
    \\
    \label{DEF:ETAKG}
    \eta_{\Ga k} &\coloneqq \abs{G_1^\prime(\psi_k)}^{r-2}G_1^\prime(\psi_k),
\end{align}
where we use the convention $0^0=1$ in the case $r=2$.
We now fix $k\in\N$ and we multiply the equation
\begin{align*}
    \mu = -\Lap\phi + F^\prime(\phi) \quad\text{a.e.~in~} Q
\end{align*}
by $\eta_k$ and integrate over $\Om$, and we multiply
\begin{align*}
    \theta = -\Lapg\psi + G^\prime(\psi) + \alpha\deln\phi \quad\text{a.e.~on~} \Sigma
\end{align*}
by $\eta_{\Ga k}$ and integrate over $\Ga$. Applying integration by parts, and adding the resulting equations, yields by means of the boundary condition 
\begin{align*}
    \begin{cases} 
        K\deln\phi = \alpha\psi - \phi 
        &\text{if} \ K\in (0,\infty), \\
        \deln\phi = 0 
    &\text{if} \ K = \infty
    \end{cases} 
    && \text{a.e.~on} \ \Sigma,
\end{align*}
that
\begin{align}\label{testeq}
    \begin{split}
        &(r-1)\intO \Grad\phi\cdot\Grad\phi_k \abs{F_1^\prime(\phi_k)}^{r-2}F_1^{\prime\prime}(\phi_k) \dx \\
        &\qquad + (r-1)\intG \Gradg\psi\cdot\Gradg\psi_k \abs{G_1^\prime(\psi_k)}^{r-2}G_1^{\prime\prime}(\psi_k) \dG \\
        &\qquad + \intO F_1^\prime(\phi) F_1^\prime(\phi_k) \abs{F_1^\prime(\phi_k)}^{r-2}\dx + \intG G_1^\prime(\psi) G_1^\prime(\psi_k) \abs{G_1^\prime(\psi_k)}^{r-2}\dG \\
        &\quad = \intO \mu\, F_1^\prime(\phi_k)\abs{F_1^\prime(\phi_k)}^{r-2}\dx + \intG \theta\,G_1^\prime(\psi_k)\abs{G_1^\prime(\psi_k)}^{r-2}\dG \\
        &\qquad - \intO F_2^\prime(\phi) F_1^\prime(\phi_k)\abs{F_1^\prime(\phi_k)}^{r-2}\dx - \intG G_2^\prime(\psi) G_1^\prime(\psi_k)\abs{G_1^\prime(\psi_k)}^{r-2}\dG \\
        &\qquad - \sigma(K)\intG (\alpha\psi - \phi)\big(\alpha G_1^\prime(\psi_k)\abs{G_1^\prime(\psi_k)}^{r-2} - F_1^\prime(\phi_k)\abs{F_1^\prime(\phi_k)}^{r-2}\big)\dG \\
        &\quad\eqqcolon \sum_{j=1}^5 I_j 
    \end{split}
\end{align}
almost everywhere on $[0,T]$. Since $F_1$ and $G_1$ are convex, we infer that the first two terms on the left-hand side of \eqref{testeq} are non-negative. 
Moreover, we infer from assumption \ref{S4} that $F_1^\prime(\tau)$ and $F_1^\prime(\alpha\sigma_k(\alpha^{-1}\tau))$ have the same sign for all $\tau\in(-1,1)$. In combination with \eqref{sigmak}, this yields
\begin{align}\label{est:Fk:squared}
    F_1^\prime(\phi_k)^2\leq F_1^\prime(\phi) F_1^\prime(\phi_k)\quad\text{a.e.~in~} Q.
\end{align}
Analogously, we obtain
\begin{align}\label{est:Gk:squared}
    G_1^\prime(\psi_k)^2\leq G_1^\prime(\psi) G_1^\prime(\psi_k)\quad\text{a.e.~on~}\Sigma.
\end{align}
Hence, we conclude that
\begin{align}\label{est:Lr}
    \begin{split}
        &\intO \abs{F_1^\prime(\phi_k)}^r\dx + \intG \abs{G_1^\prime(\psi_k)}^r\dG \\
        &\quad\leq \intO F_1^\prime(\phi) F_1^\prime(\phi_k) \abs{F_1^\prime(\phi_k)}^{r-2}\dx + \intG G_1^\prime(\psi) G_1^\prime(\psi_k) \abs{G_1^\prime(\psi_k)}^{r-2}\dG\\
        &\quad\le \sum_{j=1}^5 I_j 
    \end{split}
\end{align}
almost everywhere on $[0,T]$. Now, we intend to bound the terms $I_j$ for $j=1,\ldots,5$. Due to the Lipschitz continuity of $F_2^\prime$ and $G_2^\prime$, we find
\begin{align}
    \label{est:I13}
    I_1 + I_3 \leq \intO \big(\abs{\mu} + d_F\abs{\phi}\big)\abs{F_1^\prime(\phi_k)}^{r-1}\dx,
\end{align}
as well as
\begin{align}
    \label{est:I24}
    I_2 + I_4 \leq \intG \big(\abs{\theta} + d_G\abs{\psi}\big)\abs{G_1^\prime(\psi_k)}^{r-1}\dG,
\end{align}
almost everywhere on $[0,T]$ for some constants $d_F$ and $d_G$ depending only on $F$ and $G$, respectively. 

To handle the term $I_5$, which is only present if $K\in (0,\infty)$, we first notice that \ref{S4} implies
\begin{alignat*}{2}
    \phi(x,t) &\in D(F_1^\prime) \subset [-1,1]
    &&\quad\text{for almost all $(x,t)\in Q$,}
    \\
    \psi(z,t) &\in D(G_1^\prime) \subset [-1,1]
    &&\quad\text{for almost all $(z,t)\in \Sigma$.}
\end{alignat*}
Recalling the domination property \eqref{domination} and the definition of $\phi_k$, and exploiting the monotonicity of the function $\R\ni \tau\mapsto F_1^\prime(\alpha\sigma_k(\tau))\abs{F_1^\prime(\alpha\sigma_k(\tau))}^{r-2}$, we obtain
\begin{align}\label{est:I5}
    \begin{split}
        I_5 &= - \alpha\sigma(K) \intG (\psi - \alpha^{-1}\phi)\big(F_1^\prime(\alpha\psi_k)\abs{F_1^\prime(\alpha\psi_k)}^{r-2} - F_1^\prime(\phi_k)\abs{F_1^\prime(\phi_k)}^{r-2}\big)\dG \\
        &\quad - \sigma(K) \intG (\alpha\psi - \phi)\big(\alpha G_1^\prime(\psi_k)\abs{G_1^\prime(\psi_k)}^{r-2} - F_1^\prime(\alpha\psi_k)\abs{F_1^\prime(\alpha\psi_k)}^{r-2}\big)\dG \\
        &\leq 2\sigma(K)\Big(\intG \abs{G_1^\prime(\psi_k)}^{r-1}\dG + 2^{r-2}\kappa_1^{r-1}\intG \abs{G_1^\prime(\psi_k)}^{r-1}\dG + 2^{r-2}\kappa_2^{r-1}\abs{\Ga}\Big)
    \end{split}
\end{align}
almost everywhere on $[0,T]$. Here, we additionally made use of the Minkowski inequality
\begin{align*}
    \abs{a + b}^{r-1} \leq 2^{r-2}\big(\abs{a}^{r-1} + \abs{b}^{r-1}\big)
    \quad\text{for any $a,b\in\R$.}
\end{align*}
Thus, by collecting the above estimates \eqref{est:Lr}-\eqref{est:I24} and \eqref{est:I5} in combination with \eqref{testeq}, we employ Hölder's and Young's inequality to deduce
\begin{align*}
    \begin{split}
        &\norm{F_1^\prime(\phi_k)}_{L^r(\Om)}^r + \norm{G_1^\prime(\psi_k)}_{L^r(\Ga)}^r 
        \\
        &\leq\Big(\norm{\mu}_{L^r(\Om)} + d_F\norm{\phi}_{L^r(\Om)}\Big)\norm{F_1^\prime(\phi_k)}_{L^r(\Om)}^{r-1} 
        + 2^{r-1}\kappa^{r-1}_2\abs{\Ga}\sigma(K)
        \\
        &\quad + \Big(2\sigma(K) + 2^{r-1}\kappa_1^{r-1}\sigma(K) + \norm{\theta}_{L^r(\Ga)} + d_G\norm{\psi}_{L^r(\Ga)}\Big)\norm{G_1^\prime(\psi_k)}_{L^r(\Ga)}^{r-1} \\
        &\leq \frac12\norm{F_1^\prime(\phi_k)}_{L^r(\Om)}^r + \frac12\norm{G_1^\prime(\psi_k)}_{L^r(\Ga)}^r + \frac1r\left(\frac{2(r-1)}{r}\right)^{r-1}\Big(\norm{\mu}_{L^r(\Om)} + d_F\norm{\phi}_{L^r(\Om)}\Big)^r\\
        &\quad + \frac1r\left(\frac{2(r-1)}{r}\right)^{r-1}
        \Big(2\sigma(K) + 2^{r-1}\kappa_1^{r-1}\sigma(K) + \norm{\theta}_{L^r(\Ga)} + d_G\norm{\psi}_{L^r(\Ga)}\Big)^r \\
        &\quad + 2^{r-1}\kappa^{r-1}_2\abs{\Ga}\sigma(K)
    \end{split}
\end{align*}
almost everywhere on $[0,T]$. Thus, absorbing the respective terms and taking the $r$-th root, we find
\allowdisplaybreaks[0]
\begin{align}\label{Est:FG:L^r:K>0}
        &\norm{F_1^\prime(\phi_k)}_{L^r(\Om)} + \norm{G_1^\prime(\psi_k)}_{L^r(\Ga)} \nonumber \\
        &\leq 2\cdot2^{1/r}\Bigg[\frac{1}{r^{1/r}}\left(\frac{2(r-1)}{r}\right)^{(r-1)/r}\Big(\norm{\mu}_{L^r(\Om)} + d_F\norm{\phi}_{L^r(\Om)}\Big) \nonumber\\
        &\quad + \frac{1}{r^{1/r}}\left(\frac{2(r-1)}{r}\right)^{(r-1)/r}\Big(2\sigma(K) + 2^{r-1}\kappa_1^{r-1}\sigma(K) + \norm{\theta}_{L^r(\Ga)} + d_G\norm{\psi}_{L^r(\Ga)}\Big) \\
        &\quad + 2^{(r-1)/r}\kappa_2^{(r-1)/r}\abs{\Ga}^{1/r}\sigma(K)^{1/r}\Bigg]  
        \nonumber\\
        &\leq C\Big(1 + \norm{\mu}_{L^r(\Om)} + \norm{\theta}_{L^r(\Ga)} + d_F\norm{\phi}_{L^r(\Om)} + d_G\norm{\psi}_{L^r(\Ga)}\Big)\nonumber
\end{align}
\allowdisplaybreaks
almost everywhere on $[0,T]$ for a constant $C = C(r)>0$ that may depend on $r$. Exploiting the regularites $\phi,\mu\in L^2(0,T;H^1(\Om))$ as well as $\psi,\theta\in L^2(0,T;H^1(\Ga))$ along with Sobolev embeddings we infer \eqref{F':LinfLr} and \eqref{G':LinfLr} for $r,s\in[2,6]$ if $d = 3$ and $r,s\in[2,\infty)$ if $d = 2$. The cases $r,s\in[1,2)$ can be obtained by means of Hölder's inequality. 

We still need to verify \eqref{G':LinfLr} in the case $s>6$ if $d=3$. Therefore, let $s>6$ be arbitrary. Testing
\begin{align*}
    \theta = -\Lapg\psi + G^\prime(\psi) + \alpha\deln\phi\qquad\text{a.e. on~}\Sigma
\end{align*}
with $\eta_{\Ga k}$ as introduced in \eqref{DEF:ETAKG}
and applying integration by parts, we deduce
\begin{align}\label{testeq:Gk}
    \begin{split}
        &(s-1)\intG\Gradg\psi\cdot\Gradg\psi_k\abs{G_1^\prime(\psi_k)}^{s-2}G_1^{\prime\prime}(\psi_k)\dG \\
        &\qquad + \intG G_1^\prime(\psi)G_1^\prime(\psi_k)\abs{G_1^\prime(\psi_k)}^{s-2}\dG \\
        &\quad = \intG\theta\;G_1^\prime(\psi_k)\abs{G_1^\prime(\psi_k)}^{s-2}\dG - \intG G_2^\prime(\psi)G_1^\prime(\psi_k)\abs{G_1^\prime(\psi_k)}^{s-2}\dG \\
        &\qquad - \alpha\sigma(K)\intG(\alpha\psi - \phi) G_1^\prime(\psi_k)\abs{G_1^\prime(\psi_k)}^{s-2}\dG
    \end{split}
\end{align}
almost everywhere on $[0,T]$.
Here, we also used that $K\deln\phi = \alpha\psi-\phi$ holds almost everywhere on $\Sigma$. Proceeding similarly as above, we derive the estimate
\begin{align}\label{est:Gk:r:pre:1}
    \norm{G_1^\prime(\psi_k)}_{L^s(\Ga)}^s 
    \leq C \big(1+ \norm{\theta}_{L^s(\Ga)} + d_G\norm{\psi}_{L^s(\Ga)}\big)\norm{G_1^\prime(\psi_k)}_{L^s(\Ga)}^{s-1}.
\end{align}
almost everywhere on $[0,T]$. Due to the regularities $\psi,\theta\in L^2(0,T;H^1(\Ga))$ and the Sobolev embedding $H^1(\Ga) \emb L^s(\Ga)$ for any $s\in[1,\infty)$, this implies that \eqref{G':LinfLr} holds true if $s>6$. This completes the proof in the case $K\in(0,\infty]$.

\textit{The case $K=0$.}
First, let $2\leq r \leq 4$ be arbitrary. Performing the same testing procedure as in the case $K\in(0,\infty]$, we obtain
\begin{align}\label{Est:FG:k:0}
    \begin{split}
        &(r-1)\intO \Grad\phi\cdot\Grad\phi_k\abs{F_1^\prime(\phi_k)}^{r-2}F_1^{\prime\prime}(\phi_k)\dx \\
        &\qquad + (r-1) \intG \Gradg\psi\cdot\Gradg\psi_k\abs{G_1^\prime(\psi_k)}^{r-2}G_1^{\prime\prime}(\psi_k)\dG \\
        &\qquad + \intO F_1^\prime(\phi)F_1^\prime(\phi_k)\abs{F_1^\prime(\phi_k)}^{r-2}\dx + \intG G_1^\prime(\psi)G_1^\prime(\psi_k)\abs{G_1^\prime(\psi_k)}^{r-2}\dG \\
        &\quad = \intO \mu F_1^\prime(\phi_k)\abs{F_1^\prime(\phi_k)}^{r-2}\dx + \intG \theta G_1^\prime(\psi_k)\abs{G_1^\prime(\psi_k)}^{r-2}\dG \\
        &\qquad - \intO F_2^\prime(\phi)F_1^\prime(\phi_k)\abs{F_1^\prime(\phi_k)}^{r-2}\dx - \intG G_2^\prime(\psi)G_1^\prime(\psi_k)\abs{G_1^\prime(\psi_k)}^{r-2}\dG \\
        &\qquad - \intG \deln\phi\big(\alpha G_1^\prime(\psi_k)\abs{G_1^\prime(\psi_k)}^{r-2} - F_1^\prime(\phi_k)\abs{F_1^\prime(\phi_k)}^{r-2}\big)\dG.
    \end{split}
\end{align}
This equation differs from \eqref{testeq} only in the last term on the right-hand side, which now needs to be handled in a different way. We first notice that due to the regularity $\phi\in L^2(0,T;H^2(\Omega))$, the trace embedding $H^1(\Om)\emb L^4(\Ga)$ implies $\deln\phi\in L^2(0,T;L^4(\Ga))$. Then, in view of the trace relation $\phi_k = \alpha\psi_k$ a.e. on $\Sigma$ as well as the domination property \eqref{domination}, we find that
\begin{align*}
    \begin{split}
        &- \intG \deln\phi\big(\alpha G_1^\prime(\psi_k)\abs{G_1^\prime(\psi_k)}^{r-2} - F_1^\prime(\phi_k)\abs{F_1^\prime(\psi_k)}^{r-2}\big)\dG \\
        &\quad =  - \intG \deln\phi\big(\alpha G_1^\prime(\psi_k)\abs{G_1^\prime(\psi_k)}^{r-2} - F_1^\prime(\alpha\psi_k)\abs{F_1^\prime(\alpha\psi_k)}^{r-2}\big)\dG \\
        &\quad \leq \intG \abs{\deln\phi}\Big(\big(1 + 2^{r-2}\kappa_1^{r-1}\big)\abs{G_1^\prime(\psi_k)}^{r-1} + 2^{r-2}\kappa_2^{r-1}\Big)\dG.
    \end{split}
\end{align*}
Consequently,
\begin{align*}
        &\intO \abs{F_1^\prime(\phi_k)}^r\dx + \intG \abs{G_1^\prime(\psi_k)}^r \dG \\
        &\quad\leq \Big(\norm{\mu}_{L^r(\Om)} + d_F\norm{\phi}_{L^r(\Om)}\Big)\norm{F_1^\prime(\phi_k)}_{L^r(\Om)}^{r-1} \\
        &\qquad + \Big(\norm{\theta}_{L^r(\Ga)} + d_G\norm{\psi}_{L^r(\Ga)} + \big(1 + 2^{r-2}\kappa_1^{r-1}\big)\norm{\deln\phi}_{L^r(\Ga)}\Big)\norm{G_1^\prime(\psi_k)}_{L^r(\Ga)}^{r-1} \\
        &\qquad + 2^{r-2}\kappa_2^{r-1}\abs{\Ga}^{1-1/r}\norm{\deln\phi}_{L^r(\Ga)} \\
        &\quad \leq \frac12\norm{F_1^\prime(\phi_k)}_{L^r(\Om)}^r + \frac12\norm{G_1^\prime(\psi_k)}_{L^r(\Ga)}^r \\
        &\qquad + \frac1r\left(\frac{2(r-1)}{r}\right)^{r-1}\Big(\norm{\mu}_{L^r(\Om)} + d_F\norm{\phi}_{L^r(\Om)}\Big)^r \\
        &\qquad + \frac1r\left(\frac{2(r-1)}{r}\right)^{r-1}\Big(\norm{\theta}_{L^r(\Ga)} + d_G\norm{\psi}_{L^r(\Ga)} + \big(1 + 2^{r-2}\kappa_1^{r-1}\big)\norm{\deln\phi}_{L^r(\Ga)}\Big)^r \\
        &\qquad + 2^{r-2}\kappa_2^{r-1}\abs{\Ga}^{1-1/r}\norm{\deln\phi}_{L^r(\Ga)}.
\end{align*}
Thus, absorbing the respective terms and taking the $r$-th root, we deduce that
\begin{align}\label{Est:FG:L^r:K=0}
    &\norm{F_1^\prime(\phi_k)}_{L^r(\Om)} + \norm{G_1^\prime(\psi_k)}_{L^r(\Ga)}\nonumber  \\
    &\quad\leq 2\cdot 2^{1/r}\Bigg[\frac{1}{r^{1/r}}\left(\frac{2(r-1)}{r}\right)^{(r-1)/r}\Big(\norm{\mu}_{L^r(\Om)} + d_F\norm{\phi}_{L^r(\Om)}\Big) \nonumber\\
    &\qquad + \frac{1}{r^{1/r}}\left(\frac{2(r-1)}{r}\right)^{(r-1)/r}\Big(\norm{\theta}_{L^r(\Ga)} + d_G\norm{\psi}_{L^r(\Ga)} + \big(1 + 2^{r-2}\kappa_1^{r-1}\big)\norm{\deln\phi}_{L^r(\Ga)}\Big)  \nonumber \\
    &\qquad + 2^{(r-2)/r}\kappa_2^{(r-1)/r}\abs{\Ga}^{1/r - 1/r^2}\norm{\deln\phi}_{L^r(\Ga)}^{1/r}\Bigg] \\
    &\quad\leq C\Big(1 + \norm{\mu}_{L^r(\Om)} + \norm{\theta}_{L^r(\Om)} + d_F\norm{\phi}_{L^r(\Om)} + d_G\norm{\psi}_{L^r(\Ga)} + \norm{\deln\phi}_{L^r(\Ga)}\Big) \nonumber.
\end{align}
Exploiting $\phi,\mu\in L^2(0,T;H^1(\Om))$, $\psi,\theta\in L^2(0,T;H^1(\Ga))$ and $\deln\phi\in L^2(0,T;L^4(\Ga))$, we conclude that \eqref{F':LinfLr} and \eqref{G':LinfLr} hold for all $r,s\in[2,4]$. Then, using regularity theory for bulk-surface elliptic system (see Proposition~\ref{Prop:Appendix}), we find that $(\phi,\psi)\in L^2(0,T;\mathcal{W}^{2,4})$. Consequently, by the trace theorem, we even have $\deln\phi\in L^2(0,T;W^{3/4,4}(\Ga))\emb L^2(0,T;L^\infty(\Ga))$. This allows us to take $r\in[2,6]$ in our previous calculations, which verifies \eqref{F':LinfLr} and \eqref{G':LinfLr} in all cases $r,s\in[2,6]$. The cases $r\in[1,2)$ follow easily from Hölder's inequality. To obtain \eqref{G':LinfLr} also for $s>6$ if $d=3$, we proceed similarly as in the derivation of inequality \eqref{est:Gk:r:pre:1}. Instead of \eqref{est:Gk:r:pre:1}, we now obtain
\begin{align}\label{est:Gk:r:pre:K=0}
    \norm{G_1^\prime(\psi_k)}_{L^s(\Ga)}^s \leq C \big(\norm{\theta}_{L^s(\Ga)} + d_G\norm{\psi}_{L^s(\Ga)} + \norm{\deln\phi}_{L^s(\Ga)}\big)\norm{G_1^\prime(\psi_k)}_{L^s(\Ga)}^{s-1}
\end{align}
a.e. on $[0,T]$. Due to the regularities $\psi,\theta\in L^2(0,T;H^1(\Ga))$, the Sobolev embedding $H^1(\Ga)\emb L^s(\Ga)$ for any $s\in[1,\infty)$, and the aforementioned regularity of $\deln\phi$, it easily follows that \eqref{G':LinfLr} holds true for $s>6$. Therefore, the proof is complete.
\end{proof}

\medskip

As a consequence of Proposition~\ref{Prop:FG:Linf}, we obtain higher regularity results for the phase-fields by means of elliptic regularity theory.

\begin{corollary}\label{Cor:pp:W2r}
    Suppose that the assumptions of Proposition~\ref{Prop:FG:Linf} hold. Then
    \begin{align}
        \phi&\in L^2(0,T;W^{2,r}(\Om)), \label{phi:W2r}\\
        \psi&\in L^2(0,T;W^{2,s}(\Ga)), \label{psi:W2r}
    \end{align}
    where in \eqref{phi:W2r} we have $r\in[1,6]$ if $d = 3$, and $r\in[1,\infty)$ if $d = 2$, and in \eqref{psi:W2r} we have $s\in[1,\infty)$.
\end{corollary}

\begin{proof}
    For $r,s\in[2,6]$ if $d = 3$ or $r,s\in[2,\infty)$ if $d = 2$, the desired regularities follow immediately from $\scp{\mu}{\theta}\in L^2(0,T;\mathcal{H}^1)$, the regularities established in Proposition~\ref{Prop:FG:Linf}, and the $L^p$ regularity theory for bulk-surface elliptic systems established in Proposition~\ref{Prop:Appendix}. 
    The cases $r,s\in[1,2)$ can be handled by Hölder's inequality. Based on these regularities, 
    the trace theorem (see, e.g., \cite[Chapter~2, Theorem~2.24]{Brezzi1987}) further implies $\deln\phi\in L^2(0,T;W^{5/6,6}(\Ga))$. Invoking the Sobolev embedding $W^{5/6,6}(\Ga)\emb L^\infty(\Ga)$, we apply regularity theory for the Laplace--Beltrami equation (see, e.g., \cite[Lemma B.1]{Vitillaro2017}) to obtain \eqref{psi:W2r} for any $s\in[1,\infty)$. 
\end{proof}    

\medskip

\subsection{Existence of strong solutions}

In view of the regularity properties established in Proposition~\ref{Prop:FG:Linf} and Corollary~\ref{Cor:pp:W2r}, we now proceed with the proof of Theorem~\ref{thm:highreg}.

\begin{proof}[Proof of Theorem~\ref{thm:highreg}]
\textbf{Step 1.} We start by showing $\scp{\phi}{\psi}\in W^{1,\infty}(0,T;(\mathcal{H}^1_{L,\beta})^\prime)\cap H^1(0,T;\mathcal{H}^1)$ and $\scp{\Grad\mu}{\Gradg\theta}\in L^\infty(0,T;\mathcal{L}^2)$. 

For any function $f:[0,T+1]\rightarrow X$, where $X$ is a Banach space, $t\in[0,T+1]$ and $h\in(0,1)$, we denote by
\begin{align*}
    \del_t^h f(t) \coloneqq \frac{f(t+h) - f(t)}{h}
\end{align*}
the forward difference quotient of $f$ at time $t$. 

As the mobilities are assumed to be constant, we simply set $m_\Om\equiv 1$ and $m_\Ga\equiv 1$ without loss of generality.
We now extend our velocity fields onto the time interval $[0,T+1]$ such that we have 
\begin{align*}
    \boldsymbol{v} &\in H^1(0,T+1;\mathbf{L}^{6/5}(\Om))\cap L^2(0,T+1;\mathbf{L}_\Div^3(\Om)),
    \\
    \boldsymbol{w} &\in H^1(0,T+1;\mathbf{L}^{1+\omega}(\Ga))\cap L^2(0,T+1;\mathbf{L}_\Div^2(\Ga)).
\end{align*}
Note that the final result is independent of the extension as it does not depend on the values of $\boldsymbol{v}$ and $\boldsymbol{w}$ on $(T,T+1]$.

We now also extend the solution $(\phi,\psi,\mu,\theta)$ onto the interval $[0,T+1]$. This is possible since the final time $T>0$ in Theorem~\ref{THEOREM:EOWS} is arbitrary, and due to the uniqueness established in Theorem~\ref{THEOREM:UNIQUE:SING}, the solutions corresponding to $T$ and $T+1$ coincide on the interval $[0,T]$.
 
Owing to \eqref{WF:PP:SING}, for any $h\in(0,1)$, the weak solution satisfies
\begin{align}\label{wf:delth}
    \begin{split}
        \bigang{\bigscp{\delt\delt^h\phi}{\delt\delt^h\psi}}{\bigscp{\zeta}{\zeta_\Ga}}_{\mathcal{H}^1_{L,\beta}} &= \bigscp{\delt^h\phi\,\boldsymbol{v}}{\Grad\zeta}_{L^2(\Om)} + \bigscp{\delt^h\boldsymbol{v}\,\phi(\cdot+h)}{\Grad\zeta}_{L^2(\Om)} \\
        &\quad + \bigscp{\delt^h\psi\,\boldsymbol{w}}{\Gradg\zeta_\Ga}_{L^2(\Ga)} + \bigscp{\delt^h\boldsymbol{w}\,\psi(\cdot+h)}{\Gradg\zeta_\Ga}_{L^2(\Ga)} \\
        &\quad - \bigscp{\scp{\delt^h\mu}{\delt^h\theta}}{\scp{\zeta}{\zeta_\Ga}}_{L,\beta}
    \end{split}
\end{align}
almost everywhere on $[0,T]$ for all $\scp{\zeta}{\zeta_\Ga}\in\mathcal{H}^1_{L,\beta}$, where
\begin{subequations}\label{STRG:DELTH}
    \begin{align}
        &\delt^h\mu = -\Lap\delt^h\phi + \delt^h F^\prime(\phi) \qquad &&\text{a.e.~in~} Q, \label{STRG:DELTH:MU}\\
        &\delt^h\theta = - \Lapg\delt^h\psi + \delt^h G^\prime(\psi) + \alpha\deln\delt^h\phi \qquad &&\text{a.e.~on~} \Sigma, \label{STRG:DELTH:TH}\\
        & \begin{cases} 
            K\deln\delt^h\phi = \alpha\delt^h\psi - \delt^h\phi &\text{if} \ K\in [0,\infty), \\
            \deln\delt^h\phi = 0 &\text{if} \ K = \infty
        \end{cases} &&  \text{a.e.~on } \Sigma.
    \end{align}
\end{subequations}
By definition of the operator $\mathcal{S}_{L,\beta}$ (see~\ref{PRELIM:bulk-surface-elliptic}), we have
\begin{align*}
    \bigscp{\mathcal{S}_{L,\beta}(\delt^h\phi,\delt^h\psi)}{\scp{\delt^h\mu}{\delt^h\theta}}_{L,\beta} = - \bigscp{\scp{\delt^h\phi}{\delt^h\psi}}{\scp{\delt^h\mu}{\delt^h\theta}}_{\mathcal{L}^2}.
\end{align*}
Recalling that 
\begin{align*}
    \mean{\mathcal{S}^\Om_{L,\beta}(\delt^h\phi,\delt^h\psi)}{\mathcal{S}^\Ga_{L,\beta}(\delt^h\phi,\delt^h\psi)} = 0,
\end{align*}
we test \eqref{wf:delth} with $\scp{\zeta}{\zeta_\Ga} \coloneqq \mathcal{S}_{L,\beta}(\delt^h\phi,\delt^h\psi)$ and find
\begin{align}\label{est:ddt1}
    \begin{split}
        &\frac12\ddt \norm{\scp{\delt^h\phi}{\delt^h\psi}}_{L,\beta,\ast}^2 + \bigscp{\scp{\delt^h\mu}{\delt^h\theta}}{\scp{\delt^h\phi}{\delt^h\psi}}_{\mathcal{L}^2} \\
        &\quad = -\bigscp{\delt^h\phi\,\boldsymbol{v}}{\Grad\mathcal{S}^\Om_{L,\beta}(\delt^h\phi,\delt^h\psi)}_{L^2(\Om)} - \bigscp{\delt^h\boldsymbol{v}\,\phi(\cdot+h)}{\Grad\mathcal{S}^\Om_{L,\beta}(\delt^h\phi,\delt^h\psi)}_{L^2(\Om)} \\
        &\qquad - \bigscp{\delt^h\psi\,\boldsymbol{w}}{\Gradg\mathcal{S}^\Ga_{L,\beta}(\delt^h\phi,\delt^h\psi)}_{L^2(\Ga)} - \bigscp{\delt^h\boldsymbol{w}\,\psi(\cdot+h)}{\Gradg\mathcal{S}^\Ga_{L,\beta}(\delt^h\phi,\delt^h\psi)}_{L^2(\Ga)} \\
        &\quad\eqqcolon \sum_{k=1}^4 J_k.
    \end{split}
\end{align}
Next, exploiting \eqref{STRG:DELTH}, we have
\begin{align}\label{wf:delth:mt:pp}
    \begin{split}
        -\bigscp{\scp{\delt^h\mu}{\delt^h\theta}}{\scp{\delt^h\phi}{\delt^h\psi}}_{\mathcal{L}^2} &= -\norm{\scp{\delt^h\phi}{\delt^h\psi}}_{K,\alpha}^2 - \intO \delt^h[F^\prime(\phi)]\delt^h\phi\dx \\
        &\qquad - \intG\delt^h[G^\prime(\psi)]\delt^h\psi\dG
    \end{split}
\end{align}
almost everywhere on $[0,T]$. Therefore, due to the splitting $F = F_1 + F_2$ and $G = G_1 + G_2$, where $F_1^\prime$ and $G_1^\prime$ are monotonically increasing, while $F_2^\prime$ and $G_2^\prime$ are Lipschitz continuous (see \ref{S2} and \ref{S4}), and exploiting \eqref{STRG:DELTH}, we can bound the last two terms on the right-hand side of \eqref{wf:delth:mt:pp} as
\begin{align}\label{est:diffq:FG}
    - \intO \delt^h[F^\prime(\phi)]\delt^h\phi\dx - \intG\delt^h[G^\prime(\psi)]\delt^h\psi\dG \leq C\norm{\scp{\delt^h\phi}{\delt^h\psi}}_{\mathcal{L}^2}^2
\end{align}
Thus, combining \eqref{wf:delth:mt:pp} and \eqref{est:diffq:FG}, we use Ehrling's lemma to deduce that
\begin{align}\label{est:ehr}
    -\bigscp{\scp{\delt^h\mu}{\delt^h\theta}}{\scp{\delt^h\phi}{\delt^h\psi}}_{\mathcal{L}^2} \leq -\frac34 \norm{\scp{\delt^h\phi}{\delt^h\psi}}_{K,\alpha}^2 + C\norm{\scp{\delt^h\phi}{\delt^h\psi}}_{L,\beta,\ast}^2.
\end{align}
We now intend to bound the terms $J_k$, $k=1,\ldots,4$. To this end, we recall that the velocity fields $\boldsymbol{v}$ and $\boldsymbol{w}$ are divergence-free, and we use integration by parts, Hölder's inequality, Young's inequality, Sobolev embeddings, and the bulk-surface Poincar\'{e} inequality \ref{PRELIM:POINCINEQ} to obtain
\begin{align}
    J_1 + J_3 &\leq \norm{\Grad\delt^h\phi}_{\mathbf{L}^2(\Om)}\norm{\boldsymbol{v}}_{\mathbf{L}^3(\Om)}\norm{\mathcal{S}_{L,\beta}^\Om(\delt^h\phi,\delt^h\psi)}_{L^6(\Om)} \nonumber \\
    &\quad + \norm{\Gradg\delt^h\psi}_{\mathbf{L}^2(\Ga)}\norm{\boldsymbol{w}}_{\mathbf{L}^{2+\omega}(\Ga)}\norm{\mathcal{S}_{L,\beta}^\Ga(\delt^h\phi,\delt^h\psi)}_{L^{{2(2+\omega)/\omega}}(\Ga)} \nonumber \\
    &\leq \tfrac14\norm{\scp{\delt^h\phi}{\delt^h\psi}}_{K,\alpha}^2 + C\norm{\scp{\boldsymbol{v}}{\boldsymbol{w}}}_{\mathcal{Y}_\omega}^2\norm{\scp{\delt^h\phi}{\delt^h\psi}}_{L,\beta,\ast}^2, 
    \nonumber \\[1ex]
    J_2 + J_4 &\leq \norm{\Grad\phi(\cdot + h)}_{\mathbf{L}^\infty(\Om)}\norm{\delt^h\boldsymbol{v}}_{\mathbf{L}^{6/5}(\Om)}\norm{\mathcal{S}^\Om_{L,\beta}(\delt^h\phi,\delt^h\psi)}_{L^6(\Om)} \label{est:Jk} \\
    &\quad + \norm{\Gradg\psi(\cdot + h)}_{\mathbf{L}^\infty(\Ga)}\norm{\delt^h\boldsymbol{w}}_{\mathbf{L}^{1+\omega}(\Ga)}\norm{\mathcal{S}^\Ga_{L,\beta}(\delt^h\phi,\delt^h\psi)}_{L^{{(1+\omega)/\omega}}(\Ga)} \nonumber \\
    &\leq \norm{\scp{\delt^h\boldsymbol{v}}{\delt^h\boldsymbol{w}}}_{\mathcal{X}_\omega}^2 + C\norm{\scp{\phi(\cdot + h)}{\psi(\cdot + h)}}_{\mathcal{W}^{1,\infty}}^2\norm{\scp{\delt^h\phi}{\delt^h\psi}}_{L,\beta,\ast}^2, \nonumber 
\end{align}
where $\mathcal{Y}_\omega \coloneqq \mathbf{L}^3(\Om)\times \mathbf{L}^{2+\omega}(\Ga)$ and $\mathcal{X}_\omega \coloneqq \mathbf{L}^{6/5}(\Om)\times \mathbf{L}^{1+\omega}(\Ga)$.
Thus, collecting \eqref{est:ehr}-\eqref{est:Jk}, we deduce from \eqref{est:ddt1} that
\begin{align}\label{Gronwall:1}
    \begin{split}
        &\frac12\ddt \norm{\scp{\delt^h\phi}{\delt^h\psi}}_{L,\beta,\ast}^2 + \frac12\norm{\scp{\delt^h\phi}{\delt^h\psi}}_{K,\alpha}^2 \\
        &\quad\leq \norm{\scp{\delt^h\boldsymbol{v}}{\delt^h\boldsymbol{w}}}_{\mathcal{X}_\omega}^2 + C\Lambda_h\norm{\scp{\delt^h\phi}{\delt^h\psi}}_{L,\beta,\ast}^2,
    \end{split}
\end{align}
where 
\begin{align*}
    \Lambda_h\coloneqq 1 + \norm{\scp{\boldsymbol{v}}{\boldsymbol{w}}}_{\mathcal{Y}_\omega}^2 + \norm{\scp{\phi(\cdot + h)}{\psi(\cdot + h)}}_{\mathcal{W}^{1,\infty}}^2.
\end{align*}
We readily note that in view of the regularity of the phase-fields and the velocity fields we have
\begin{align}\label{reg:Lambdah}
    \sup_{h\in(0,1)}\norm{\Lambda_h}_{L^1(0,T)} < \infty.
\end{align}
Moreover, it holds that
\begin{align}
    &\bignorm{\scp{\delt^h\boldsymbol{v}(t)}{\delt^h\boldsymbol{w}(t)}}_{\mathcal{X}_\omega} \leq \frac1h\int_t^{t+h}\norm{\scp{\delt\boldsymbol{v}(s)}{\delt\boldsymbol{w}(s)}}_{\mathcal{X}_\omega}\ds, \label{est:delth:1} \\
    &\lim\limits_{h\rightarrow 0}\frac1h\int_t^{t+h}\norm{\scp{\delt\boldsymbol{v}(s)}{\delt\boldsymbol{w}(s)}}_{\mathcal{X}_\omega}\ds = \norm{\scp{\delt\boldsymbol{v}(t)}{\delt\boldsymbol{w}(t)}}_{\mathcal{X}_\omega} \label{est:delth:2}
\end{align}
for almost every $t\in[0,T]$.
Thus, in view of the regularity of $\boldsymbol{v}$ and $\boldsymbol{w}$, we readily infer from \eqref{est:delth:1} and \eqref{est:delth:2} that 
\begin{align}\label{vel:Xep}
    \sup_{h\in(0,1)}\norm{\scp{\delt^h\boldsymbol{v}}{\delt^h\boldsymbol{w}}}_{L^2(0,T;\mathcal{X}_\omega)} < +\infty.
\end{align}
Next, we intend to control the initial data. Arguing as for \eqref{est:ehr}, we see that
\begin{align}\label{initcond:ehrling}
    \begin{split}
        - \bigscp{\scp{\mu - \mu_0}{\theta - \theta_0}}{\scp{\phi - \phi_0}{\psi - \psi_0}}_{\mathcal{L}^2} 
        &\leq -\frac12 \norm{\scp{\phi - \phi_0}{\psi - \psi_0}}_{K,\alpha}^2 \\
        &\qquad + C\norm{\scp{\phi - \phi_0}{\psi - \psi_0}}_{L,\beta,\ast}^2.
    \end{split}
\end{align}
We now recall $\abs{\phi}\le 1$ a.e.~in $Q$ and $\abs{\psi}\le 1$ a.e.~on $\Sigma$.
Invoking the weak formulations \eqref{WF:PP:SING}-\eqref{WF:MT:SING} and condition \ref{cond:MT:0}, we use \eqref{initcond:ehrling} to derive the estimate
\begin{align*}
    \begin{split}
    &\ddt\frac12 \norm{\scp{\phi - \phi_0}{\psi - \psi_0}}_{L,\beta,\ast}^2 
    \\
    &\quad= - \bigscp{\scp{\mu}{\theta}}{\scp{\phi - \phi_0}{\psi - \psi_0}}_{\mathcal{L}^2} \\
    &\qquad - \bigscp{\phi\,\boldsymbol{v}}{\Grad\mathcal{S}^\Om_{L,\beta}(\phi - \phi_0, \psi - \psi_0)}_{L^2(\Omega)} 
    - \bigscp{\psi\,\boldsymbol{w}}{\Gradg\mathcal{S}^\Ga_{L,\beta}(\phi - \phi_0, \psi - \psi_0)}_{L^2(\Gamma)} 
    \end{split}
    \\
    \begin{split}
    &\quad= - \bigscp{\scp{\mu - \mu_0}{\theta - \theta_0}}{\scp{\phi - \phi_0}{\psi - \psi_0}}_{\mathcal{L}^2} 
    + \bigscp{\scp{\mu_0}{\theta_0}}{\scp{\phi - \phi_0}{\psi - \psi_0}}_{\mathcal{L}^2} \\
    &\qquad - \bigscp{\phi\,\boldsymbol{v}}{\Grad\mathcal{S}^\Om_{L,\beta}(\phi - \phi_0, \psi - \psi_0)}_{L^2(\Omega)} 
    - \bigscp{\psi\,\boldsymbol{w}}{\Gradg\mathcal{S}^\Ga_{L,\beta}(\phi - \phi_0, \psi - \psi_0)}_{L^2(\Gamma)} 
    \end{split}
    \\
    \begin{split}
    &\quad\leq -\frac12 \norm{\scp{\phi - \phi_0}{\psi - \psi_0}}_{K,\alpha}^2 +  C\norm{\scp{\phi - \phi_0}{\psi - \psi_0}}_{L,\beta,\ast}^2 \\
    &\qquad + \norm{\scp{\mu_0}{\theta_0}}_{L,\beta}\norm{\scp{\phi - \phi_0}{\psi - \psi_0}}_{L,\beta,\ast} \\
    &\qquad + \norm{\scp{\boldsymbol{v}}{\boldsymbol{w}}}_{\mathcal{L}^2}\norm{\scp{\phi - \phi_0}{\psi - \psi_0}}_{L,\beta,\ast},
    \end{split}
\end{align*}
which gives the differential inequality
\begin{align*}
    \begin{split}
         &\ddt\frac12 \norm{\scp{\phi - \phi_0}{\psi - \psi_0}}_{L,\beta,\ast}^2 \\
         &\qquad \leq C\Big( 1 + \norm{\scp{\mu_0}{\theta_0}}_{L,\beta} + \norm{\scp{\boldsymbol{v}}{\boldsymbol{w}}}_{\mathcal{L}^2}\Big)\norm{\scp{\phi - \phi_0}{\psi - \psi_0}}_{L,\beta,\ast}.
    \end{split}
\end{align*}
Now, a quadratic variant of Gronwall's lemma (see, e.g., \cite[Lemma A.5]{Brezis}) yields
\begin{align*}
    \norm{\scp{\phi(t) - \phi_0}{\psi(t) - \psi_0}}_{L,\beta,\ast} \leq C \big( 1 + \norm{\scp{\mu_0}{\theta_0}}_{L,\beta}\big) t 
    + C\int_0^t \norm{\scp{\boldsymbol{v}(s)}{\boldsymbol{w}(s)}}_{\mathcal{L}^2}\ds
\end{align*}
for almost all $t\in[0,T]$, which in turn entails
\begin{align}\label{Gronwall:ic}
    \begin{split}
        \norm{\scp{\delt^h\phi(0)}{\delt^h\psi(0)}}_{L,\beta,\ast} &\leq C \big( 1 + \norm{\scp{\mu_0}{\theta_0}}_{L,\beta}\big)  
        + \frac{C}{h}\int_0^h \norm{\scp{\boldsymbol{v}(s)}{\boldsymbol{w}(s)}}_{\mathcal{L}^2}\ds \\
        &\leq C\big( 1 + \norm{\scp{\mu_0}{\theta_0}}_{L,\beta} + \norm{\scp{\boldsymbol{v}}{\boldsymbol{w}}}_{L^\infty(0,T;\mathcal{L}^2)}\big)
    \end{split}
\end{align}
for all $h\in(0,1)$. Thus, applying Gronwall's lemma to \eqref{Gronwall:1}, and invoking \eqref{reg:Lambdah}, \eqref{vel:Xep} as well as \eqref{Gronwall:ic}, we deduce that
\begin{align}\label{Gronwall:2}
    \begin{split}
        \norm{\scp{\delt^h\phi(t)}{\delt^h\psi(t)}}_{L,\beta,\ast}^2 &\leq \norm{\scp{\delt^h\phi(0)}{\delt^h\psi(0)}}_{L,\beta,\ast}^2\exp\Big(C\int_0^t \Lambda_h\dtau\Big) \\
        &\qquad + C\int_0^t \norm{\scp{\delt^h\boldsymbol{v}}{\delt^h\boldsymbol{w}}}_{\mathcal{X}_\omega}^2\exp\Big(C\int_s^t \Lambda_h\dtau\Big)\ds \\
        &\leq C
    \end{split}
\end{align}
for almost all $t\in[0,T]$.
Therefore, owing to Ehrling's lemma, we conclude from \eqref{Gronwall:1} and \eqref{Gronwall:2} that
\begin{align}\label{est:delth}
    \norm{\scp{\delt^h\phi}{\delt^h\psi}}_{L^\infty(0,T;(\mathcal{H}^1_{L,\beta})^\prime)} + \norm{\scp{\delt^h\phi}{\delt^h\psi}}_{L^2(0,T;\mathcal{H}^1)} \leq C.
\end{align}
Hence, the time derivative $\scp{\delt\phi}{\delt\psi}$ exists in the weak sense and satisfies
\begin{align*}
    \scp{\delt\phi}{\delt\psi} \in L^\infty(0,T;(\mathcal{H}^1_{L,\beta})^\prime)
    \cap
    L^2(0,T;\mathcal{H}^1).
\end{align*}
As a consequence, we thus have
\begin{align}\label{hreg:pp}
    \scp{\phi}{\psi}\in W^{1,\infty}\big(0,T;(\mathcal{H}^1_{L,\beta})^\prime\big)\cap H^1(0,T;\mathcal{H}^1).
\end{align}
If $L\in[0,\infty)$, testing \eqref{WF:PP:SING} with $\scp{\mu - \beta\mean{\mu}{\theta}}{\theta - \mean{\mu}{\theta}} \in \mathcal{H}^1_{L,\beta}$, and employing the bulk-surface Poincar\'e inequality (see~\ref{PRELIM:POINCINEQ}) together with Young's inequality, we find that
\begin{align}
    \label{EST:MUTHLB}
    \norm{\scp{\mu}{\theta}}_{L,\beta}^2 \leq C\norm{\scp{\delt\phi}{\delt\psi}}_{(\mathcal{H}^1_{L,\beta})^\prime}^2 + C\norm{\scp{\boldsymbol{v}}{\boldsymbol{w}}}_{\mathcal{L}^2}^2.
\end{align}
Here, we used once more that $\abs{\phi}\le 1$ a.e.~in $Q$ and $\abs{\psi}\le 1$ a.e.~on $\Sigma$. In the case $L = \infty$, the estimate \eqref{EST:MUTHLB} can be derived similarly by choosing $\scp{\mu - \meano{\mu}}{\theta - \meang{\theta}}$ as the test function.

In view of the regularity \eqref{hreg:pp} established above, as well as the regularity of the velocity fields $\boldsymbol{v}$ and $\boldsymbol{w}$, we conclude
\begin{align}\label{REG:MT:LB}
    t\mapsto \norm{\scp{\mu(t)}{\theta(t)}}_{L,\beta}\in L^\infty(0,T),
\end{align}
and, in particular,
\begin{align}
    \label{REG:GMT:1}
    \scp{\Grad\mu}{\Gradg\theta}\in L^\infty(0,T;\mathcal{L}^2).
\end{align}

\pagebreak[1]

\textbf{Step 2:} By means of elliptic regularity theory we now show $\scp{\mu}{\theta}\in L^2(0,T;\mathcal{H}^2)$. 

In view of the regularities established in Step 1, and using that $\boldsymbol{v}$ and $\boldsymbol{w}$ are divergence-free in $\Om$ and on $\Ga$, respectively, we find that
\begin{align}\label{reg:mt:2}
    \begin{split}
        &\intO\Grad\mu(t)\cdot\Grad\zeta\dx + \intG\Gradg\theta(t)\cdot\Gradg\zeta_\Ga\dG 
        \\
        &\qquad + \sigma(L)\intG(\beta\theta(t) - \mu(t))(\beta\zeta_\Ga - \zeta)\dG 
        \\
        &\quad = -\intO\delt\phi(t)\zeta\dx - \intG\delt\psi(t)\zeta_\Ga\dG 
        \\
        &\qquad -\intO\Grad\phi(t)\cdot\boldsymbol{v}(t)\zeta\dx - \intG\Gradg\psi(t)\cdot\boldsymbol{w}(t)\zeta_\Ga\dG
    \end{split}
\end{align}
for almost all $t\in[0,T]$ and all $\scp{\zeta}{\zeta_\Ga}\in\mathcal{H}^1_{L,\beta}$. 
This means that the pair $\scp{\mu(t)}{\theta(t)}$ is a weak solution of the bulk-surface elliptic problem 
\begin{subequations}
    \begin{alignat*}{3}
        &-\Lap\mu(t) = -\delt\phi(t) - \Grad\phi(t)\cdot\boldsymbol{v}(t) 
        &&\quad\text{in } \Om, \\
        &-\Lapg\theta(t) + \beta\deln\mu(t) = -\delt\psi(t) - \Gradg\psi(t)\cdot\boldsymbol{w}(t) &&\quad\text{on } \Ga, \\
        &\begin{cases} 
        L \deln\mu(t) = \beta\theta(t) - \mu(t) &\text{if} \  L\in[0,\infty), \\
        \deln\mu(t) = 0 &\text{if} \ L=\infty
        \end{cases} &&\quad\text{on } \Ga
    \end{alignat*}
\end{subequations}
for almost all $t\in[0,T]$. Applying regularity theory for elliptic problems with bulk-surface coupling (see Proposition~\ref{Prop:Appendix} with $p=2$), we find that
\begin{align*}
    \norm{\scp{\mu(t)}{\theta(t)}}_{\mathcal{H}^2} &\leq C\norm{\scp{-\delt\phi(t) - \Grad\phi(t)\cdot\boldsymbol{v}(t)}{-\delt\psi(t) - \Gradg\psi(t)\cdot\boldsymbol{w}(t)}}_{\mathcal{L}^2} \\
    &\leq C\norm{\scp{\delt\phi(t)}{\delt\psi(t)}}_{\mathcal{L}^2} + C\norm{\scp{\Grad\phi(t)}{\Gradg\psi(t)}}_{\mathcal{L}^\infty}\norm{\scp{\boldsymbol{v}(t)}{\boldsymbol{w}(t)}}_{\mathcal{L}^2} \\
    &\leq C\norm{\scp{\delt\phi(t)}{\delt\psi(t)}}_{\mathcal{L}^2} + C\norm{\scp{\phi(t)}{\psi(t)}}_{\mathcal{W}^{1,\infty}}\norm{\scp{\boldsymbol{v}(t)}{\boldsymbol{w}(t)}}_{\mathcal{L}^2}
\end{align*}
for almost all $t\in[0,T]$. Squaring and integrating this inequality in time from $0$ to $T$ yields $\scp{\mu}{\theta}\in L^2(0,T;\mathcal{H}^2)$.

\textbf{Step 3:} We next show $\scp{\phi}{\psi}\in L^\infty(0,T;\mathcal{W}^{2,6})\cap \big( C(\overline{Q})\times C(\overline{\Sigma})\big)$, $\scp{\mu}{\theta}\in L^\infty(0,T;\mathcal{H}^1)$ and $(F^\prime(\phi), G^\prime(\psi))\in L^\infty(0,T;\mathcal{L}^6)$.

To prove these regularities, we restrict ourselves to the case $L\in[0,\infty)$. If $L = \infty$, we can argue analogously (with obvious modifications).

First, note that, in view of Theorem~\ref{THEOREM:EOWS}, we have
\begin{subequations}\label{STRG:MT}
    \begin{align}
         \label{STRG:MU}
        &\mu = -\Lap\phi + F'(\phi) &&\text{a.e.~in } Q, \\
         \label{STRG:THETA}
        &\theta = -\Lapg\psi + G'(\psi) + \alpha\deln\phi &&\text{a.e.~on } \Sigma, \\
         \label{STRG:PHIPSI}
        & \begin{cases} 
            K\deln\phi = \alpha\psi - \phi &\text{if} \ K\in [0,\infty), \\
            \deln\phi = 0 &\text{if} \ K = \infty
        \end{cases} &&  \text{a.e.~on } \Sigma.
    \end{align}
\end{subequations}
Due to regularity theory for bulk-surface elliptic systems (see Proposition~\ref{Prop:Appendix} with $p=6$), we further have the estimate
\begin{align}\label{PP:H^2:EST}
    \norm{\scp{\phi(t)}{\psi(t)}}_{\mathcal{W}^{2,6}} \leq C\big( \norm{\scp{\mu(t)}{\theta(t)}}_{\mathcal{L}^6} + \norm{\scp{F^\prime(\phi(t))}{G^\prime(\psi(t)}}_{\mathcal{L}^6}\big)
\end{align}
for almost all $t\in[0,T]$. Our goal is to estimate the right-hand side of \eqref{PP:H^2:EST}. To do so, we first notice that by the Miranville--Zelik inequality (see \cite[Appendix A.1]{Miranville2004} or \cite[p.908]{Gilardi2009}), there exist positive constants $c_1,c_2$ and a non-negative constant $c_3$ such that
\begin{align}
    \begin{split}\label{EST:MZ:2}
        &c_1\norm{F_1^\prime(\phi)}_{L^1(\Om)} + c_2\norm{G_1^\prime(\psi)}_{L^1(\Ga)} - c_3 \\
        &\quad\leq \intO F_1^\prime(\phi)(\phi - \beta\mean{\phi}{\psi})\dx + \intG G_1^\prime(\psi)(\psi - \mean{\phi}{\psi})\dG.
    \end{split}
\end{align}
Next, we multiply \eqref{STRG:MU} and \eqref{STRG:THETA} by $\big(\beta^2\abs{\Om} + \abs{\Ga}\big)^{-1}\beta$ and $\big(\beta^2\abs{\Om} + \abs{\Ga}\big)^{-1}$, respectively, and employ integration by parts. This yields
\begin{align}\label{EST:MEANMT}
    \abs{\mean{\mu}{\theta}} \leq C\big( 1 + \norm{F_1^\prime(\phi)}_{L^1(\Om)} + \norm{G_1^\prime(\psi)}_{L^1(\Ga)}\big).
\end{align}
Furthermore, multiplying \eqref{STRG:MU} and \eqref{STRG:THETA} by $\phi - \beta\mean{\phi}{\psi}$ and $\psi - \mean{\phi}{\psi}$, respectively, using integration by parts as well as the bulk-surface Poincar\'{e} inequality, we infer
\begin{align}\label{EST:FG:MEAN}
    \begin{split}
        &\intO F_1^\prime(\phi)(\phi - \beta\mean{\phi}{\psi})\dx + \intG G_1^\prime(\psi)(\psi - \mean{\phi}{\psi})\dG \\
        &\quad= - \intO F_2^\prime(\phi)(\phi - \beta\mean{\phi}{\psi})\dx - \intG G_2^\prime(\psi)(\psi - \mean{\phi}{\psi}) \dG \\
        &\qquad + \intO (\mu - \beta\mean{\mu}{\theta})(\phi - \beta\mean{\phi}{\psi})\dx 
        - \intO \abs{\Grad\phi}^2 \dx
        \\
        &\qquad+ \intG (\theta - \mean{\mu}{\theta})(\psi - \mean{\phi}{\psi})\dG 
        - \intG \abs{\Gradg\psi}^2 \dG
        \\
        &\qquad - \intG \deln\phi 
        \big(\alpha\psi - \phi - (\alpha-\beta)\mean{\phi}{\psi} \big)
        \dG
        \\[1ex]
        &\quad\leq C\big( 1 + \norm{(\mu,\theta)}_{L,\beta} + \gamma(K)\norm{\deln\phi}_{L^2(\Ga)}\big),
    \end{split}
\end{align}
where $\gamma(K) = \mathbf{1}_{\{0\}}(K)$ denotes again the characteristic function of the set $\{0\}$.
This yields
\begin{align}\label{Est:FG:L^1:K}
    \norm{F_1^\prime(\phi)}_{L^1(\Om)} + \norm{G_1^\prime(\psi)}_{L^1(\Ga)} \leq C\big(1 + \norm{(\mu,\theta)}_{L,\beta} + \gamma(K)\norm{\deln\phi}_{L^2(\Ga)}\big)
\end{align}
a.e. on $[0,T]$. Hence, in view of \eqref{EST:MEANMT}, we obtain
\begin{align}\label{Est:Mean:mt:K}
    \abs{\mean{\mu}{\theta}} \leq C\big(1 + \norm{(\mu,\theta)}_{L,\beta} + \gamma(K)\norm{\deln\phi}_{L^2(\Ga)}\big).
\end{align} 
Then, returning to the proof of Proposition~\ref{Prop:FG:Linf}, in particular \eqref{Est:FG:L^r:K>0} in the case $K\in(0,\infty]$ and \eqref{Est:FG:L^r:K=0} in the case $K = 0$, we find that
\begin{align}\label{Est:FG:L^r:K}
    \begin{split}
        \norm{(F_1^\prime(\phi),G_1^\prime(\psi))}_{\mathcal{L}^r} &\leq C\big(1 + \norm{(\mu,\theta)}_{\mathcal{L}^r} + \gamma(K)\norm{\deln\phi}_{L^r(\Ga)}\big) \\
        &\leq C\big( 1 + \norm{(\mu,\theta)}_{L,\beta} + \gamma(K)\norm{\deln\phi}_{L^r(\Ga)}\big)
    \end{split}
\end{align}
for any $r\in[2,6]$. Here, we made again use of the bulk-surface Poincar\'{e} inequality. 

If $K\in(0,\infty]$, we readily deduce that $(F_1^\prime(\phi), G_1^\prime(\psi))\in L^\infty(0,T;\mathcal{L}^6)$ as well as $\abs{\mean{\mu}{\theta}}\in L^\infty(0,T)$. Thus, by means of the bulk-surface Poincar\'{e} inequality, we get
\begin{align}
    \norm{(\mu,\theta)}_{\mathcal{L}^2} \leq C\big(1 + \norm{(\mu,\theta)}_{L,\beta}\big).
\end{align}
In combination with \eqref{REG:MT:LB} and \eqref{REG:GMT:1}, this proves $(\mu,\theta)\in L^\infty(0,T;\mathcal{H}^1)$. Then, we readily deduce from \eqref{PP:H^2:EST} that
\begin{align*}
    (\phi,\psi)\in L^\infty(0,T;\mathcal{W}^{2,6}).
\end{align*}

If $K = 0$, we additionally need to take care of the normal derivative $\deln\phi$. To this end, choosing $r = 2$ in \eqref{Est:FG:L^r:K} and exploiting the estimate \eqref{Est:Mean:mt:K}, elliptic regularity theory yields that
\begin{align}
    \begin{split}
        \norm{(\phi,\psi)}_{\mathcal{H}^2} &\leq C\big(\norm{(\mu,\theta)}_{\mathcal{L}^2} + \norm{(F^\prime(\phi),G^\prime(\psi))}_{\mathcal{L}^2}\big) \\
        &\leq C\big(1 + \norm{(\mu,\theta)}_{L,\beta} + \norm{\deln\phi}_{L^2(\Ga)}\big).
    \end{split}
\end{align}
Then, employing again the trace theorem as in Step 6 of the proof of Theorem~\ref{THEOREM:EOWS} (cf.~\eqref{est:trace}) and using Ehrling's lemma along with an absorption argument, we obtain
\begin{align}
    \norm{(\phi,\psi)}_{\mathcal{H}^2} \leq C\big( 1 + \norm{(\mu,\theta)}_{L,\beta}\big).
\end{align}
In view of \eqref{REG:MT:LB}, this yields $(\phi,\psi)\in L^\infty(0,T;\mathcal{H}^2)$. In particular, we thus have $\deln\phi\in L^\infty(0,T;L^2(\Ga))$. 
Due to \eqref{Est:FG:L^1:K} and \eqref{Est:Mean:mt:K}, we infer $(F_1^\prime(\phi),G_1^\prime(\psi))\in L^\infty(0,T;\mathcal{L}^2)$ as well as $\abs{\mean{\mu}{\theta}}\in L^\infty(0,T)$. 
In combination with \eqref{REG:MT:LB} and \eqref{REG:GMT:1}, this proves that $(\mu,\theta)\in L^\infty(0,T;\mathcal{H}^1)$. 
Next, by the trace embedding $H^1(\Om)\emb L^4(\Ga)$, we deduce that $\deln\phi\in L^\infty(0,T;L^4(\Ga))$. Thus, choosing $r = 4$ in \eqref{Est:FG:L^r:K}, we find that $(F_1^\prime(\phi), G_1^\prime(\psi))\in L^\infty(0,T;\mathcal{L}^4)$, which in turn provides $(\phi,\psi)\in L^\infty(0,T;\mathcal{W}^{2,4})$. Finally, using the continuous embeddings $W^{1,4}(\Om)\emb W^{3/4,4}(\Ga)\emb L^\infty(\Ga)$, we have $\deln\phi\in L^\infty(0,T;L^\infty(\Ga))$. Therefore, choosing $r=6$ in \eqref{Est:FG:L^r:K}, we deduce $(F_1^\prime(\phi), G_1^\prime(\psi))\in L^\infty(0,T;\mathcal{L}^6)$, from which we finally conclude $(\phi,\psi)\in L^\infty(0,T;\mathcal{W}^{2,6})$.

Lastly, as we know from Step~1 that $\phi\in H^1(0,T;H^1(\Om))$, we use the compact embedding $H^2(\Om)\emb H^{7/4}(\Om)$, the continuous embeddings $H^{7/4}(\Om)\emb H^1(\Om)$ and $H^{7/4}(\Om)\emb C(\overline{\Om})$, and the Aubin--Lions--Simon lemma to deduce that
\begin{align*}
    \phi\in C([0,T];H^{7/4}(\Om)) \emb C([0,T];C(\overline{\Om})) \cong C(\overline{Q}).
\end{align*}
Analogously, we infer $\psi\in C(\overline{\Sigma})$. 

This means that all claims are established, apart from the regularities $F^\prime(\phi)\in L^2(0,T;L^\infty(\Om))$ and $G^\prime(\psi)\in L^2(0,T;L^\infty(\Ga))$. The proof of these statements can be found in Section~\ref{SECT:SEPARATION}, see Proposition~\ref{prop:highreg:G} and Proposition~\ref{prop:highreg:F}.
\end{proof}

\section{Separation properties}
\label{SECT:SEPARATION}
This chapter is devoted to the proof of Theorem~\ref{thm:sepprop}, which states the separation properties. Before presenting the proof of Theorem~\ref{thm:sepprop}, we first establish two important propositions.

\begin{proposition}\label{prop:highreg:G}
    Under the assumptions made in Theorem~\ref{thm:highreg} it holds that
    \begin{align}\label{reg:G:2-inf}
        G^\prime(\psi)\in L^2(0,T;L^\infty(\Ga)).
    \end{align}
\end{proposition}

\medskip

\begin{proof}
Let $k\in\N$ and consider as in the proof of Proposition~\ref{Prop:FG:Linf} the truncation $\psi_k\coloneqq\sigma_k\circ\psi$, where $\sigma_k$ is defined in \eqref{sigmak}. For $s\geq 2$, we already showed that 
\begin{align*}
    \norm{G_1^\prime(\psi_k)}_{L^s(\Ga)}^s \leq C\big(1 + \norm{\theta}_{L^s(\Ga)} + d_G\norm{\psi}_{L^s(\Ga)} + \gamma(K)\norm{\deln\phi}_{L^s(\Ga)}\big)
    \norm{G_1^\prime(\psi_k)}_{L^s(\Ga)}^{s-1},
\end{align*}
where $d_G$ was a constant only depending on $G$, cf.~\eqref{est:Gk:r:pre:1} and \eqref{est:Gk:r:pre:K=0}. Now, invoking Hölder's inequality, we obtain
\begin{align}\label{est:Gk:r:pre:2}
    \begin{split}
    \norm{G_1^\prime(\psi_k)}_{L^s(\Ga)}^s \leq C (1 + \abs{\Ga})&\Big(1 + \norm{\theta}_{L^\infty(\Ga)} + d_G\norm{\psi}_{L^\infty(\Ga)} 
    \\
    &\quad+ \gamma(K)\norm{\deln\phi}_{L^\infty(\Ga)}\Big)
    \norm{G_1^\prime(\psi_k)}_{L^s(\Ga)}^{s-1}
    \end{split}
\end{align}
almost everywhere on $[0,T]$. As before, $\gamma(K) = \mathbf{1}_{\{0\}}(K)$ denotes the characteristic function of the set $\{0\}$.
Since $\scp{\phi}{\psi}\in L^\infty(0,T;\mathcal{W}^{2,6})$ and $\theta\in L^2(0,T;H^2(\Ga))$, we use the embeddings $H^2(\Ga)\emb L^\infty(\Ga)$ and $W^{1,6}(\Omega)\emb W^{5/6,6}(\Gamma)\emb L^\infty(\Gamma)$ to infer that
\begin{align*}
    \begin{split}
        &(1 + \abs{\Ga})\big(1 + \norm{\theta}_{L^\infty(\Ga)} + d_G\norm{\psi}_{L^\infty(\Ga)} + \gamma(K)\norm{\deln\phi}_{L^\infty(\Ga)}\big) \\
        &\quad\leq C(1 + \abs{\Ga})\big(1 + \norm{\theta}_{H^2(\Ga)} + d_G\norm{\psi}_{H^2(\Ga)} + \gamma(K)\norm{\phi}_{W^{2,6}(\Omega)}\big) \\
        &\quad\leq C\big(1 + \norm{\theta}_{H^2(\Ga)} \big)
        \eqqcolon c_\Ga(t)
    \end{split}
\end{align*}
for almost all $t\in[0,T]$. Using this estimate to bound the right-hand side of \eqref{est:Gk:r:pre:2}, we conclude
\begin{align}\label{est:Gk:inf}
    \norm{G_1^\prime(\psi_k(t))}_{L^\infty(\Ga)} \leq c_\Ga(t) < \infty
\end{align}
for almost all $t\in[0,T]$. 
Furthermore, it is straightforward to verify that $\psi_k(t)\rightarrow\psi(t)$ a.e.~on $\Ga$ for almost all $t\in[0,T]$ as $k\rightarrow\infty$. 
We also recall that $\abs{\psi_k(t)} < 1$ and $\abs{\psi(t)} < 1$ a.e.~on $\Ga$ for almost all $t\in[0,T]$. 
Consequently, we have
\begin{align}\label{conv:Gk:K}
    G_1^\prime(\psi_k(t))\rightarrow G_1^\prime(\psi(t))\quad\text{a.e.~on~} \Ga \ \text{as~} k\rightarrow\infty
\end{align}
for almost all $t\in[0,T]$. In combination with \eqref{est:Gk:inf}, this proves
\begin{align}\label{est:G:inf}
    \norm{G_1^\prime(\psi(t))}_{L^\infty(\Ga)} \leq c_\Ga(t) < \infty
\end{align}
for almost every $t\in[0,T]$. Finally, taking the square on both sides of \eqref{est:G:inf} and integrating with respect to time from $0$ to $T$, we eventually conclude in view of the regularities $\psi, \theta\in L^2(0,T;H^2(\Ga))$ that $G_1^\prime(\psi)\in L^2(0,T;L^\infty(\Ga))$. As an immediate consequence, we infer $G^\prime(\psi)\in L^2(0,T;L^\infty(\Ga))$.
\end{proof}

\medskip

\begin{proposition}\label{prop:highreg:F}
    Under the assumptions made in Theorem~\ref{thm:highreg} it holds that
    \begin{align}\label{reg:FG:2-inf}
        F^\prime(\phi)\in L^2(0,T;L^\infty(\Om)).
    \end{align}
\end{proposition}

\begin{proof}
    To prove the desired regularity, we handle the cases $K\in(0,\infty]$ and $K=0$ separately. Without loss of generality, we assume $\alpha\neq 0$. The case $\alpha=0$ can be handled similarly and the computations are even easier.
    As in the proof of Proposition~\ref{Prop:FG:Linf}, we consider for $k\in\N$ the functions $\phi_k \coloneqq\alpha\sigma_k\circ(\alpha^{-1}\phi)$ and $\psi_k\coloneqq\sigma_k\circ\psi$, where $\sigma_k$ is defined as in \eqref{sigmak}.

    \textit{The case $K\in(0,\infty]$.}
    For $r\geq 2$ and $k\in\N$, we already showed that 
    \begin{align}\label{est:FG:Lr:2}
        &\intO \abs{F_1^\prime(\phi_k)}^r\dx + \intG \abs{G_1^\prime(\psi_k)}^r\dG \nonumber \\
        &\quad \leq \intO (\abs{\mu} + d_F \abs{\phi})\abs{F_1^\prime(\psi_k)}^{r-1}\dx + \intG (\abs{\theta} + d_G\abs{\psi})\abs{G_1^\prime(\psi_k)}^{r-1}\dG \\
        &\qquad - \sigma(K)\intG (\alpha\psi - \phi)\Big(\alpha G_1^\prime(\psi_k)\abs{G_1^\prime(\psi_k)}^{r-2} - F_1^\prime(\alpha\psi_k)\abs{F_1^\prime(\alpha\psi_k)}^{r-2}\Big)\dG, \nonumber
    \end{align}
    cf.~\eqref{testeq}-\eqref{est:I5}. Here, $d_F$ and $d_G$ are some constants depending only on $F$ and $G$, respectively. Using the regularity $G^\prime(\psi)\in L^2(0,T;L^\infty(\Ga))$, which was already proven in Proposition~\ref{prop:highreg:G}, along with the domination property \eqref{domination}, we obtain for the last term on the right-hand side of \eqref{est:FG:Lr:2} the estimate
    \begin{align}\label{est:I5:2}
        &-\sigma(K)\intG (\alpha\psi - \phi)\Big(\alpha G_1^\prime(\psi_k)\abs{G_1^\prime(\psi_k)}^{r-2} - F_1^\prime(\alpha\psi_k)\abs{F_1^\prime(\alpha\psi_k)}^{r-2}\Big)\dG \\
        &\quad\leq 2\sigma(K)\abs{\Ga}\norm{G_1^\prime(\psi_k)}_{L^\infty(\Ga)}^{r-1} + 2^{r-1}\kappa_1^{r-1}\sigma(K)\abs{\Ga}\norm{G_1^\prime(\psi_k)}_{L^\infty(\Ga)}^{r-1} + 2^{r-1}\kappa_2^{r-1}\sigma(K)\abs{\Ga}. \nonumber
    \end{align}
    Thus, by combining the above estimates \eqref{est:FG:Lr:2} and \eqref{est:I5:2}, we deduce by applying Hölder's and Young's inequality
    \begin{align*}
        \begin{split}
            &\norm{F_1^\prime(\phi_k)}_{L^r(\Om)}^r + \norm{G_1^\prime(\psi_k)}_{L^r(\Ga)}^r
            \\
            &\leq\big(\norm{\mu}_{L^r(\Om)} + d_F\norm{\phi}_{L^r(\Om)}\big)\norm{F_1^\prime(\phi_k)}_{L^r(\Om)}^{r-1} 
            + \big(\norm{\theta}_{L^r(\Ga)} + d_G\norm{\psi}_{L^r(\Ga)}\big)\norm{G_1^\prime(\psi_k)}_{L^r(\Ga)}^{r-1} \\
            &\quad + 2\sigma(K)\abs{\Ga}\Big(\big(1 + 2^{r-2}\kappa_1^{r-1}\big)\norm{G_1^\prime(\psi_k)}_{L^\infty(\Ga)}^{r-1} + 2^{r-2}\kappa^{r-1}_2\Big) \\
            &\leq \frac12\norm{F_1^\prime(\phi_k)}_{L^r(\Om)}^r + \frac12\norm{G_1^\prime(\psi_k)}_{L^r(\Ga)}^r  \\
            &\quad + \frac1r\left(\frac{2(r-1)}{r}\right)^{r-1}\Big(\norm{\mu}_{L^r(\Om)} + \norm{\theta}_{L^r(\Ga)} + d_F\norm{\phi}_{L^r(\Om)} + d_G\norm{\psi}_{L^r(\Ga)}\Big)^r \\
            &\quad + 2\sigma(K)\abs{\Ga}\Big(\big(1 + 2^{r-2}\kappa_1^{r-1}\big)\norm{G_1^\prime(\psi_k)}_{L^\infty(\Ga)}^{r-1} + 2^{r-2}\kappa^{r-1}_2\Big)
        \end{split}
    \end{align*}
    almost everywhere on $[0,T]$. Thus, absorbing the respective terms and taking the $r$-th root, we find
    \begin{align}\label{est:FG:r:final}
            &\norm{F_1^\prime(\phi_k)}_{L^r(\Om)} + \norm{G_1^\prime(\psi_k)}_{L^r(\Ga)} \nonumber \\
            &\leq 2\cdot2^{1/r}\Bigg[\frac{1}{r^{1/r}}\left(\frac{2(r-1)}{r}\right)^{(r-1)/r}\Big(\norm{\mu}_{L^r(\Om)} + \norm{\theta}_{L^r(\Ga)} + d_F\norm{\phi}_{L^r(\Om)} + d_G\norm{\psi}_{L^r(\Ga)}\Big) \nonumber \\
            &\quad + (2\sigma(K)\abs{\Ga})^{1/r}\Big(\big(1 + 2^{r-2}\kappa_1^{r-1}\big)^{1/r}\norm{G_1^\prime(\psi_k)}_{L^\infty(\Ga)}^{(r-1)/r} + 2^{(r-2)/r}\kappa^{(r-1)/r}_2\Big)\Bigg] \\
            &\leq C\Big(1 + \norm{\mu}_{L^r(\Om)} + \norm{\theta}_{L^r(\Ga)} + d_F\norm{\phi}_{L^r(\Om)} + d_G\norm{\psi}_{L^r(\Ga)}\Big)\nonumber
    \end{align}
    almost everywhere on $[0,T]$ for all $r\in[2,\infty)$ and a constant $C>0$ independent of $r$. Next, since $\scp{\phi}{\psi}, \scp{\mu}{\theta}\in L^2(0,T;\mathcal{H}^2)$ and $\mathcal{H}^2\emb\mathcal{L}^\infty$ we find
    \begin{align}\label{est:FkGk:inf}
        \begin{split}
            &\bigabs{F_1^\prime(\phi_k(x,t))} + \bigabs{G_1^\prime(\psi_k(z,t))} \\
            &\leq \norm{F_1^\prime(\phi_k(t))}_{L^\infty(\Om)} + \norm{G_1^\prime(\psi_k(t))}_{L^\infty(\Ga)} \\
            &\leq C\Big(1 + \norm{\mu}_{H^2(\Om)} + \norm{\theta}_{H^2(\Ga)} + d_F\norm{\phi}_{H^2(\Om)} + d_G\norm{\psi}_{H^2(\Ga)}\Big) \\
            &\leq C\big(1 + \norm{\scp{\mu(t)}{\theta(t)}}_{\mathcal{H}^2} \big) 
        \end{split}
    \end{align}
    for almost all $x\in\Omega$, $z\in\Gamma$ and $t\in[0,T]$. 
    Furthermore, it is straightforward to verify that $\phi_k(t)\rightarrow\phi(t)$ a.e.~in $\Om$ for almost all $t\in[0,T]$. We also recall that $\abs{\phi_k(t)} < 1$ and $\abs{\phi(t)} < 1$ a.e.~on $\Om$ for almost all $t\in[0,T]$. Consequently,
    \begin{align}\label{conv:Fk}
        F_1^\prime(\phi_k(t))\rightarrow F_1^\prime(\phi(t))\quad\text{a.e.~in~} \Om \ \text{as~} k\rightarrow\infty
    \end{align}
    for almost all $t\in[0,T]$. Analogously, we infer 
    \begin{align}\label{conv:Gk}
        G_1^\prime(\psi_k(t))\rightarrow G_1^\prime(\psi(t))\quad\text{a.e.~on~} \Ga \ \text{as~} k\rightarrow\infty
    \end{align}
    for almost all $t\in[0,T]$. In combination with \eqref{est:FkGk:inf}, this proves
    \begin{align}\label{est:FG:inf*}
        \bigabs{F_1^\prime(\phi(x,t))} + \bigabs{G_1^\prime(\psi(z,t))} 
        \leq 
        C\big(1 + \norm{\scp{\mu(t)}{\theta(t)}}_{\mathcal{H}^2} \big) < \infty,
    \end{align}
    for almost all $x\in\Omega$, $z\in\Gamma$ and $t\in[0,T]$, which directly implies 
    \begin{align}\label{est:FG:inf}
        \norm{F_1^\prime(\phi(t))}_{L^\infty(\Om)} + \norm{G_1^\prime(\psi(t))}_{L^\infty(\Ga)} \leq C\big(1 + \norm{\scp{\mu(t)}{\theta(t)}}_{\mathcal{H}^2} \big) < \infty
    \end{align}
    for almost all $t\in[0,T]$. Finally, taking the square on both sides of \eqref{est:FG:inf} and integrating with respect to time from $0$ to $T$, we eventually conclude in view of the regularities $\scp{\phi}{\psi}, \scp{\mu}{\theta}\in L^2(0,T;\mathcal{H}^2)$ that $\scp{F_1^\prime(\phi)}{G_1^\prime(\psi)}\in L^2(0,T;\mathcal{L}^\infty)$. As an immediate consequence, we infer $\scp{F^\prime(\phi)}{G^\prime(\psi)}\in L^2(0,T;\mathcal{L}^\infty)$.

    \textit{The case $K=0$.}
    For fixed $r\geq 2$ and $k\in\N$, we test the equation
    \begin{align*}
        \mu = -\Lap\phi + F^\prime(\phi) \qquad\text{a.e. in~}Q
    \end{align*}
    once more with $\eta_k \coloneqq \abs{F_1^\prime(\phi_k)}^{r-2}F_1^\prime(\phi_k)$. 
    The crucial difference compared to the case $K\in(0,\infty]$ is that we now have the trace relation $\phi_k=\alpha\psi_k$ a.e.~on $\Sigma$ at hand.
    Integrating the resulting equation and performing integration by parts, we thus deduce that
    \begin{align*}
        \intO \abs{F_1^\prime(\phi_k)}^r\dx &\leq \intO \big(\abs{\mu} + d_F\abs{\phi}\big)\abs{F_1^\prime(\phi_k)}^{r-1}\dx \\
        &\quad - \intG \deln\phi \,F_1^\prime(\alpha\psi_k)\abs{F_1^\prime(\alpha\psi_k)}^{r-2}\dG.
    \end{align*}
    To control the last term involving the normal derivative $\deln\phi$, we employ the domination property \eqref{domination} together with the regularities $\deln\phi\in L^\infty(0,T;L^\infty(\Ga))$ and $G_1^\prime(\psi_k)\in L^2(0,T;L^\infty(\Ga))$ to obtain
    \begin{align*}
        &- \intG \deln\phi F_1^\prime(\alpha\psi_k)\abs{F_1^\prime(\alpha\psi_k)}^{r-2}\dG \\
        &\quad\leq 2^{r-2}\intG \abs{\deln\phi}\big(\kappa_1^{r-1}\abs{G_1^\prime(\psi_k)}^{r-1} + \kappa_2^{r-1}\big)\dG \\
        &\quad\leq 2^{r-2}\abs{\Ga}\norm{\deln\phi}_{L^\infty(\Ga)}\big(\kappa_1^{r-1}\norm{G_1^\prime(\psi_k)}_{L^\infty(\Ga)}^{r-1} + \kappa_2^{r-1}\big).
    \end{align*}
    This yields
    \begin{align*}
        \norm{F_1^\prime(\phi_k)}_{L^r(\Om)}^r &\leq \Big(\norm{\mu}_{L^r(\Om)} + d_F\norm{\phi}_{L^r(\Om)}\Big)\norm{F_1^\prime(\phi_k)}_{L^r(\Om)}^{r-1} \\
        &\quad + 2^{r-2}\abs{\Ga}\norm{\deln\phi}_{L^\infty(\Ga)}\Big(\kappa_1^{r-1}\norm{G_1^\prime(\psi_k)}_{L^\infty(\Ga)}^{r-1} + \kappa_2^{r-1}\Big) \\
        &\leq \frac12\norm{F_1^\prime(\phi_k)}_{L^r(\Om)}^r + \frac1r\left(\frac{2(r-1)}{r}\right)^{r-1}\Big(\norm{\mu}_{L^r(\Om)} + d_F\norm{\phi}_{L^r(\Om)}\Big)^r \\
        &\quad + 2^{r-2}\abs{\Ga}\norm{\deln\phi}_{L^\infty(\Ga)}\Big(\kappa_1^{r-1}\norm{G_1^\prime(\psi_k)}_{L^\infty(\Ga)}^{r-1} + \kappa_2^{r-1}\Big)
    \end{align*}
    almost everywhere on $[0,T]$. Proceeding as in the case $K\in(0,\infty]$ and invoking \eqref{est:G:inf} from the proof of Proposition~\ref{prop:highreg:G}, we infer 
    \begin{align}\label{est:FG:inf:0}
        \norm{F_1^\prime(\phi(t))}_{L^\infty(\Om)}
        + \norm{G_1^\prime(\phi(t))}_{L^\infty(\Ga)}
        \leq C\big(1 + \norm{\scp{\mu(t)}{\theta(t)}}_{\mathcal{H}^2} \big) 
        <\infty
    \end{align}
    for almost all $t\in[0,T]$. From this estimate, we finally conclude $\scp{F^\prime(\phi)}{G^\prime(\psi)}\in L^2(0,T;\mathcal{L}^\infty)$ and thus, the proof is complete.
\end{proof}

\medskip

\begin{proof}[Proof of Theorem~\ref{thm:sepprop}]
    The proof of part \ref{thm:sepprop:a} simply follows by using the regularity $F^\prime(\phi)\in L^2(0,T;L^\infty(\Om))$ and $G^\prime(\psi)\in L^2(0,T;L^\infty(\Ga))$ established in Proposition~\ref{prop:highreg:G} as well as Proposition~\ref{prop:highreg:F}. 
    In particular, due to \eqref{est:FG:inf}, we can choose
    \begin{align*}
        \delta(t) \coloneqq 1 - \max\Big\{ \big(F_1'\big)^{-1}\big(c(t)\big) \,{,}\, \big(G_1'\big)^{-1}\big(c(t)\big)\Big\}
    \end{align*}
    with 
    \begin{equation*}
        c(t)\coloneq C\big(1 + \norm{\scp{\mu(t)}{\theta(t)}}_{\mathcal{H}^2} \big).
    \end{equation*}
    for almost all $t\in [0,T]$, where $C$ is chosen as in
    \eqref{est:FG:inf} if $K\in(0,\infty]$ or \eqref{est:FG:inf:0} if $K=0$.
    
    Since the proofs of \ref{thm:sepprop:b} and \ref{thm:sepprop:c} work similarly (with obvious modifications), we will only show the latter.    
    First, recall that from \ref{S5} there exist constants $C_1, C_2 > 0$ and $\lambda\in[1,2)$ such that
    \begin{align}\label{est:F''}
        \abs{F^{\prime\prime}(s)}\leq C_1 e^{C_2\abs{F_1^\prime(s)}^\lambda}\quad\text{for all~}s\in(-1,1).
    \end{align}
    Next, since $d=2$, we know that for all $u\in H^1(\Om)$ and $r\in[2,\infty)$ it holds
    \begin{align}
        \norm{u}_{L^r(\Om)}\leq C_\Om\sqrt{r}\norm{u}_{H^1(\Om)}
    \end{align}
    for some constant $C_\Om$ depending only on $\Om$ (see, e.g., \cite[p. 479]{Trudinger1967}). Consequently, we infer from \eqref{est:FG:r:final} that
    \begin{align}\label{est:Lr:r}
        \begin{split}
            \norm{F_1^\prime(\phi)}_{L^r(\Om)} &\leq C\sqrt{r}\Big(\norm{\scp{\mu}{\theta}}_{L^\infty(0,T;\mathcal{H}^1)} + \norm{\scp{\phi}{\psi}}_{L^\infty(0,T;\mathcal{H}^1)}\Big)\\
            &\leq C\sqrt{r}
        \end{split}
    \end{align}
    almost everywhere on $[0,T]$ for all $r\in[2,\infty)$. This inequality allows us to argue similarly to the proof of \cite[Theorem 3.1]{Gal2023b}, and we deduce that
    \begin{align}\label{est:exp}
        \intO\abs{e^{C_2\abs{F_1^\prime(\phi)}^\lambda}}^m\dx \leq C(m)
    \end{align}
    for all $m\in\N$ and some constant $C(m)$ depending on $m$ and the constant $C$ from \eqref{est:Lr:r}. In view of \eqref{est:F''}, inequality \eqref{est:exp} with $m=6$ immediately entails that
    \begin{align}\label{est:F'':inf:6}
        \norm{F_1^{\prime\prime}(\phi)}_{L^\infty(0,T;L^6(\Om))} \leq C.
    \end{align}
    Thus, utilizing the chain rule as well as the continuous embedding $H^2(\Om)\emb W^{1,6}(\Om)$ we obtain
    \begin{align}\label{est:F':cr}
        \begin{split}
            \norm{\Grad F_1^\prime(\phi)}_{L^\infty(0,T;\mathbf{L}^3(\Om))} &\leq \norm{F_1^{\prime\prime}(\phi)}_{L^\infty(0,T;L^6(\Om))}\norm{\Grad\phi}_{L^\infty(0,T;\mathbf{L}^6(\Om))} \\
            &\leq C\norm{\phi}_{L^\infty(0,T;H^2(\Om))} \leq C.
        \end{split}
    \end{align}
    Finally, using \eqref{est:Lr:r} written for $r = 3$, and exploiting the continuous embedding $W^{1,3}(\Om)\emb C(\overline{\Om})$, we infer from \eqref{est:F':cr} that
    \begin{align}\label{est:F':inf:inf}
        \norm{F_1^\prime(\phi(t))}_{L^\infty(\Om)} \leq \norm{F_1^\prime(\phi)}_{L^\infty(0,T;W^{1,3}(\Om))} \leq C \eqqcolon C_\star
    \end{align}
    for almost every $t\in[0,T]$. Since Theorem~\ref{thm:highreg} implies that $\phi\in C(\overline{Q})$, the inequality \eqref{est:F':inf:inf} holds true even for every $t\in[0,T]$. Lastly, taking $\delta_\star \coloneqq 1 - (F_1^\prime)^{-1}(C_\star)$, leads to the desired conclusion.
\end{proof}

\section*{Appendix: \texorpdfstring{$L^p$}{Lp} regularity theory for bulk-surface elliptic systems}
\addcontentsline{toc}{section}{Appendix: \texorpdfstring{$L^p$}{Lp} regularity theory for bulk-surface elliptic systems}
\renewcommand\thesection{A}
\setcounter{theorem}{0}
\setcounter{equation}{0}

\begin{proposition}\label{Prop:Appendix}
    Suppose that $\Omega\subset\R^d$ with $d\in\{2,3\}$ is a bounded domain of class $C^2$ with boundary $\Gamma\coloneqq\partial\Omega$, and let $\alpha\in\R$ and $K\in [0,\infty]$ be arbitrary. Moreover, suppose that the pair $(f,g)\in \mathcal{L}^p$ with $p\ge 2$ fulfills the compatibility condition
    \begin{equation*}
        \begin{cases}
            \alpha\abs{\Omega}\meano{f} + \abs{\Gamma}\meang{g} = 0
            &\text{if $K\in[0,\infty)$},
            \\
            \meano{f} = 0 \quad\text{and}\quad \meang{g} = 0
            &\text{if $K=\infty$},
        \end{cases}
    \end{equation*} 
    and assume that the pair $(u,v) \in \mathcal{H}^1_{K,\alpha}$ is a weak solution of the bulk-surface elliptic problem%
    \begin{subequations}
    \label{BSE}
        \begin{align}
            \label{BSE:1}
            &-\Lap u = f &\text{in $\Omega$},\\
            \label{BSE:2}
            &-\Lapg v + \alpha \deln u = g &\text{on $\Gamma$},\\
            \label{BSE:3}
            &\begin{cases}
                u = \alpha v &\text{if $K=0$},\\
                K\deln u = \alpha v - u &\text{if $K\in (0,\infty)$},\\
                \deln u = 0 &\text{if $K = \infty$}
            \end{cases}
            &\text{on $\Gamma$},
        \end{align}        
    \end{subequations}
    which means that
    \begin{equation*}
        \bigscp{(u,v)}{(\zeta,\xi)}_{K,\alpha} = \bigscp{(f,g)}{(\zeta,\xi)}_{\mathcal{L}^2}
        \quad\text{for all $(\zeta,\xi)\in \mathcal{H}^1_{K,\alpha}$}.
    \end{equation*}
    Then, it holds
    \begin{subequations}
    \begin{align}
        \label{BSE:REG:B}
        u&\in W^{2,p}(\Omega),
        \\
        \label{BSE:REG:S}
        v&\in W^{2,p}(\Gamma),
    \end{align}
    \end{subequations}
    and there exists a constant $C>0$ independent of $u$ and $v$ such that
    \begin{equation}
    \label{REG:W2P}
        \norm{(u,v)}_{\mathcal{W}^{2,p}} \le  C \norm{(f,g)}_{\mathcal{L}^p}.
    \end{equation}
\end{proposition}

\begin{proof}
    We will consider the cases $K=0$, $K\in(0,\infty)$ and $K=\infty$ separately. In this proof, the letter $C$ denotes generic positive constants that are independent of $u$ and $v$.

    \textbf{The case $K=0$.} The regularity theory for bulk-surface elliptic system developed in \cite[Theorem 3.3]{Knopf2021a} already yields $(u,v) \in \mathcal{H}^2$ and the equations in \eqref{BSE} are fulfilled in the strong sense.
    Therefore, in the following, we assume $p>2$.
    Moreover, \cite[Theorem 3.3]{Knopf2021a} provides the estimate
    \begin{align}
    \label{REG:A1}
        \norm{(u,v)}_{\mathcal{H}^2} &\le C \norm{(f,g)}_{\mathcal{L}^2}.
    \end{align}   
    Recalling that the boundary is a $(d-1)$-dimensional submanifold of $\R^d$, Sobolev's embedding theorem implies
    \begin{equation*}
        u\vert_\Gamma = \alpha v \in W^{t,p}(\Gamma)
        \quad\text{with}\quad
        t = \frac 52 + \frac{d-1}{p} - \frac{d}{2}.
    \end{equation*}
    Since $f\in L^p(\Omega)$, we use elliptic regularity theory for Poisson's equation with an inhomogeneous Dirichlet boundary condition (see, e.g.,\cite[Theorem~3.2]{Brezzi1987} or \cite[Theorem~{A.2}]{Colli2019a}) to deduce
    \begin{equation*}
        u \in W^{s,p}(\Omega) 
        \quad\text{with}\quad
        s = \min\left\{2, \frac 52 + \frac{d}{p} - \frac{d}{2}\right\}
        \ge 1 + \frac 2p.
    \end{equation*}
    This directly entails
    \begin{equation*}
        \Grad u \in \mathbf{W}^{s-1,p}(\Omega) \emb \mathbf{W}^{2/p,p}(\Omega). 
    \end{equation*}
    Since $\frac2p - \frac1p = \frac 1p$ is positive and not an integer, the trace theorem
    (see, e.g., \cite[Chapter~2, Theorem~2.24]{Brezzi1987}) implies
    \begin{equation*}
        \deln u \in W^{1/p,p}(\Gamma) \emb L^p(\Gamma). 
    \end{equation*}
    Since $g\in L^p(\Gamma)$, applying regularity theory for the Laplace--Beltrami equation (see, e.g., \cite[Lemma~B.1]{Vitillaro2017}), we thus conclude \eqref{BSE:REG:S} along with the estimate
    \begin{align}
    \label{REG:A2}
        \norm{v}_{W^{2,p}(\Gamma)} \le C \norm{g}_{L^p(\Gamma)}.
    \end{align}
    This, in turn, yields
    \begin{equation*}
        u\vert_\Gamma = \alpha v \in W^{2,p}(\Gamma).
    \end{equation*}
    Hence, by means of elliptic regularity theory for Poisson's equation with an inhomogeneous Dirichlet boundary condition (see, e.g.,\cite[Theorem~3.2]{Brezzi1987} or \cite[Theorem~{A.2}]{Colli2019a}), we conclude that \eqref{BSE:REG:B} holds with    
    \begin{align}
    \label{REG:A3}
        \norm{u}_{W^{2,p}(\Omega)} 
        \le C \big( \norm{f}_{L^p(\Omega)} + \norm{v}_{W^{2,p}(\Gamma)} \big)
        \le C \norm{(f,g)}_{\mathcal{L}^p}.
    \end{align}
    Combining \eqref{REG:A2} and \eqref{REG:A3}, we obtain \eqref{REG:W2P}.
    
    \textbf{The case $K\in(0,\infty)$.}
    Also in this case, the regularity theory for bulk-surface elliptic system developed in \cite[Theorem 3.3]{Knopf2021a} already provides $(u,v) \in \mathcal{H}^2$ and the equations in \eqref{BSE} are fulfilled in the strong sense.
    Therefore, in the following, we assume $p>2$.
    Using the trace theorem as well as Sobolev's embedding theorem, we have
    \begin{equation*}
        u \in H^{2}(\Omega) \emb H^{\frac 32}(\Gamma) 
        \emb W^{t,p}(\Gamma)
        \quad\text{with}\quad
        t = 2 + \frac{d-1}{p} - \frac{d}{2}.
    \end{equation*}
    In particular, we also have $v\in H^2(\Gamma) \emb W^{t,p}(\Gamma)$.
    This directly implies
    \begin{equation*}
        \deln u = \frac 1K \big(\alpha v - u\big)
        \in W^{t,p}(\Gamma).
    \end{equation*}
    Since $g\in L^p(\Gamma)$, applying regularity theory for the Laplace--Beltrami equation (see, e.g., \cite[Lemma~B.1]{Vitillaro2017}), we thus conclude \eqref{BSE:REG:S} with
    \begin{align}
    \label{REG:A4}
        \begin{split}
        &\norm{v}_{W^{2,p}(\Gamma)} 
        \le C \big( \norm{g}_{L^p(\Gamma)} + \norm{\deln u}_{W^{t,p}(\Gamma)} \big)
        \\
        &\quad \le C \big( \norm{g}_{L^p(\Gamma)} + \norm{(u,v)}_{\mathcal{H}^2} \big)
        \le C \norm{(f,g)}_{\mathcal{L}^p}.
        \end{split}
    \end{align}
    Moreover, as $f\in L^p(\Omega)$, we use elliptic regularity theory for Poisson's equation with an inhomogeneous Neumann boundary condition (see, e.g.,\cite[Theorem~3.2]{Brezzi1987} or \cite[Theorem~{A.2}]{Colli2019a}) to deduce
    \begin{equation*}
        u \in W^{s,p}(\Omega) 
        \quad\text{with}\quad
        s = \min\left\{2, 3 + \frac{d}{p} - \frac{d}{2}\right\}
        \ge 1 + \frac 2p
    \end{equation*}
    along with the estimate
    \begin{align}
    \label{REG:A5}
    \begin{split}
        &\norm{u}_{W^{s,p}(\Omega)} 
        \le C \big( \norm{f}_{L^p(\Omega)} + \norm{\deln u}_{W^{t,p}(\Gamma)} \big)
        \\
        &\quad \le C \big( \norm{f}_{L^p(\Omega)} + \norm{(u,v)}_{\mathcal{H}^2} \big)
        \le C \norm{(f,g)}_{\mathcal{L}^p}.
    \end{split}
    \end{align}
    Since $1 + \frac2p - \frac1p = 1 + \frac 1p$ is positive and not an integer, the trace theorem implies
    \begin{equation*}
        u\vert_\Gamma \in W^{1+1/p,p}(\Gamma). 
    \end{equation*}
    This entails
    \begin{equation*}
        \deln u = \frac 1K \big(\alpha v - u\big)
        \in W^{1+1/p,p}(\Gamma).
    \end{equation*}
    Hence, using again elliptic regularity theory for Poisson's equation with an inhomogeneous Neumann boundary condition (see, e.g.,\cite[Theorem~3.2]{Brezzi1987} or \cite[Theorem~{A.2}]{Colli2019a}), we conclude that \eqref{BSE:REG:B} holds with 
    \begin{align}
    \label{REG:A6}
    \begin{split}
        &\norm{u}_{W^{2,p}(\Omega)} 
        \le C \big( \norm{f}_{L^p(\Omega)} + \norm{\deln u}_{W^{1+1/p,p}(\Gamma)} \big)
        \\
        &\quad \le C \big( \norm{f}_{L^p(\Omega)} + \norm{u}_{W^{s,p}(\Omega)} \big)
        \le C \norm{(f,g)}_{\mathcal{L}^p}.
    \end{split}
    \end{align}
    Combining \eqref{REG:A4} and \eqref{REG:A6}, we obtain \eqref{REG:W2P}.
    
    \textbf{The case $K=\infty$.}
    In this case, the Poisson--Neumann problem $\big(\eqref{BSE:1},\eqref{BSE:3}\big)$ and the Laplace--Beltrami equation \eqref{BSE:2} are completely decoupled. Applying elliptic regularity for Poisson's equation with homogeneous Neumann boundary condition (see, e.g.,\cite[Theorem~3.2]{Brezzi1987} or \cite[Theorem~{A.2}]{Colli2019a}) and regularity theory for the Laplace--Beltrami equation (see, e.g., \cite[Lemma~B.1]{Vitillaro2017}), respectively, we directly conclude \eqref{BSE:REG:B} and \eqref{BSE:REG:S} along with estimate \eqref{REG:W2P}.

    As the regularities \eqref{BSE:REG:B} and \eqref{BSE:REG:S} as well as the estimate \eqref{REG:W2P} have been established for every choice $K\in[0,\infty]$, the proof is complete.
\end{proof}



\section*{Acknowledgment}
The authors would like to thank the anonymous referee for the thorough review of our manuscript and the very helpful comments, which definitely helped to improve this paper.
This research was funded by the Deutsche Forschungsgemeinschaft (DFG, German Research Foundation) – Project 52469428. Moreover, the authors were partially supported by the Deutsche Forschungsgemeinschaft (DFG, German Research Foundation) – RTG 2339. The support is gratefully acknowledged.




\footnotesize

\bibliographystyle{abbrv}
\bibliography{KS_CCH}

\end{document}